\documentclass[11pt]{amsart}

\usepackage[dvipsnames]{xcolor}
\usepackage{enumitem}
\usepackage{amssymb,mathrsfs,graphicx,mathtools}
\usepackage{tikz}
\usetikzlibrary{calc,decorations.pathreplacing}
\usepackage[margin=1in]{geometry}
\usepackage{etoolbox}

\usepackage[square]{natbib}

\usepackage[citecolor=PineGreen,colorlinks=true,linkcolor=RoyalBlue,urlcolor=BrickRed]{hyperref}

\theoremstyle{plain}
\newtheorem{theorem}{Theorem}[section]
\newtheorem{proposition}[theorem]{Proposition}
\newtheorem{corollary}[theorem]{Corollary}
\newtheorem{lemma}[theorem]{Lemma}

\theoremstyle{definition}
\newtheorem{definition}[theorem]{Definition}
\newtheorem{example}[theorem]{Example}
\newtheorem{remark}[theorem]{Remark}

\newcommand{\Z}{\mathbb{Z}}
\newcommand{\R}{\mathbb{R}}
\newcommand{\N}{\mathbb{N}}
\newcommand{\E}{\mathbb{E}}
\renewcommand{\P}{\mathbb{P}}
\newcommand{\eps}{\varepsilon}
\newcommand{\one}{\mathbf{1}}
\DeclareMathOperator{\var}{Var}

\makeatletter
\renewcommand{\l@section}{\@tocline{1}{0pt}{1.5em}{}{}}
\renewcommand{\l@subsection}{\@tocline{2}{0pt}{3.0em}{}{}}
\makeatother
\makeatletter
\patchcmd{\@settitle}{\uppercasenonmath\@title}{\Large}{}{}
\patchcmd{\@setauthors}{\MakeUppercase}{\large}{}{}
\makeatother

\title{Divisible sandpiles via random walks in random scenery }

\author{Ahmed Bou-Rabee} \address{Department of Mathematics, University of Pennsylvania, USA} \email{ahmedmb@sas.upenn.edu}
\author{Yuval Peres} \address{Beijing Institute of Mathematical Sciences and Applications, China} \email{yperes@gmail.com}
\author{Ecaterina Sava-Huss}  \address{Department of Mathematics, University of Innsbruck, Austria} \email{Ecaterina.Sava-Huss@uibk.ac.at}

\date{April 15, 2026}

\begin{document}
	
	\begin{abstract}
		We analyze an optimal stopping problem for random walk in random scenery on general graphs, and determine when it has a finite optimum. We use this to extend a theorem of \citet*{div-sand-crit}. They  proved that on a vertex-transitive graph, the divisible sandpile with i.i.d.\ initial masses of mean~$\mu$ stabilizes almost surely if $\mu < 1$, explodes if $\mu > 1$, and explodes if $\mu = 1$ with positive finite variance. Their proofs rely on conservation of mean mass under toppling. This conservation extends to unimodular random graphs, but fails on  general graphs. We prove explosion for all infinite bounded-degree graphs whenever~$\mu \geq 1$, and stabilization for~$\mu<1$ provided the initial masses have finite $p$-th moment for some $p>3$. Our conditions are nearly sharp: we exhibit unbounded-degree graphs on which sandpiles with $\mu > 1$ stabilize, and for every $p < 3$ we construct bounded-degree graphs on which sandpiles with~$\mu < 1$ and finite $p$-th moment explode.
	\end{abstract}
	\maketitle
	
	{\small \tableofcontents}

	\section{Introduction}
	
	\begin{figure}[t]
		\centering
		\includegraphics[width=0.5\textwidth]{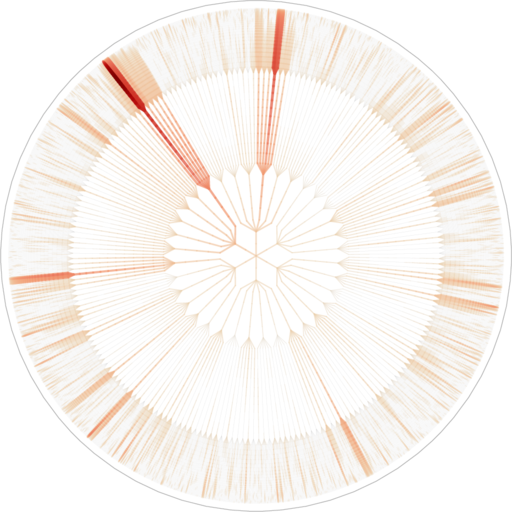}
		\caption{Approximation of the odometer on a tree-of-pipes graph (Section~\ref{sec:transient-nonstab}), embedded in the Poincar\'e disk with the root at the center and deeper pipes near the boundary. Edge color and width are proportional to the odometer. The i.i.d.\ initial masses have mean~$\mu = 0.05 \ll 1$, yet a single heavy-tailed fluctuation near the boundary cascades inward through the tree hierarchy and drives the odometer at the root far above~$1$. Theorem~\ref{thm:transient-nonstab} shows that this explosion occurs almost surely when only a finite $p$-th moment with $p<3$ is assumed.}
		\label{fig:subcrit-explosion}
	\end{figure}
	
	In the abelian sandpile model of \citet*{BTW}, each vertex of a graph carries an integer number of grains, and a vertex with at least as many grains as its degree topples by sending one grain to each neighbor. The \textbf{divisible sandpile}, introduced by \citet{LP09, LP10}, is the continuous-mass version, where vertices that have mass greater than 1 distribute the excess equally among their neighbors (see the precise definition below). \citet*{div-sand-crit} established a phase transition for  the divisible sandpile with i.i.d.\ initial masses on transitive graphs depending on the mean of the initial masses, and our main goal is to extend this to general graphs. 
	In this general setting, we deduce the sandpile phase transition from an optimal stopping result for random walk in random scenery, which is of independent interest. Recall that a graph~$G = (V,E)$ is called \textbf{doubly transient} if its Green function $g(o,v) \coloneqq \E_o[\sum_{n \geq 0} \one_{\{X_n = v\}}] / \deg(v)$ satisfies $\sum_{v \in V}  g(o,v)^2 < \infty$ for all~$o \in V$, where~$(X_n)_{n\in\N}$ is the simple random walk on $G$; for example, $\Z^d$ is doubly transient if and only if $d \geq 5$.
	\begin{theorem}[Explosion and stabilization for random walk]\label{thm:OS}
		Let~$G=(V,E)$ be an infinite, locally finite, connected graph with degree bounded by~$d$, let~$(X_n)_{n\geq 0}$ be simple random walk on~$G$, let $(\xi(v))_{v\in V}$ be i.i.d.\ and independent of the walk, and $S_n\coloneqq\sum_{k=0}^{n-1}\xi(X_k)/\deg(X_k)$.
		\begin{enumerate}[label=\textup{(\roman*)}]
			\item \textup{(Explosion)}\; If $\E[\xi]\in(0,\infty]$, then $\sup_n\E_x[S_n\mid\xi]=\infty$ almost surely for every $x\in V$.
			\item \textup{(Explosion)}\; Suppose that $\E[\xi]=0$ and either $\var(\xi)\in(0,\infty)$, or $\xi\not\equiv 0$ and symmetric.   Then $\sup_{\tau} \E_x[S_\tau \mid\xi]=\infty$ almost surely for every $x\in V$, where the supremum is over bounded stopping times; if $G$ is not doubly transient, then $\sup_n\E_x[S_n\mid\xi]=\infty$ almost surely for every $x\in V$.
			\item \textup{(Stabilization)}\; If $\E[\xi]\in[-\infty,0)$ and $\E[(\xi^+)^p]<\infty$ for some $p>3$, then $\E_x[ (\sup_n S_n)^q]<\infty$ for every $q\in[1,(p-1)/2)$ and~$x \in V$\,.
		\end{enumerate}
	\end{theorem}
	The division by the degree in the definition of~$S_n$ is useful in our application but it can be removed.  
	We now recall the definition of the divisible sandpile. Given a connected, locally finite graph $G=(V,E)$, a configuration $\sigma\colon V \to \R$ assigns a real-valued mass to each vertex. A vertex~$v$ with $\sigma(v)>1$ is \textbf{unstable}; it topples by keeping mass~$1$ and distributing the excess $\sigma(v) - 1$ equally among its $\deg(v)$ neighbors. In the \textbf{parallel toppling procedure} all unstable vertices topple simultaneously: writing $a^+\coloneqq\max(a,0)$ and setting $\sigma_0 \coloneqq \sigma$ and $u_0 \coloneqq 0$, for $n \geq 0$ define
	\begin{align}
		u_{n+1}(v) &= u_n(v) + \frac{(\sigma_n(v) - 1)^+}{\deg(v)}, \label{eq:u-update} \\
		\sigma_{n+1}(v) &= \min\bigl(\sigma_n(v),\,1\bigr) + \sum_{w \sim v} \frac{(\sigma_n(w)-1)^+}{\deg(w)}\,. \label{eq:sigma-update}
	\end{align}
	The quantity $\deg(v) u_n(v)$ is the total mass emitted by~$v$ up to time~$n$; we call $u_n$ the \textbf{odometer}. Since $u_n(v)$ is nondecreasing in~$n$, the limit $u_\infty(v) \coloneqq \lim_{n \to \infty} u_n(v) \in [0, \infty]$ exists for each~$v$.
	If $u_\infty(v) < \infty$ for every~$v \in V$, then the configuration \textbf{stabilizes}; otherwise it \textbf{explodes}. A configuration with finite total mass on an infinite connected graph always stabilizes. When the total mass is infinite---for instance when the initial masses are i.i.d.---stabilization depends on the interplay between the law of the initial mass and the geometry of the underlying graph. 
	
	On vertex-transitive graphs such as~$\Z^d$, \citet*{div-sand-crit} showed that for i.i.d.\ initial masses with common mean~$\mu$, if $\mu<1$, then the sandpile stabilizes almost surely, and if $\mu>1$, then it explodes almost surely. At $\mu=1$, explosion holds when the variance is positive and finite. Their proofs exploit a key symmetry: on vertex-transitive graphs, the expected mass is conserved under toppling.
	
	On a general graph, toppling redistributes mass unevenly across vertices of different degree, and conservation fails (Example~\ref{ex:comb-nonconservation}). This leads to a surprising phenomenon: on a bounded-degree graph with $\mu < 1$, the sandpile can explode if the initial masses have a sufficiently heavy tail. The geometry of the graph can amplify the effect of a single heavy-tailed outlier, even when the average mass is subcritical (Figure~\ref{fig:subcrit-explosion}).
	
	In this paper we extend the supercritical and finite-variance critical explosion results to all infinite bounded-degree graphs, with no assumptions on the graph beyond bounded degree. When the centered law~$\sigma-1$ is symmetric, a convexity argument removes the finite-variance hypothesis entirely (Proposition~\ref{prop:convexity-reduction}), a result which is new even on~$\Z^d$. In the subcritical regime, we prove stabilization under a finite $p$-th moment for some $p > 3$, and show that this condition is nearly sharp.
	
	\begin{theorem}[Explosion]\label{thm:explosion}
		Let~$G = (V,E)$ be an infinite, locally finite, connected graph with
		degree bounded by~$d$, and let $(\sigma(v))_{v\in V}$ be i.i.d.\ random variables
		with $\E[\sigma(v)] = \mu$.
		\begin{enumerate}[label=\textup{(\roman*)}]
			\item If $\mu \in (1,\infty]$, then
			$\P(\sigma \text{ stabilizes}) = 0$.
			\item If $\mu = 1$ and either $\var(\sigma(v))\in(0,\infty)$, or $\sigma(v)\not\equiv 1$ and $\sigma(v)-1$ is symmetric, then
			$\P(\sigma \text{ stabilizes}) = 0$.
		\end{enumerate}
	\end{theorem}
	
	Part~(i) follows from Proposition~\ref{prop:supercritical} and Corollary~\ref{cor:RW-infinite}. The finite-variance case of part~(ii) follows from Propositions~\ref{prop:critical} and~\ref{prop:doubly-transient-really-general}, together with Corollary~\ref{cor:RW-infinite}; the symmetric case follows from Proposition~\ref{prop:convexity-reduction} and Corollary~\ref{cor:RW-infinite}.
	
	\begin{theorem}[Stabilization]\label{thm:stab}
		Let~$G = (V,E)$ be an infinite, locally finite, connected graph with
		degree bounded by~$d$, and let $(\sigma(v))_{v\in V}$ be i.i.d.\ random variables
		with $\E[\sigma(v)] = \mu \in [-\infty,1)$.
		\begin{enumerate}[label=\textup{(\roman*)}]
			\item If $\E[(\sigma(v)^+)^p] < \infty$ for some $p > 3$, then
			$\sup_{v \in V}\E[u_\infty(v)^q] < \infty$ for every $q\in[1,(p-1)/2)$; in particular $\P(\sigma \text{ stabilizes}) = 1$.
			\item If $|B(o,r)|\leq C r^{d_f}$ for some $o\in V$, some $C>0$, some $d_f\geq 1$, and all $r\geq 1$, and if $\E[(\sigma(v)^+)^p] < \infty$ for some $p > d_f$, then $\P(\sigma \text{ stabilizes}) = 1$.
		\end{enumerate}
	\end{theorem}

	Part~(i) follows from Proposition~\ref{prop:subcritical} and part~(ii) from Proposition~\ref{prop:poly-growth}, both combined with the universal heat kernel bound $\sup_x\P_x(X_n=x)\leq C n^{-1/2}$ on bounded-degree graphs~\citep[Example~5.14]{Grigoryan} and the Carne--Varopoulos sub-Gaussian displacement bound~\citep[Section~13.2]{LP16}, which supplies hypothesis~\ref{H3} with~$d_w=2$. On graphs with better spectral properties both thresholds can be sharpened; see those propositions for precise statements involving the spectral and walk dimensions.
	
	Theorems~\ref{thm:explosion} and~\ref{thm:stab} follow from Theorem~\ref{thm:OS} via the random walk representation of the odometer established in Section~\ref{sec:rw-rep} (Theorem~\ref{thm:RW} and Corollary~\ref{cor:RW-infinite}). This representation---expressing the sandpile odometer as an optimal stopping value---is classical in the free boundary literature \citep{PS06} but appears to be new in the sandpile setting; it replaces the conservation-of-mass machinery of~\citet{div-sand-crit}. The walk payoff decomposes into a scenery-independent drift of rate~$\mu-1$ and a random walk in random scenery weighted by local times; the three regimes of Theorem~\ref{thm:OS} correspond to the sign of this drift.

	Our conditions are nearly sharp in the following senses.
	\begin{enumerate}[label=(\alph*)]
		\item \textit{The threshold $p > 3$ is sharp on general bounded-degree graphs.} For every $p \in (0,3)$, we construct a transient bounded-degree graph (Theorem~\ref{thm:transient-nonstab}) and, for every growth exponent $d_f \in (\max(p,2),3)$, a recurrent bounded-degree graph of growth~$d_f$ (Theorem~\ref{thm:recurrent-nonstab}) on which, for every $\mu < 1$, there exist i.i.d.\ masses bounded from below with mean~$\mu$ and finite $p$-th moment that explode almost surely.
		\item \textit{On graphs with heat kernel bounds, the refined threshold of Proposition~\ref{prop:subcritical} is optimal.} Lemma~\ref{lem:moment-sharpness} shows that violating the moment condition forces the \emph{expected} odometer to diverge. On~$\Z^d$ this gives the sharp exponent $p > 1+2/d$: violating it forces $\E[u_\infty(o)]=\infty$ even though stabilization holds for every $\mu<1$ by conservation of mass on vertex-transitive graphs (Section~\ref{sec:stationary}).
		\item \textit{Bounded degree cannot be removed from either part of Theorem~\ref{thm:explosion}.} We construct a locally finite tree on which every i.i.d.\ law with $\E[|\sigma|]<\infty$ stabilizes almost surely (Lemma~\ref{ex:counterexample}); in particular, this includes $\mu>1$ and the critical case $\mu=1$ with positive finite variance. The tree has $\sum_v g(x,v)<\infty$, so stabilization follows from Proposition~\ref{prop:short-clock}.
	\end{enumerate}
	
	The assumption of identical distribution can be weakened in some of the arguments to a uniform integrability assumption, but for clarity of exposition we do not pursue this.

	\subsection{Proof ideas}
	On stationary random graphs (Section~\ref{sec:stationary}), the phase transition follows from conservation of the degree-weighted mass $\E[\sigma_k(\rho)/\deg(\rho)]$ under toppling, with a nearly identical proof as in the vertex-transitive case treated by~\citet{div-sand-crit}.
	
	On general bounded-degree graphs, conservation fails. Our replacement is the random walk representation. Let~$\{X_k\}$ denote simple random walk on the graph and $S_n\coloneqq\sum_{k=0}^{n-1}\xi(X_k)/\deg(X_k)$ for~$\xi(v)\coloneqq\sigma(v)-1$; the random walk representation (Corollary~\ref{cor:RW-infinite}) is that
	\[
	u_\infty(x)=\sup_{\tau\ \mathrm{bounded}}\E_x[S_\tau\mid\sigma]\, , 
	\]
	where the supremum is taken over all bounded stopping times~$\tau$. The walk payoff~$S_n$ decomposes as
	\begin{equation}
		S_n = D_n + W_n,\qquad
		D_n \coloneqq (\mu - 1)\sum_{k=0}^{n-1} \frac{1}{\deg(X_k)},\qquad
		W_n \coloneqq \sum_{v\in V}\frac{L_n(v)}{\deg(v)}\bigl(\sigma(v)-\mu\bigr)\,.
		\label{eq:dn-wn-stuff}
	\end{equation}
	The drift $D_n$ is of order $(\mu-1)n$ on a bounded-degree graph. The fluctuation $W_n$, where $L_n(v)=\#\{k<n:X_k=v\}$ is the local time, weights the i.i.d.\ centered masses by the local times. The three regimes correspond to the sign of the drift, in increasing order of difficulty:
	\begin{itemize}
		\item \textit{Supercritical} ($\mu>1$; Section~\ref{sec:supercritical}). The positive drift overwhelms the fluctuation: $S_n$ grows linearly, so the walker gains by running forever and $u_\infty=\infty$ almost surely.
		\item \textit{Critical} ($\mu=1$; Section~\ref{sec:critical}). The drift vanishes and explosion must come from the fluctuation alone. The argument splits into two sub-cases. If $\sum_v g(o,v)^2 = \infty$ (examples: $\Z$ through~$\Z^4$), then the variance of~$W_n$ grows without bound and a second-moment argument produces explosion. If $\sum_v g(o,v)^2<\infty$ (examples: $\Z^d$ for $d\geq 5$), then the variance saturates and a more delicate construction is needed: we stop the walk in rare regions where the initial masses are unusually small, producing a large positive expected payoff; an Efron--Stein variance bound then upgrades to quenched explosion.  When~$\sigma-1$ is symmetric, a convexity argument removes the finite-variance hypothesis entirely: the nested-volume formula makes the odometer convex in the initial masses, and a truncation that preserves the law reduces the problem to the finite-variance case (Proposition~\ref{prop:convexity-reduction}).
		\item \textit{Subcritical} ($\mu<1$; Section~\ref{sec:subcritical}). The negative drift penalizes long runs, so the optimal strategy is to stop early. We give two proofs with complementary results.
		The first (Proposition~\ref{prop:subcritical}) proves finite moments $\E[u_\infty^q]<\infty$ under the condition $p>3$. The difficulty is that $W_n$ is not a sum of independent increments, since the walk revisits vertices. Conditionally on the walk trajectory, however, the masses $\sigma(v)-\mu$ at distinct vertices are independent, so concentration inequalities apply to~$W_n$ as a conditional sum. We decompose the supremum dyadically and on each scale apply the Fuk--Nagaev inequality. The moment condition $p>3$ enters through local time estimates that control the conditional variance.
		The second (Proposition~\ref{prop:poly-growth}) gives almost sure stabilization under the weaker condition $p>d_f$ on graphs with polynomial volume growth $|B(o,r)|\leq Cr^{d_f}$, but without moment bounds on~$u_\infty$. Here sub-Gaussian bounds on the walk displacement allow us to truncate the scenery outside a ball whose radius grows with the walk. On graphs satisfying the Einstein relation, this gives the sharp threshold $p>1$.
	\end{itemize}
	Section~\ref{sec:counterexamples} establishes the sharpness of our conditions. We show that bounded degree cannot be removed from the supercritical result by constructing a locally finite tree on which $\mu>1$ sandpiles stabilize. We show that the moment condition in the subcritical regime is necessary by exhibiting bounded-degree graphs---both transient and recurrent---on which for every $p<3$, sandpiles with $\mu<1$ and finite $p$-th moment explode almost surely. We also show that the refined moment threshold of Proposition~\ref{prop:subcritical} is optimal on graphs with two-sided heat kernel bounds. There we also show that the i.i.d.\ assumption is essential: $2$-dependent critical masses on~$\Z^d$ can stabilize even though i.i.d.\ critical masses explode.
	
	\subsection{Related work}
	For general surveys of sandpile models see \citet{Jarai18} and \citet{LP17}; for the divisible sandpile specifically see~\citet{survey-ruszel-div-sandpile}, who asked whether the stabilization results of~\citet{div-sand-crit} extend beyond vertex-transitive graphs. The present paper answers this in the affirmative for a large class of graphs including all those with bounded degree.
	
	On vertex-transitive graphs, the divisible sandpile is by now well understood: \citet*{MR3877548} identified the scaling limit of the odometer on the torus as a bi-Laplacian Gaussian field, with extensions to heavy-tailed masses~\citep{MR3834853}, dependent weights~\citep{CdGR20}, long-range toppling~\citep{MR3796366, MR4282689}, and macroscopic limits~\citep{BouRabee-abelian-sandpile}. Progress on non-transitive graphs has so far been confined to specific families: the Sierpi\'nski gasket~\citep{MR3957187, KaiserSH}, the comb lattice~\citep{HussSava12}, and mated-CRT maps~\citep{BouRabeeGwynne-IDLA, BouRabeeGwynne-harmonic}. An analogous stabilization--explosion phase transition occurs for activated random walks~\citep{RollaSidoravicius12, SidoraviciusTeixeira17, johnson2025odometer} on~$\Z^d$.
	
	The parallel toppling recursion (Lemma~\ref{lem:recursion}), without the positive-part truncation, is the Jacobi relaxation for the Poisson equation $\Delta u = 1 - \sigma$. The truncation makes it a projected Jacobi iteration for the obstacle problem $u \geq 0$ and $\sigma + \Delta u \leq 1$, connecting the divisible sandpile to diffusive load balancing~\citep{Cybenko89, RSW}. Our random walk representation (Theorem~\ref{thm:RW}) replaces this iterative viewpoint with a probabilistic one, expressing the odometer as the value function of an optimal stopping problem. Conditionally on the walk, the fluctuation $W_n = \sum_v (L_n(v)/\deg(v))(\sigma(v)-\mu)$ is a sum of independent random variables weighted by local times, an instance of a random walk in random scenery~\citep{KestenSpitzer79}.
	
	\subsection{Notation and setup}
	Throughout, $G=(V,E)$ denotes an infinite, locally finite, connected graph. We write $\deg(v)$ for the degree of~$v$, $\mathrm{dist}(\cdot,\cdot)$ for the graph distance, and $B(x,R)\coloneqq\{v:\mathrm{dist}(v,x)\leq R\}$ for the closed ball. For $a\in\R$, write $a^+\coloneqq\max(a,0)$ and $a^-\coloneqq\max(-a,0)$.
	
	Let $(X_n)_{n\geq 0}$ be the simple random walk on~$G$, with $\P_x(X_1=y)=1/\deg(x)$ for $y\sim x$. We write $\P_x$ and $\E_x$ for probability and expectation of the walk started at~$x$. The hitting time of vertex~$v \in V$ is $T_v\coloneqq\inf\{n\geq 0:X_n=v\}$ and the first return time is $T_v^+\coloneqq\inf\{n\geq 1:X_n=v\}$.
	
	The \textbf{Green function} is the degree-normalized expected occupation kernel
	\begin{equation}\label{eq:green-def}
		g(x,v)\coloneqq\frac{1}{\deg(v)}\,\E_x\biggl[\sum_{n\geq 0}\one_{\{X_n=v\}}\biggr]\in(0,\infty]\,.
	\end{equation}
	By reversibility, $g(x,v)=g(v,x)$. 
	For background on Green functions and their connection to electrical networks, see \citet[Chapter~2]{LP16}.
	
	For $f\colon V\to \R$, the discrete Laplacian is $\Delta f(x) \coloneqq \sum_{y \sim x}\bigl(f(y) - f(x)\bigr)$.
	We fix a vertex $o\in V$ throughout; since the graph is connected, stabilization does not depend on the choice of~$o$. For a finite connected set $C\subseteq V$ with $o\in C$ and $V\setminus C\neq\varnothing$, the \textbf{killed Green function} $g_C(o,v)$ is the unique function satisfying
	\begin{equation}\label{eq:gC-PDE}
		g_C=0\ \text{on }V\setminus C,\qquad \Delta g_C=0\ \text{on }C\setminus\{o\},\qquad -\Delta g_C(o)=1\,;
	\end{equation}
	equivalently, $g_C(o,v)=\E_o[L_{\tau_C}(v)]/\deg(v)$ where $\tau_C\coloneqq\inf\{k\geq 0:X_k\notin C\}$. For a general source $y\in C$, define $g_C(y,v)\coloneqq\E_y[L_{\tau_C}(v)]/\deg(v)$.
	
	The \textbf{finite-time Green function} and the associated \textbf{inverse-degree clock} and \textbf{fluctuation scale} are
	\begin{equation}\label{eq:gn-def}
		g_n(x,v)\coloneqq\frac{\E_x[L_n(v)]}{\deg(v)},\qquad
		A_n(x)\coloneqq\sum_{v\in V}g_n(x,v)=\E_x\biggl[\sum_{k=0}^{n-1}\frac{1}{\deg(X_k)}\biggr],\qquad
		\Sigma_n(x)\coloneqq\sum_{v\in V}g_n(x,v)^2\,.
	\end{equation}
	On a bounded-degree graph, $g_n(x,v)\asymp\E_x[L_n(v)]$ and $A_n(x)\asymp n$ and $\Sigma_n(x)\asymp\sum_{v\in V}\E_x[L_n(v)]^2$.
	
	Given an initial configuration $\sigma\colon V\to\R$, the walk is taken independent of~$\sigma$. We write $\P$ and $\E$ for probability and expectation under the joint law of~$\sigma$ and the walk. When the dependence on~$\sigma$ needs to be made explicit, we write $u_\infty(v;\sigma)$ for the odometer. When we write $\E[\sigma(v)]=\mu\in[-\infty,\infty]$, we mean the extended expectation: $\mu=+\infty$ means $\E[\sigma(v)^+]=\infty$ and $\E[\sigma(v)^-]<\infty$, and $\mu=-\infty$ means $\E[\sigma(v)^-]=\infty$ and $\E[\sigma(v)^+]<\infty$. In particular, the indeterminate case $\E[\sigma(v)^+]=\E[\sigma(v)^-]=\infty$ is excluded throughout. When the masses are i.i.d., we sometimes write~$\sigma$ for a generic copy of~$\sigma(v)$; thus $\var(\sigma)$ and $\E[\sigma^2]$ refer to the common marginal.

	\section{Stationary random graphs}\label{sec:stationary}
	On the stationary random rooted graphs of \citet{BenjaminiCurien12}, the conservation-of-mass techniques of \citet{div-sand-crit} extend readily from the vertex-transitive setting and give a clean phase transition (Theorem~\ref{thm:stationary-phase}). We include this self-contained section because stationary random rootings cover many non-vertex-transitive examples, including the Sierpi\'nski gasket~\citep{MR3957187, KaiserSH} and supercritical percolation clusters~\citep{survey-ruszel-div-sandpile}. The more involved random-walk techniques needed for general bounded-degree graphs begin in Section~\ref{sec:rw-rep}.
	
	\begin{figure}[ht]
		\centering
		\includegraphics[width=0.46\textwidth]{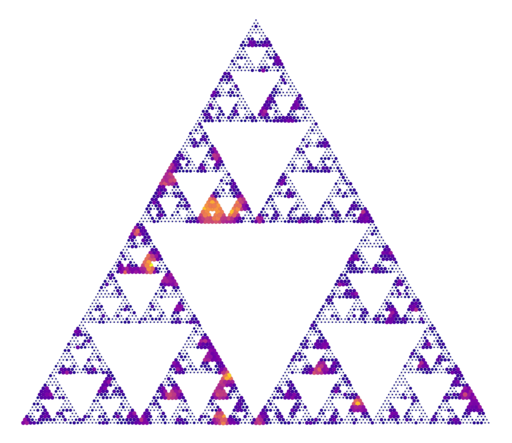}%
		\hfill
		\includegraphics[width=0.48\textwidth]{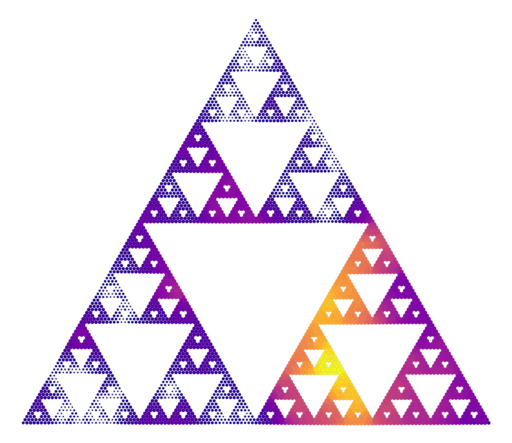}
		\caption{Odometer~$u_\infty(v)$ on the Sierpi\'nski gasket (level~$6$) with i.i.d.\ exponential initial masses of mean~$\mu$; vertex color and size are proportional to~$u_\infty(v)$. Left: subcritical ($\mu=0.6$); right: critical ($\mu=1.0$). The infinite Sierpi\'nski gasket admits a stationary random rooting, so Theorem~\ref{thm:stationary-phase} gives stabilization for~$\mu < 1$ and explosion for~$\mu > 1$.}
		\label{fig:sierpinski}
	\end{figure}
	
	A \textbf{rooted graph} is a pair $(G,\rho)$ where $G=(V,E)$ is a locally finite connected graph and $\rho \in V$ is a distinguished vertex. Two rooted graphs are \textbf{isomorphic} if there is a graph isomorphism mapping the root to the root.
	
	\begin{definition}\label{def:stationary-graph}
		If $(G,\rho)$ has the same law as $(G,X_1)$, where, conditionally on $(G,\rho)$, the vertex $X_1$ is chosen among the neighbors of~$\rho$ with probability $1/\deg(\rho)$, then $(G,\rho)$ is called \textbf{stationary}. Equivalently, for every non-negative Borel function $h$ on the space of isomorphism classes of rooted graphs (equipped with the local topology),
		\[
		\E[h(G,\rho)] = \E\biggl[\frac{1}{\deg(\rho)}\sum_{y \sim \rho} h(G,y)\biggr]\,.
		\]
	\end{definition}
	
	By~\citet[Proposition~2.1]{BenjaminiCurien12}, stationarity of $(G,\rho)$ is equivalent to shift-invariance of the path-space law of the simple random walk on~$G$ started at~$\rho$. If $\one_A(G,o)=\one_A(G,o')$ for every rooted graph~$(G,o)$ and every neighbor $o'\sim o$, then a Borel set~$A$ of rooted graphs is called \textbf{rerooting-invariant}. We write $\mathcal{I}_G$ for the $\sigma$-algebra of all rerooting-invariant events in~$\sigma(G,\rho)$.
	
	\begin{definition}\label{def:ergodic-graph}
		If every rerooting-invariant event has probability $0$ or~$1$ (that is, $\mathcal{I}_G$ is trivial), then a stationary random rooted graph $(G,\rho)$ is called \textbf{ergodic}.
	\end{definition}
	
	By the ergodic decomposition theorem, every stationary law decomposes into ergodic components: conditionally on~$\mathcal{I}_G$, the law of $(G,\rho)$ is stationary and ergodic. Every vertex-transitive graph is stationary, since rooting at each vertex produces isomorphic rooted graphs. Stationary random graphs provide a random-graph analogue of this invariance; they are discussed in~\citet{BenjaminiCurien12} and include reversible random rooted graphs, which are related to the unimodular random networks of~\citet{AldousLyons07} by degree biasing. The infinite Sierpi\'nski gasket admits a stationary random rooting.
	
	A \textbf{rooted network} $(G,\rho,\sigma)$ adds vertex marks $\sigma\colon V \to \R^m$ to a rooted graph; isomorphisms must preserve all marks. The notions of stationarity, rerooting-invariance, and ergodicity extend verbatim, with $h(G,v,\sigma)$ replacing $h(G,v)$:
	\[
	\E[h(G,\rho,\sigma)] = \E\biggl[\frac{1}{\deg(\rho)}\sum_{y \sim \rho} h(G,y,\sigma)\biggr]\,.
	\]
	
	\begin{proposition}\label{prop:iid-stationary}
		Let $(G,\rho)$ be a stationary random rooted graph, and given~$G$, let $(\sigma(v))_{v\in V}$ be i.i.d.\ with a fixed common law~$\nu$, sampled independently of~$(G,\rho)$. Then $(G,\rho,\sigma)$ is a stationary random rooted network.
	\end{proposition}
	
	\begin{proof}
		Let $h$ be a non-negative Borel function on rooted networks and define the function $\bar{h}(G,v) \coloneqq \E[h(G,v,\sigma) \mid G]$. Since the marks are i.i.d., $\bar{h}$ depends only on the isomorphism class of~$(G,v)$. Iterated expectation and the stationarity of~$(G,\rho)$ give 
		\[
		\E[h(G,\rho,\sigma)]=\E[\bar{h}(G,\rho)]=\E[\frac{1}{\deg(\rho)}\sum_{y\sim\rho}\bar{h}(G,y)]=\E[\frac{1}{\deg(\rho)}\sum_{y\sim\rho}h(G,y,\sigma)]\,.
		\]
	\end{proof}
	
	\begin{theorem}[Phase transition on stationary graphs]\label{thm:stationary-phase}
		Let $(G,\rho)$ be a stationary random rooted graph that is almost surely infinite, and let $(\sigma(v))_{v\in V}$ be i.i.d.\ random variables sampled independently of~$(G,\rho)$, with $\E[\sigma] = \mu$ and $\E[|\sigma(\rho)|/\deg(\rho)] < \infty$.
		\begin{enumerate}[label=\textup{(\roman*)}]
			\item If $\mu > 1$, then $\P(\sigma \textup{ stabilizes}) = 0$.
			\item If $\mu < 1$, then $\P(\sigma \textup{ stabilizes}) = 1$.
		\end{enumerate}
	\end{theorem}
	
	The key ingredient is conservation of the mass $\E[\sigma_k(\rho)/\deg(\rho)]$ under parallel toppling.
	
	\begin{theorem}[Conservation of degree-weighted mass]\label{thm:stationary-toppling}
		Let $(G,\rho,\sigma)$ be a stationary random rooted network with $\E[|\sigma(\rho)|/\deg(\rho)] < \infty$. Under parallel toppling, for each $k \geq 0$:
		\begin{itemize}
			\item[\em (a)] The random rooted network $(G,\rho,\sigma_k,u_k)$ is stationary.
			\item[\em (b)] The degree-weighted mass is conserved: $\E[\sigma_k(\rho)/\deg(\rho)] = \E[\sigma(\rho)/\deg(\rho)]$.
		\end{itemize}
	\end{theorem}
	
	\begin{proof}
		(a) Since the parallel toppling rule depends only on the graph structure and the initial configuration, the map $(G,v,\sigma)\mapsto(G,v,\sigma_k,u_k)$ is isomorphism-equivariant. Composing with every Borel test function $h$ and applying the stationarity of $(G,\rho,\sigma)$ gives stationarity of $(G,\rho,\sigma_k,u_k)$.
		
		\medskip
		
		(b) We prove the identity and the integrability $\E[|\sigma_k(\rho)|/\deg(\rho)]<\infty$ simultaneously by induction on~$k$. For $k=0$, integrability is the hypothesis and the identity is trivial. Suppose both hold at step~$k-1$. The parallel toppling update reads
		\begin{equation}\label{eq:stationary-update}
			\sigma_k(v) = \min\bigl(\sigma_{k-1}(v),\,1\bigr) + \sum_{w \sim v} \frac{(\sigma_{k-1}(w)-1)^+}{\deg(w)}\,.
		\end{equation}
		By part~(a), $(G,\rho,\sigma_{k-1},u_{k-1})$ is stationary, hence so is $(G,\rho,\sigma_{k-1})$ by forgetting the~$u$-mark. The function $h(G,v,\sigma_{k-1}) = (\sigma_{k-1}(v)-1)^+/\deg(v)$ is integrable since $(a-1)^+\leq|a|+1$ and $\E[|\sigma_{k-1}(\rho)|/\deg(\rho)]<\infty$. Applying the stationarity identity to~$h$ gives
		\begin{equation}\label{eq:stationary-applied}
			\E\biggl[\frac{1}{\deg(\rho)}\sum_{w \sim \rho} \frac{(\sigma_{k-1}(w)-1)^+}{\deg(w)}\biggr] = \E\biggl[\frac{(\sigma_{k-1}(\rho)-1)^+}{\deg(\rho)}\biggr]\,.
		\end{equation}
		Dividing~\eqref{eq:stationary-update} by $\deg(\rho)$, taking expectations, and substituting~\eqref{eq:stationary-applied} yields
		\[
		\E\biggl[\frac{\sigma_k(\rho)}{\deg(\rho)}\biggr] = \E\biggl[\frac{\min(\sigma_{k-1}(\rho),\,1) + (\sigma_{k-1}(\rho)-1)^+}{\deg(\rho)}\biggr] = \E\biggl[\frac{\sigma_{k-1}(\rho)}{\deg(\rho)}\biggr]\,.
		\]
		The second equality is the identity $\min(a,1) + (a-1)^+ = a$. For the integrability: since the update~\eqref{eq:stationary-update} is $\min(\sigma_{k-1}(v),1)$ plus a nonnegative inflow, we have $\sigma_k(v) \geq \sigma_0(v) \wedge 1$ by induction. In particular, $\E[\sigma_k(\rho)^-/\deg(\rho)] < \infty$. For the positive part, \eqref{eq:stationary-update} and~\eqref{eq:stationary-applied} give
		\[
		\E\biggl[\frac{\sigma_k(\rho)^+}{\deg(\rho)}\biggr] \leq \E\biggl[\frac{1}{\deg(\rho)}\biggr] + \E\biggl[\frac{(\sigma_{k-1}(\rho)-1)^+}{\deg(\rho)}\biggr] < \infty\,.
		\]
		This completes the induction.
	\end{proof}
	
	\begin{lemma}[i.i.d.\ marking preserves ergodicity]\label{lem:ergodic-marked-stationary}
		Let $(G,\rho)$ be a stationary and ergodic random rooted graph that is almost surely infinite. Given $(G,\rho)$, let $(\sigma(v))_{v\in V}$ be i.i.d.\ random variables with common law~$\nu$, sampled independently of~$(G,\rho)$. Then $(G,\rho,\sigma)$ is a stationary and ergodic random rooted network.
	\end{lemma}
	
	\begin{proof}
		Stationarity is Proposition~\ref{prop:iid-stationary}. For ergodicity, let $A$ be a rerooting-invariant event in $\sigma(G,\rho,\sigma)$. The function $q(G,o)\coloneqq\P((G,o,\sigma)\in A\mid G,o)$ is rerooting-invariant, so by ergodicity of~$(G,\rho)$ there exists $c\in[0,1]$ with $q(G,\rho)=c$ a.s.\ and $c=\P(A)$.
		
		Fix $\eps>0$ and approximate~$A$ by a Borel function $f$ depending only on the marked ball of radius~$r$ around the root, with $0\leq f\leq 1$ and $\E|\one_A-f(G,\rho,\sigma)|<\eps$. Set $\bar{f}(G,o)\coloneqq\E[f(G,o,\sigma)\mid G,o]$; then $\E|\bar{f}(G,\rho)-c|<\eps$ by Jensen's inequality.
		
		Let $(X_n)_{n\geq 0}$ be simple random walk on~$G$ started at~$\rho$, independent of~$\sigma$, and write $f_n\coloneqq f(G,X_n,\sigma)$ and $\bar{f}_n\coloneqq\bar{f}(G,X_n)$. On the event $D_n\coloneqq\{\mathrm{dist}_G(X_0,X_n)>2r\}$, the variables $f_0$ and $f_n$ depend on disjoint i.i.d.\ marks conditionally on $(G,\rho,X_n)$, so $\E[f_0 f_n\one_{D_n}\mid G,\rho,X_n]=\bar{f}_0\bar{f}_n\one_{D_n}$. Since $\P_\rho(X_n=v)\to 0$ for each~$v$ on infinite graphs \citep[Exercise~2.1(f)]{LP16}, dominated convergence gives $\P(D_n^c)\to 0$. Choosing~$n$ large enough that $\P(D_n^c)<\eps$, the triangle inequality yields
		\[
		|\P(A)-c^2|
		\leq\underbrace{|\P(A)-\E[f_0 f_n]|}_{<\,2\eps}
		+\underbrace{|\E[f_0 f_n]-\E[\bar{f}_0\bar{f}_n]|}_{\leq\,\P(D_n^c)\,<\,\eps}
		+\underbrace{|\E[\bar{f}_0\bar{f}_n]-c^2|}_{<\,2\eps}<5\eps\,.
		\]
		Since $\eps>0$ was arbitrary, $c=c^2$, so $\P(A)\in\{0,1\}$.
	\end{proof}
	
	\begin{lemma}[$0$--$1$ law for stabilization]\label{lem:01-stationary}
		Let $(G,\rho)$ be a stationary random rooted graph that is almost surely infinite, and let $(\sigma(v))_{v\in V}$ be i.i.d.\ random variables sampled independently of~$(G,\rho)$. Then
		\[
		\P(\sigma\textup{ stabilizes}\mid\mathcal{I}_G)\in\{0,1\}\qquad\text{a.s.}
		\]
	\end{lemma}
	
	\begin{proof}
		Let $A\coloneqq\{\sigma\text{ stabilizes}\}$. Since stabilization does not depend on the choice of root, the event~$A$ is rerooting-invariant. By the ergodic decomposition~\citep[Section~2.1]{BenjaminiCurien12}, conditionally on~$\mathcal{I}_G$ the law of $(G,\rho)$ is stationary and ergodic. Since $\mathcal{I}_G\subseteq\sigma(G,\rho)$ and the masses are sampled independently of~$(G,\rho)$, conditionally on~$\mathcal{I}_G$ they remain i.i.d.\ and independent of~$(G,\rho)$. Lemma~\ref{lem:ergodic-marked-stationary} applied on each ergodic component gives $\P(A\mid\mathcal{I}_G)\in\{0,1\}$ a.s.
	\end{proof}
	
	Proposition~\ref{prop:01-law} below gives a different $0$--$1$ law, valid on every fixed locally finite connected graph without stationarity, using the random walk representation.
	
	\begin{proof}[Proof of Theorem~\ref{thm:stationary-phase}]
		We prove both parts by contrapositive, following the same strategy as~\citet[Lemmas~4.1 and~4.2]{div-sand-crit}. Let $A\coloneqq\{\sigma\text{ stabilizes}\}$. By Lemma~\ref{lem:01-stationary}, $\P(A\mid\mathcal{I}_G)\in\{0,1\}$ a.s. Let $B\coloneqq\{\P(A\mid\mathcal{I}_G)=1\}\in\mathcal{I}_G$. Then
		\[
		\P(A\cap B^c)=\E[\one_{B^c}\P(A\mid\mathcal{I}_G)]=0,\qquad \P(A^c\cap B)=\E[\one_B\P(A^c\mid\mathcal{I}_G)]=0\,.
		\]
		Therefore $A=B$ a.s.\ and $\P(A)=\P(B)$. For $E\in\{B,B^c\}$ with $\P(E)>0$, write $\P_E(\cdot)\coloneqq\P(\cdot\mid E)$ and $\E_E[\cdot]\coloneqq\E[\cdot\mid E]$. Since $B\in\mathcal{I}_G$ is rerooting-invariant, conditioning on~$B$ or~$B^c$ preserves stationarity of $(G,\rho)$; since~$B$ is graph-measurable, the masses remain i.i.d.\ with common mean~$\mu$ and independent of~$(G,\rho)$ under each conditional law. By Theorem~\ref{thm:stationary-toppling}, conservation holds under each conditional law. Under $\P_B$ the sandpile stabilizes a.s.; under $\P_{B^c}$ it does not.
		
		\medskip
		\textup{(i)}\;
		Assume $\P(B)>0$. Fix $M>1$ and set $\sigma^{(M)}\coloneqq\sigma\wedge M$. Since $\sigma^{(M)}\leq\sigma$ pointwise, induction on~$n$ using~\eqref{eq:u-update} and~\eqref{eq:sigma-update} gives $u_n(\cdot;\sigma^{(M)})\leq u_n(\cdot;\sigma)$ for all~$n$, so $\sigma^{(M)}$ also stabilizes $\P_B$-a.s. On a fixed graph, the per-step emission $e_k(v)\coloneqq(\sigma_k^{(M)}(v)-1)^+/\deg(v)$ satisfies $0\leq e_k(v)\leq M-1$ for every $k$ and $v$: at $k=0$ this holds since $\sigma_0^{(M)}\leq M$, and if $e_{k-1}(w)\leq M-1$ for all~$w$, then~\eqref{eq:stationary-update} gives $(\sigma_k^{(M)}(v)-1)^+\leq\sum_{w\sim v}e_{k-1}(w)\leq\deg(v)(M-1)$. Since $\sigma^{(M)}$ stabilizes, $e_k(\rho)\to 0$ $\P_B$-a.s., and dominated convergence gives $\E_B[e_k(\rho)]\to 0$.
		
		Under $\P_B$ the truncated masses are i.i.d.\ with mean $\E[\sigma\wedge M]$ and independent of the graph. Conservation (Theorem~\ref{thm:stationary-toppling}) therefore gives
		\[
		\E_B\biggl[\frac{\sigma_k^{(M)}(\rho)}{\deg(\rho)}\biggr] = \E[\sigma\wedge M]\,\E_B\biggl[\frac{1}{\deg(\rho)}\biggr]\,.
		\]
		The identity $a=\min(a,1)+(a-1)^+$ and the bound $\min(\sigma_k^{(M)}(\rho),1)\leq 1$ yield
		\[
		\E[\sigma\wedge M]\,\E_B\biggl[\frac{1}{\deg(\rho)}\biggr] \leq \E_B\biggl[\frac{1}{\deg(\rho)}\biggr] + \E_B[e_k(\rho)]\,.
		\]
		Letting $k\to\infty$ and dividing by $\E_B[1/\deg(\rho)]>0$ gives $\E[\sigma\wedge M]\leq 1$. Since $\sigma$ is independent of $(G,\rho)$, the hypothesis $\E[|\sigma(\rho)|/\deg(\rho)]<\infty$ gives $\E[|\sigma|]\cdot\E[1/\deg(\rho)]<\infty$ and hence $\E[|\sigma|]<\infty$. Sending $M\to\infty$ by dominated convergence yields $\mu\leq 1$.
		
		\medskip
		\textup{(ii)}\;
		Assume $\P(B^c)>0$. If $u_\infty(v)=\infty$ and $u_\infty(w)<\infty$ for some neighbor $w\sim v$, then $\sigma_n(w)\geq\sigma(w)+u_n(v)-\deg(w)\,u_n(w)\to\infty$, contradicting $u_\infty(w)<\infty$. By connectivity, $u_\infty\equiv\infty$ on~$V$ under~$\P_{B^c}$. Since $\rho$ topples at least once and~\eqref{eq:stationary-update} preserves $\sigma_k(\rho)\geq 1$, we have $\liminf_k\sigma_k(\rho)\geq 1$ $\P_{B^c}$-a.s. The variables $Y_k\coloneqq(\sigma_k(\rho)-\sigma(\rho)\wedge 1)/\deg(\rho)$ are non-negative, and conservation gives $\E_{B^c}[Y_k]=(\mu-\E[\sigma\wedge 1])\,\E_{B^c}[1/\deg(\rho)]$. By Fatou's lemma and $\liminf_k\sigma_k(\rho)\geq 1$, dividing by $\E_{B^c}[1/\deg(\rho)]>0$ gives $\mu\geq 1$.
		
		\medskip
		If $\mu>1$, then part~(i) gives $\P(B)=0$, hence $\P(A)=\P(B)=0$. If $\mu<1$, then part~(ii) gives $\P(B^c)=0$, hence $\P(A)=\P(B)=1$.
	\end{proof}
	
	\section{Random walk representation}\label{sec:rw-rep}
	In this section we fix a locally finite connected graph $G=(V,E)$ and an initial configuration $\sigma\colon V\to\R$. All walk expectations $\E_x$ are with respect to the simple random walk on~$G$; when $\sigma$ is random, the statements below are to be read conditionally on~$\sigma$.
	
	The parallel toppling update~\eqref{eq:sigma-update} may be written compactly as $\sigma_{n+1} = \sigma_n + \Delta[(\sigma_n - 1)^+/\deg]$, where $(\sigma_n - 1)^+/\deg$ denotes the function $x \mapsto (\sigma_n(x) - 1)^+/\deg(x)$. It is immediate by induction that $\sigma_n = \sigma + \Delta u_n$ for every~$n \geq 0$.
	Define the \textbf{excess mass} $\xi(v)\coloneqq\sigma(v)-1$ and write $\zeta(v)\coloneqq\xi(v)/\deg(v)$ for the per-neighbor excess (so that $S_n=\sum_{k=0}^{n-1}\xi(X_k)/\deg(X_k)=\sum_{k=0}^{n-1}\zeta(X_k)$ as in Theorem~\ref{thm:OS}). Define the \textbf{local time} $L_n(v)\coloneqq\sum_{k=0}^{n-1}\one_{\{X_k=v\}}$; for a stopping time~$\tau$, write $L_\tau(v)\coloneqq\sum_{k<\tau}\one_{\{X_k=v\}}$.
	The parallel toppling odometer satisfies the following dynamic programming recursion.
	\begin{lemma}\label{lem:recursion}
		For all $n \geq 0$ and $x \in V$,
		\[
		u_{n+1}(x) = \biggl(\frac{1}{\deg(x)}\sum_{y \sim x} u_n(y) + \zeta(x)\biggr)^+\,.
		\]
	\end{lemma}
	\begin{proof}
		Substituting $\sigma_n = \sigma + \Delta u_n$ into~\eqref{eq:u-update} and using $a + (b - a)^+ = \max(a, b)$ gives
		\[
		u_{n+1}(x) = \max\biggl(u_n(x),\; \frac{1}{\deg(x)}\sum_{y \sim x} u_n(y) + \zeta(x)\biggr)\,.
		\]
		Since $u_n \geq u_{n-1}$ pointwise and $u_0 = 0$, induction gives $u_n(x) \leq \bigl(\frac{1}{\deg(x)}\sum_{y \sim x} u_n(y) + \zeta(x)\bigr)^+$, so the maximum reduces to the positive part.
	\end{proof}
	
	We next give a random walk representation of the odometer via the \textbf{walk payoff} $(S_n)_{n\geq 0}$ defined by $S_0=0$ and $S_n=\sum_{k=0}^{n-1}\zeta(X_k)$.
	For a stopping time~$\tau$, set $S_\tau\coloneqq\sum_{k<\tau}\zeta(X_k)$.
	\begin{theorem}[Random walk representation]\label{thm:RW}
		Let~$v_n(x) \coloneqq \sup_{\tau \leq n} \E_x[S_\tau]$, where the supremum is taken over stopping times (with respect to the natural filtration of $(X_k)_{k\in\N}$) bounded by $n$.
		For all $n \geq 0$, we have
		\[
		u_n(x) = v_n(x) = \E_x\bigl[S_{\tau_n^*}\bigr]\,.
		\]
		Here $\tau_n^* \coloneqq \min\bigl\{0 \leq k \leq n : v_{n-k}(X_k) = 0\bigr\}$.
	\end{theorem}
	\begin{proof}
		By Lemma~\ref{lem:recursion} it suffices to show that $v_n$ satisfies the same recursion as $u_n$ and that the supremum defining $v_n$ is attained by $\tau_n^*$. We prove both claims by induction on~$n$; both hold trivially for $n = 0$.
		For the inductive step, every stopping time $\tau \leq n+1$ either satisfies $\tau = 0$, giving $\E_x[S_\tau] = 0$, or $\tau \geq 1$. In the latter case, the shifted variable $\tau' \coloneqq \tau - 1$ is a stopping time for the walk $(X_{1+k})_{k \geq 0}$ with $\tau' \leq n$. By the Markov property, conditionally on $X_1 = y$, this walk has law $\P_y$, so $\E_y[S_{\tau'}] \leq v_n(y)$ and
		\[
		\E_x[S_\tau] = \zeta(x) + \E_x\bigl[\E_{X_1}[S_{\tau'}]\bigr] \leq \zeta(x) + \frac{1}{\deg(x)}\sum_{y \sim x} v_n(y)\,.
		\]
		Since $\tau = 0$ gives value~$0$, maximizing over all $\tau \leq n+1$ yields
		\[
		v_{n+1}(x) \leq \Bigl(\frac{1}{\deg(x)}\sum_{y \sim x} v_n(y) + \zeta(x)\Bigr)^+\,.
		\]
		For the matching lower bound, if $\zeta(x) + \frac{1}{\deg(x)}\sum_{y \sim x} v_n(y) \leq 0$, then $v_{n+1}(x)=0=(\ldots)^+$. Otherwise, by the inductive hypothesis the stopping time $\tau\coloneqq 1+\tau_n^*\circ\theta_1$ (where $\theta_1$ is the time shift) satisfies $\tau\leq n+1$ and the Markov property gives $\E_x[S_\tau] = \zeta(x) + \frac{1}{\deg(x)}\sum_{y \sim x} v_n(y)$, matching the upper bound. Since $\tau_{n+1}^*$ equals~$\tau$ when $v_{n+1}(x)>0$ and equals~$0$ otherwise, the attainment claim follows.
	\end{proof}

	The random walk representation also yields an exact characterization of the odometer at $o$ in terms of voltage sums over finite connected sets. This is the starting point for the counterexample constructions in Section~\ref{sec:counterexamples}.
	
	For a finite connected $C\subseteq V$ with $o\in C$ and $V\setminus C\neq\varnothing$, the killed Green function $g_C(o,v)$ from~\eqref{eq:gC-PDE} satisfies
	\begin{equation}\label{eq:voltage-identity}
		\E_o[S_{\tau_C}]=\sum_{v\in C}g_C(o,v)\bigl(\sigma(v)-1\bigr)\,.
	\end{equation}
	When the source~$o$ is understood, we write $g_C(v)$ for $g_C(o,v)$.
	
	For a proper finite connected $K\ni o$, define
	\[
	v_K(x)\coloneqq\sup_{\tau\leq\tau_K}\E_x[S_\tau]\,.
	\]
	The supremum is over stopping times bounded by~$\tau_K$.
	
	\begin{proposition}[Finite-volume reduction]\label{prop:finite-vol}
		Let $K\subseteq V$ be finite, connected, and proper, with $o\in K$. Then $v_K(x)<\infty$ for every $x\in V$. Let $D_K\coloneqq\{x\in K:v_K(x)>0\}$ and let $D_K(o)$ denote the connected component of~$o$ in~$D_K$ (with $D_K(o)\coloneqq\varnothing$ if $o\notin D_K$). If $o\notin D_K$, then $v_K(o)=0$. If $o\in D_K$, then
		\begin{equation}
			v_K(o)=\E_o\bigl[S_{\tau_{D_K(o)}}\bigr]\,.
		\end{equation}
		In particular,
		\begin{equation}\label{eq:vK-sup}
			v_K(o)=\max\biggl\{0,\;\sup_{\substack{C\subseteq K,\; C\textup{ connected}\\ o\in C}}\sum_{v\in C}g_C(v)\bigl(\sigma(v)-1\bigr)\biggr\}\,.
		\end{equation}
	\end{proposition}
	
	\begin{proof}
		Since $K$ is finite and proper, $\E_x[\tau_K]<\infty$ for every $x\in K$, so $v_K(x)<\infty$. If $x\notin K$, then $\tau_K=0$ and $v_K(x)=0$. The same first-step decomposition as in the proof of Theorem~\ref{thm:RW} gives
		\begin{equation}\label{eq:vK-bellman}
			v_K(x)=\biggl(\zeta(x)+\frac{1}{\deg(x)}\sum_{y\sim x}v_K(y)\biggr)^{\!+},\qquad x\in K\,.
		\end{equation}
		Let $D\coloneqq D_K(o)$. If $o\notin D_K$, then $v_K(o)=0$. If $o\in D$, then the positive part in~\eqref{eq:vK-bellman} is not active on~$D$, and since $D$ is a connected component of~$D_K$, every neighbor of~$D$ outside~$D$ has $v_K=0$. The exit payoff $x\mapsto\E_x[S_{\tau_D}]$ satisfies the same equation on~$D$ with the same zero boundary values (by conditioning on the first step), so the maximum principle gives $v_K(o)=\E_o[S_{\tau_D}]$.	For~\eqref{eq:vK-sup}, every connected $C\subseteq K$ with $o\in C$ gives $\tau_C\leq\tau_K$, so $\E_o[S_{\tau_C}]\leq v_K(o)$. The identity~\eqref{eq:voltage-identity} and the connectedness of~$D$ then yield~\eqref{eq:vK-sup}.
	\end{proof}
	
	\begin{theorem}[Nested-volume representation]\label{thm:nested-vol}
		Assume $G$ is infinite. Then
		\begin{equation}\label{eq:nested-exact}
			u_\infty(o)=\max\biggl\{0,\;\sup_{\substack{C\subseteq V,\; C\textup{ finite connected}\\ o\in C}}\sum_{v\in C}g_C(v)\bigl(\sigma(v)-1\bigr)\biggr\}\,.
		\end{equation}
	\end{theorem}
	
	\begin{proof}
		For the upper bound, fix $n\geq 1$. For every stopping time $\tau\leq n$ and every $k<\tau$, the walk satisfies $\mathrm{dist}(o,X_k)\leq k\leq n-1$, so $X_k\in B(o,n-1)$ and therefore $\tau\leq\tau_{B(o,n-1)}$. Since $G$ is infinite, the ball $B(o,n-1)$ is a proper finite connected subset containing~$o$, and
		\[
		u_n(o)=\sup_{\tau\leq n}\E_o[S_\tau]\leq v_{B(o,n-1)}(o)\,.
		\]
		By Proposition~\ref{prop:finite-vol}, the right-hand side is $ \max\Bigl\{0,\,\sup_{\substack{C\subseteq B(o,n-1),\; C\textup{ connected}\\ o\in C}}\sum_{v\in C} g_C(v)(\sigma(v)-1)\Bigr\}$.     
		Taking the supremum over $n\geq 1$ and using that every finite connected $C$, with $o\in C$, lies in some ball gives the upper bound in~\eqref{eq:nested-exact}.
		
		For the lower bound, let $C\subseteq V$ be finite and connected with $o\in C$ and $V\setminus C\neq\varnothing$. Since $C$ is finite, $\E_o[\tau_C]<\infty$ and $|S_{\tau_C\wedge m}|\leq\max_{v\in C}|\zeta(v)|\cdot\tau_C$, so dominated convergence gives $\E_o[S_{\tau_C\wedge m}]\to\E_o[S_{\tau_C}]$ as $m\to\infty$. Each $\tau_C\wedge m$ is bounded, so Theorem~\ref{thm:RW} gives $\E_o[S_{\tau_C\wedge m}]\leq u_\infty(o)$. Passing to the limit and using~\eqref{eq:voltage-identity} yields
		\[
		\sum_{v\in C}g_C(v)\bigl(\sigma(v)-1\bigr)=\E_o[S_{\tau_C}]\leq u_\infty(o)\,.
		\]
		Taking the supremum over~$C$ and the maximum with~$0$ gives the matching lower bound.
	\end{proof}
	
	\begin{corollary}\label{cor:RW-infinite}
		For every configuration $\sigma$ and every $x\in V$,
		\[
		u_\infty(x)=\sup_{\tau\ \mathrm{bounded}}\E_x[S_\tau\mid\sigma]\,.
		\]
	\end{corollary}
	\begin{proof}
		Since $u_n(x)\uparrow u_\infty(x)$ and every bounded stopping time satisfies $\tau\leq n$ for some deterministic~$n$, it follows that $u_\infty(x)=\sup_{\tau\textup{ bounded}}\E_x[S_\tau\mid\sigma]$.
	\end{proof}
	
	\begin{remark}\label{rem:bolza}
		The recursion in Lemma~\ref{lem:recursion} is a Wald--Bellman equation. The random walk representation, Theorem~\ref{thm:RW}, is an instance of the Bolza formulation of optimal stopping for Markov chains: set $g\equiv 0$, $\alpha=1$, and $c=-\zeta$ in the notation of~\citep[Section~1.2, eqs.~(1.2.65) and~(1.2.73)]{PS06}.
	\end{remark}
	
	\section{Zero-one law for stabilization}\label{sec:01-law}
	On stationary random rooted graphs, the $0$--$1$ law for stabilization is a consequence of ergodicity (Lemma~\ref{lem:01-stationary}). On a fixed graph without symmetry, a different argument is needed. The random walk representation reduces the question to classical tail and exchangeability considerations. On a transient graph only independence is used: changing finitely many masses perturbs the odometer by a finite amount, bounded by Green function values. Stabilization is therefore a tail event, and Kolmogorov's $0$--$1$ law applies.  On a recurrent graph the Green function is infinite and the tail-event argument fails. Instead, we show that swapping the masses at two vertices perturbs the expected walk payoff by a bounded amount, using the voltage between the two vertices. Stabilization is therefore invariant under transpositions, and the Hewitt--Savage $0$--$1$ law applies and requires identical distribution. On recurrent graphs, independence is not enough: Example~\ref{ex:01-fails} below gives independent non-identically distributed masses for which $\P(\sigma\text{ stabilizes})\in(0,1)$.

	\begin{proposition}\label{prop:01-law}
		Let $G=(V,E)$ be an infinite, locally finite, connected graph and let
		$(\sigma(v))_{v\in V}$ be i.i.d.\ random variables. Then
		$\P(\sigma \text{ stabilizes})\in\{0,1\}$ and $\P(\sup_n\E_o[S_n\mid\xi]=\infty)\in\{0,1\}$.
	\end{proposition}
	\begin{proof}
		By Corollary~\ref{cor:RW-infinite}, $u_\infty(x)=\sup_{\tau\textup{ bounded}}\E_x[S_\tau\mid\sigma]$.
		
		\emph{Transient case.} Since the initial masses are independent, by Kolmogorov's $0$--$1$ law it suffices to show that changing $\sigma$ on a finite set $J\subset V$ does not affect stabilization. Let $\sigma'$ agree with $\sigma$ outside~$J$. For every bounded stopping time $\tau$,
		\[
		\bigl|\E_x[S'_\tau-S_\tau]\bigr|=\biggl|\sum_{v\in J}(\sigma'(v)-\sigma(v))\frac{\E_x[L_\tau(v)]}{\deg(v)}\biggr|\leq\sum_{v\in J}|\sigma'(v)-\sigma(v)|\,g(x,v)\eqqcolon C_J(x)\,,
		\]
		and $C_J(x)<\infty$ by~\eqref{eq:green-def}. Hence $u_\infty(x;\sigma')\leq u_\infty(x;\sigma)+C_J(x)$ and $u_\infty(x;\sigma)\leq u_\infty(x;\sigma')+C_J(x)$, so $u_\infty(x;\sigma)<\infty$ if and only if $u_\infty(x;\sigma')<\infty$.
		
		\emph{Recurrent case.} By the Hewitt--Savage zero--one law, it suffices to show that swapping $\sigma(a)$ and $\sigma(b)$ for distinct vertices $a,b\in V$ does not affect stabilization.  The function
		\[
		f(x) \coloneqq \frac{\P_x(T_a<T_b)}{\deg(a)\,\P_a(T_b < T_a^+)}
		\]
		satisfies $0 \le f \le \|f\|_\infty= 1/(\deg(a)\,\P_a(T_b < T_a^+))<\infty$ and $\Delta f = \delta_b - \delta_a$ (see, e.g., \citet[Proposition~2.1 and eq.~(2.4)]{LP16}). Let $\sigma'$ denote the swapped configuration, so that $\zeta'-\zeta = -(\sigma(b)-\sigma(a))\,\Delta f/\deg$. The process
		\[
		M_n\coloneqq f(X_n)-\sum_{k=0}^{n-1}\frac{\Delta f(X_k)}{\deg(X_k)}
		\]
		is a $\P_x$-martingale, and optional stopping at every bounded stopping time $\tau$ gives
		\[
		\E_x[S'_\tau - S_\tau\mid\sigma] = \bigl(\sigma(b)-\sigma(a)\bigr)\bigl(f(x) - \E_x[f(X_\tau)]\bigr)\,.
		\]
		Since $0\leq f\leq\|f\|_\infty$, $|\E_x[S'_\tau - S_\tau\mid\sigma]|\leq|\sigma(b)-\sigma(a)|\,\|f\|_\infty$. Taking the supremum over bounded $\tau$ yields $u_\infty(x;\sigma')\leq u_\infty(x;\sigma)+|\sigma(b)-\sigma(a)|\,\|f\|_\infty$, and the reverse inequality follows by symmetry. Hence $\sigma$ stabilizes if and only if $\sigma'$ does.
		
		Since every deterministic~$n$ is a bounded stopping time, the same arguments give the second conclusion. 
	\end{proof}
	
	\section{Supercritical regime}\label{sec:supercritical}
	
	Recall $S_n$, $D_n$ and $W_n$ from equation \eqref{eq:dn-wn-stuff}.
	The drift $D_n$ does not depend on the scenery. Its expectation is $\E[\xi]\,A_n(x)$, where $A_n(x)=\E_x[\sum_{k<n}1/\deg(X_k)]$ is the inverse-degree clock from~\eqref{eq:gn-def}. On a graph with degree bounded by~$d$, $A_n(x)\geq n/d$, so the drift grows at least linearly when $\E[\xi]>0$. The fluctuation $W_n$ weights the i.i.d.\ centered variables $\xi(v)-\E[\xi]$ by the local times $L_n(v)/\deg(v)$; conditionally on the walk, it is a sum of independent centered random variables weighted by local times~\citep{KestenSpitzer79}. Proposition~\ref{prop:supercritical} shows that the walk-averaged fluctuation $R_n\coloneqq\E_x[W_n\mid\xi]$ is negligible compared to the drift.
	
	\begin{lemma}[No-dominance for the clock]\label{lem:clock-no-dom}
		Let $G$ be infinite, locally finite, and connected, and fix $x\in V$. If $A_n(x)\to\infty$, then
		\[
		\frac{\sup_v g_n(x,v)}{A_n(x)}\to 0\qquad\text{and hence}\qquad\frac{\Sigma_n(x)}{A_n(x)^2}\to 0\,.
		\]
	\end{lemma}
	\begin{proof}
		By reversibility,
		\[
		g_n(x,v)=\E_v[L_n(x)]/\deg(x)\,.
		\]
		The first-visit decomposition gives $\E_v[L_n(x)]\leq\E_x[L_n(x)]$, so $\sup_v g_n(x,v)\leq g_n(x,x)$.
		If the walk is transient, then $g_n(x,x)\leq g(x,x)<\infty$ while $A_n(x)\to\infty$. If the walk is recurrent, then $g_n(x,x)\to\infty$. For each fixed~$v$, the ratio $g_n(x,v)/g_n(x,x)$ tends to~$1$; to see this, note that the first-visit decomposition gives
		\[
		\E_v[L_n(x)]=\sum_{j=0}^{n-1}\P_v(T_x=j)\,\E_x[L_{n-j}(x)]\,.
		\]
		Since $\P_x(X_k=x)\to 0$ on infinite graphs, 
		also $\E_x[L_n(x)]-\E_x[L_{n-j}(x)]=\sum_{k=n-j}^{n-1}\P_x(X_k=x)$
		tends to zero, so $\E_x[L_{n-j}(x)]/\E_x[L_n(x)]$ tends to $1$. By recurrence $\sum_j\P_v(T_x=j)=1$, so dominated convergence yields the claim. For each fixed~$r$, all vertices $v\in B(x,r)$ therefore satisfy $g_n(x,v)\geq g_n(x,x)/2$ for all large~$n$, giving $A_n(x)\geq|B(x,r)|g_n(x,x)/2$ and hence $\limsup_n g_n(x,x)/A_n(x)\leq 2/|B(x,r)|$. Since $|B(x,r)|\to\infty$ as $r\to\infty$, the ratio tends to zero. The second claim follows from $\Sigma_n(x)\leq\sup_v g_n(x,v)\cdot A_n(x)$.
	\end{proof}
	
	\begin{proposition}\label{prop:supercritical}
		Let~$G = (V,E)$ be an infinite, locally finite, connected graph, let~$(X_n)_{n\geq 0}$ be simple random walk on~$G$, and let $(\xi(v))_{v\in V}$ be i.i.d.\ and independent of the walk with
		$\E[\xi] \in (0, \infty]$. If $A_n(x)\to\infty$ for some $x\in V$, then $\sup_n\E_x[S_n\mid\xi]=\infty$ almost surely.
		In particular, if the degree is bounded by~$d$, then $A_n(x)\geq n/d\to\infty$, so the conclusion holds on every infinite bounded-degree graph.
	\end{proposition}
	\begin{proof}
		Since $\E[\xi]>0$, we have $\E[\xi^-]<\infty$. If $\E[\xi]=+\infty$, then replacing $\xi(v)$ by $\min(\xi(v),K)$ gives $\E[\min(\xi,K)]>0$ for large~$K$ so the finite-mean case below applies. We may therefore assume $\E[\xi]<\infty$. Taking walk expectations in~\eqref{eq:dn-wn-stuff},
		\[
		\E_x[S_n \mid \xi] = \E[\xi]\,A_n(x) + R_n\,.
		\]
		Here $R_n \coloneqq \sum_{v} g_n(x,v)\bigl(\xi(v)-\E[\xi]\bigr)$ is a sum of independent mean-zero random variables weighted by the Green kernel. Since $A_n(x)\to\infty$, it suffices to show $R_n/A_n(x) \to 0$ in $\P$-probability. Indeed, for every $B > 0$ and all $n$ with $\E[\xi]\, A_n(x)/2>B$,
		\[
		\bigl\{\E_x[S_n\mid\xi] \leq B\bigr\} \subseteq \Bigl\{\frac{R_n}{A_n(x)} \leq -\frac{\E[\xi]}{2}\Bigr\}\,.
		\]
		It follows that $\sup_n \E_x[S_n\mid\xi] = \infty$ a.s.
		
		We show $R_n/A_n(x)\to 0$ in probability by splitting $R_n$ into a bounded truncation and a small tail. Fix~$\delta \in (0,1/2]$ and write $R_n = Y_n + Z_n$ with
		\[Y_n \coloneqq \sum_{v} g_n(x,v)\bigl(\xi(v)-\E[\xi]\bigr)\one_{\{|\xi(v)-\E[\xi]|\leq M\}}\,.
		\qquad
		Z_n \coloneqq R_n - Y_n\,.
		\]
		Choose $M$ large enough that $\E\bigl[|\xi - \E[\xi]| \one_{\{|\xi-\E[\xi]|>M\}}\bigr] < \delta^2$.
		Markov's inequality controls the tail: $\P(|Z_n| > 2\delta A_n(x)) \leq \delta/2$. By centering of $R_n$, $|\E[Y_n]|=|\E[Z_n]|\leq\E[|Z_n|]\leq\delta^2 A_n(x)$. For the bounded part, independence gives $\var(Y_n)\leq M^2\,\Sigma_n(x)$, and Lemma~\ref{lem:clock-no-dom} gives $\Sigma_n(x)=o(A_n(x)^2)$, so Chebyshev's inequality and a union bound yield $\P(|R_n| > 4\delta A_n(x)) \leq \delta+o(1)$. Since $\delta > 0$ was arbitrary, $R_n/A_n(x) \to 0$ in probability.
	\end{proof}
	
	Theorem~\ref{thm:explosion}(i) follows: setting $\xi(v)=\sigma(v)-1$ gives $\E[\xi]=\mu-1>0$, and $\sup_n\E_x[S_n\mid\xi]=\infty$ implies $u_\infty(x)=\infty$ a.s.\ by Corollary~\ref{cor:RW-infinite}.

	\section{Critical regime}\label{sec:critical}
	
	In the supercritical regime, the walker profits from every step: when $A_n(o)\to\infty$, the mean drift $\E_o[D_n]=\E[\xi]\,A_n(o)$ grows without bound and explosion is immediate. At criticality~$\E[\xi]=0$, the drift vanishes and the payoff $S_n = W_n$ is pure fluctuation. Whether the walker can accumulate unbounded payoff depends on how fast the variance of the fluctuation grows.
	
	Recall $\Sigma_n(o)=\sum_{v\in V} g_n(o,v)^2$ from~\eqref{eq:gn-def}; we write $\Sigma_n$ when~$o$ is clear from context. Fix a vertex~$o\in V$.
	Since~$\E_o[S_n\mid\xi]=\E[\xi]\,A_n(o)+\sum_v g_n(o,v)(\xi(v)-\E[\xi])$ by~\eqref{eq:dn-wn-stuff}, the variance over~$\xi$ equals exactly $\var(\xi)\,\Sigma_n$. Moreover $\Sigma_n$ is increasing in~$n$ and increases to $\sum_v g(o,v)^2\in(0,\infty]$.
	A transient walk with~$\sum_v g(o,v)^2<\infty$ is called \textbf{doubly transient} at~$o$ (for example, on~$\Z^d$ with $d\geq 5$). Otherwise~$\Sigma_n\to\infty$ and we say the walk is \textbf{non-doubly-transient} at~$o$.
	
	The summability of the Green kernel organizes the phase structure. Supercritical explosion corresponds to $g(o,\cdot)\notin\ell^1$ (equivalently $A_n(o)\to\infty$). Short-clock stabilization corresponds to $g(o,\cdot)\in\ell^1$ when $\E[(\sigma-1)^+]<\infty$ (Proposition~\ref{prop:short-clock}). On graphs satisfying the uniform local trap condition (Definition~\ref{def:trap-family}), the critical dichotomy is whether $g(o,\cdot)\in\ell^2$ or not; this condition is automatic for bounded-degree graphs.
	
	In the non-doubly-transient case, the growth~$\Sigma_n\to\infty$ provides enough variance for a direct second-moment argument (Proposition~\ref{prop:critical}).
	In the doubly transient case~$\Sigma_n$ stays bounded, so the fluctuation variance saturates and a different mechanism is needed.
	We build a Green-function martingale and stop the walk in rare low-scenery traps to obtain a divergent expected payoff, then upgrade to quenched explosion via a variance bound (Proposition~\ref{prop:doubly-transient-really-general}).
	Theorem~\ref{thm:OS}(ii) combines Proposition~\ref{prop:critical} and Proposition~\ref{prop:doubly-transient-really-general}; Theorem~\ref{thm:explosion}(ii) then follows by the random walk representation (Corollary~\ref{cor:RW-infinite}).
	When~$\xi$ is symmetric, a convexity argument (Proposition~\ref{prop:convexity-reduction}) reduces the problem to the finite-variance case and removes the moment hypothesis entirely.
	
	\subsection{Explosion via variance growth}
	
	\begin{lemma}[Positive-part mean bound]\label{lem:positive-part}
		Let $(Z_i)_{i\in I}$ be a finite collection of independent mean-zero random variables
		with $|Z_i|\leq b_i$ a.s.\ for deterministic $b_i>0$.
		Let $B^2\coloneqq\sum_{i\in I}\var(Z_i)$.
		If $\max_{i}b_i^2\leq B^2$, then there exists a universal constant~$c_\star>0$ such that $\E\bigl[\bigl(\sum_{i\in I}Z_i\bigr)^{\!+}\bigr]\geq c_\star B$.
	\end{lemma}
	\begin{proof}
		If $B=0$, then each $Z_i=0$ a.s.\ and the bound is trivial; assume $B>0$.
		Write $Z\coloneqq\sum_i Z_i$.
		Since $|Z_i|\leq b_i$, we have $\E[Z_i^4]\leq b_i^2\var(Z_i)$.
		For independent mean-zero summands, we have
		\[
		\E[Z^4]=3B^4+\sum_i\bigl(\E[Z_i^4]-3\var(Z_i)^2\bigr)\leq 3B^4+\sum_i\E[Z_i^4]\leq 3B^4+\max_i b_i^2\cdot B^2\leq 4B^4\,.
		\]
		The Paley--Zygmund inequality applied to~$Z^2$ with~$\theta=1/2$ gives
		\[
		\P\bigl(Z^2\geq\tfrac{1}{2}B^2\bigr)
		\geq\frac{B^4}{4\E[Z^4]}
		\geq\frac{1}{16}\,
		\quad 
		\text{and so} 
		\quad 
		\E[|Z|]\geq\frac{B}{\sqrt{2}}\P\Bigl(|Z|\geq\frac{B}{\sqrt{2}}\Bigr)\geq\frac{B}{16\sqrt{2}}\,.
		\]
		Since $\E[Z]=0$, we have $\E[Z^+]=\tfrac{1}{2}\E[|Z|]\geq\frac{B}{32\sqrt{2}}$, so~$c_\star=\frac{1}{32\sqrt{2}}$ suffices.
	\end{proof}
	
	Intuitively, $\Sigma_n\to\infty$ forces the expected local-time profile to spread over many vertices, so no single vertex can carry a positive fraction of the~$\ell^2$ mass.
	
	\begin{lemma}[No-dominance]\label{lem:no-dominance}
		Let~$G=(V,E)$ be an infinite, locally finite, connected graph, and fix~$o\in V$.
		If~$\Sigma_n(o)\to\infty$, then $\max_{v}g_n(o,v)^2/\Sigma_n(o)\to 0$ as $n\to\infty$.
	\end{lemma}
	\begin{proof}
		By reversibility and the first-visit decomposition, $\max_v g_n(o,v)\leq g_n(o,o)$. If the walk is transient, then $g_n(o,o)\leq g(o,o)<\infty$ while $\Sigma_n\to\infty$. If the walk is recurrent, the argument in the proof of Lemma~\ref{lem:clock-no-dom} gives $g_n(o,v)/g_n(o,o)\to 1$ for each~$v$. For each fixed~$r$, the bound $g_n(o,v)\geq g_n(o,o)/2$ for $v\in B(o,r)$ and large~$n$ gives $\Sigma_n\geq|B(o,r)|\,g_n(o,o)^2/4$, so $\max_v g_n(o,v)^2/\Sigma_n\leq 4/|B(o,r)|\to 0$.
	\end{proof}
	\begin{lemma}[One-site sensitivity]\label{lem:sensitivity}
		Let~$G=(V,E)$ be an infinite, locally finite, connected graph, and fix
		vertices $o, v\in V$. Let~$\xi^{(v)}$ denote the scenery obtained by replacing~$\xi(v)$ with an independent copy~$\xi'(v)$.
		Then for each~$n\geq 1$,
		\begin{equation}\label{eq:sens-finite}
			\bigl|u_n(o;\xi)-u_n(o;\xi^{(v)})\bigr|
			\leq g_n(o,v)\,|\xi(v)-\xi'(v)|\,.
		\end{equation}
		If the walk is transient, then
		\[
		u_\infty(o;\xi)\leq u_\infty(o;\xi^{(v)})+g(o,v)\,|\xi(v)-\xi'(v)|\,,
		\]
		and the same bound with $\xi$ and $\xi^{(v)}$ interchanged. In particular, if both values are finite, then
		\begin{equation}
			\bigl|u_\infty(o;\xi)-u_\infty(o;\xi^{(v)})\bigr|
			\leq g(o,v)\,|\xi(v)-\xi'(v)|\,.
		\end{equation}
	\end{lemma}
	\begin{proof}
		Only~$\zeta(v)=\xi(v)/\deg(v)$ changes under resampling. For $\tau\leq n$, since $L_\tau(v)\leq L_n(v)$ pathwise,
		\[
		\left|\E_o[S_\tau(\xi)\mid\xi]-\E_o[S_\tau(\xi^{(v)})\mid\xi^{(v)}]\right|
		=\frac{\E_o[L_\tau(v)]}{\deg(v)}\,|\xi(v)-\xi'(v)|
		\leq g_n(o,v)\,|\xi(v)-\xi'(v)|\,.
		\]
		Since $|\sup\varphi-\sup\psi|\leq\sup|\varphi-\psi|$, \eqref{eq:sens-finite} follows from Theorem~\ref{thm:RW}. The infinite-horizon bounds follow by the same argument with $g(o,v)$ replacing $g_n(o,v)$, using $\E_o[L_\tau(v)]\leq\deg(v)\,g(o,v)$ for every bounded~$\tau$ and Corollary~\ref{cor:RW-infinite}.
	\end{proof}
	
	\begin{proposition}[Critical optimal stopping bounds]\label{prop:critical}
		Let~$G = (V,E)$ be an infinite, locally finite, connected graph, let~$(X_n)_{n\geq 0}$ be simple random walk on~$G$, and let $(\xi(v))_{v\in V}$ be i.i.d.\ and independent of the walk with
		$\E[\xi]=0$, $\var(\xi)>0$, and $\E[\xi^2]<\infty$.
		Let $u_n(o)\coloneqq\sup_{\tau\leq n}\E_o[S_\tau\mid\xi]$.
		
		\begin{enumerate}[label=\textup{(\alph*)}]
			\item \textup{(Variance bound)}\;
			For every~$o\in V$ and~$n\geq 1$,
			\begin{equation}\label{eq:ES-bound}
				\var(u_n(o))\leq\var(\xi)\,\Sigma_n(o)\,.
			\end{equation}
			\item \textup{(Explosion)}\;
			If~$\Sigma_n(o)\to\infty$, then there exist~$c_1,c_2>0$ depending only on~$\xi$ such that
			\begin{equation}\label{eq:PZ-bound}
				\P\bigl(u_n(o)\geq c_1\sqrt{\Sigma_n(o)}\bigr)\geq c_2\,,
			\end{equation}
			for all~$n$ sufficiently large.
			In particular, $\sup_n\E_o[S_n\mid\xi]=\infty$ almost surely.
		\end{enumerate}
	\end{proposition}
	
	The proof uses the no-dominance condition $\max_v g_n(o,v)^2/\Sigma_n(o)\to 0$, which Lemma~\ref{lem:no-dominance} guarantees whenever $\Sigma_n(o)\to\infty$.
	
	\begin{proof}
		Write~$R_n\coloneqq\E_o[S_n\mid\xi]=\sum_v g_n(o,v)\,\xi(v)$.
		Since the~$\xi(v)$ are independent and mean-zero, the stopping time~$\tau=n$ gives
		\begin{equation}\label{eq:un-geq-Rn+}
			u_n(o)\geq R_n^+\,.
		\end{equation}
		
		\medskip
		\textup{(a)}\;
		Lemma~\ref{lem:sensitivity} gives~$|u_n(o;\xi)-u_n(o;\xi^{(v)})|\leq g_n(o,v)|\xi(v)-\xi'(v)|$.
		The Efron--Stein inequality~\citep{EfronStein81} yields
		\[
		\var(u_n(o))\leq\tfrac{1}{2}\sum_v\E\bigl[(u_n(o;\xi)-u_n(o;\xi^{(v)}))^2\bigr]
		\leq\var(\xi)\,\Sigma_n\,.
		\]
		
		\medskip
		\textup{(b)}\;
		Fix~$M\geq 1$ and set $s_M \coloneqq \sqrt{\var\bigl(\xi\one_{\{|\xi|\leq M\}}\bigr)}$ and $\delta_M \coloneqq \sqrt{\E\bigl[\xi^2\one_{\{|\xi|>M\}}\bigr]}$. Since~$\E[\xi^2]<\infty$, dominated convergence gives~$s_M\to\sqrt{\var(\xi)}>0$ and~$\delta_M\to 0$ as~$M\to\infty$. Choose~$M$ large enough such that~$s_M>0$ and~$c_\star s_M>\delta_M$, where~$c_\star$ is from Lemma~\ref{lem:positive-part}.
		Define the truncated scenery
		\[
		\hat\zeta(v)\coloneqq\zeta(v)\one_{\{|\xi(v)|\leq M\}}-\E[\zeta(v)\one_{\{|\xi(v)|\leq M\}}]\,,
		\]
		which is mean-zero and bounded by~$2M/\deg(v)$. Set
		\[
		\hat R_n\coloneqq\sum_v g_n(o,v)\,\hat\zeta(v)\,\deg(v)=\sum_v\E_o[L_n(v)]\,\hat\zeta(v)\,.
		\]
		Since $|x^+-y^+|\leq|x-y|$, taking expectations gives $\E[R_n^+]\geq\E[\hat R_n^+]-\E[|R_n-\hat R_n|]$.
		Each~$\zeta(v)-\hat\zeta(v)$ is mean-zero, so $R_n-\hat R_n$ is mean-zero and Cauchy--Schwarz gives
		\begin{equation}\label{eq:trunc-error}
			\E[R_n^+]\geq\E[\hat R_n^+]-\sqrt{\var(R_n-\hat R_n)}\,.
		\end{equation}
		Since $\var\bigl(\zeta(v)\one_{\{|\xi(v)|>M\}}\bigr)\leq\delta_M^2/\deg(v)^2$, we get
		\begin{equation}\label{eq:trunc-var}
			\var(R_n-\hat R_n)
			\leq\delta_M^2\,\Sigma_n\,.
		\end{equation}
		The summands $\E_o[L_n(v)]\hat\zeta(v)$ are independent and mean-zero, with
		\[
		\bigl|\E_o[L_n(v)]\hat\zeta(v)\bigr|\leq 2M\,g_n(o,v)
		\qquad\text{and}\qquad
		\var\bigl(\hat\zeta(v)\bigr)=s_M^2/\deg(v)^2\,.
		\]
		Therefore $\var(\hat R_n)=s_M^2\,\Sigma_n$.
		By Lemma~\ref{lem:no-dominance}, $\max_v g_n(o,v)^2/\Sigma_n\to 0$, so for all~$n$ large enough,
		\[
		4M^2\max_vg_n(o,v)^{\!2}\leq s_M^2\,\Sigma_n=\var(\hat R_n)\,.
		\]
		Lemma~\ref{lem:positive-part} gives~$\E[\hat R_n^+]\geq c_\star\sqrt{\var(\hat R_n)}= c_\star s_M\sqrt{\Sigma_n}$.
		Combining with~\eqref{eq:trunc-error}, \eqref{eq:trunc-var}, and~\eqref{eq:un-geq-Rn+} yields
		\begin{equation}\label{eq:mean-lb}
			\E[u_n(o)]\geq\E[R_n^+]
			\geq\bigl(c_\star s_M-\delta_M\bigr)\sqrt{\Sigma_n}\,.
		\end{equation}
		Setting~$c_M\coloneqq c_\star s_M-\delta_M>0$, we obtain~$\E[u_n(o)]\geq c_M\sqrt{\Sigma_n}$ for all large~$n$. Since $\var(u_n(o))\leq\var(\xi)\,\Sigma_n\leq(\var(\xi)/c_M^2)\,\E[u_n(o)]^2$ by part~(a),
		\[
		\E[u_n(o)^2]\leq\Bigl(\frac{\var(\xi)}{c_M^2}+1\Bigr)\E[u_n(o)]^2\,.
		\]
		The Paley--Zygmund inequality with~$\theta=1/2$ yields
		\[
		\P\Bigl(u_n(o)\geq\tfrac{1}{2}c_M\sqrt{\Sigma_n}\Bigr)
		\geq\P\Bigl(u_n(o)\geq\tfrac{1}{2}\E[u_n(o)]\Bigr)
		\geq\frac{c_M^2}{4c_M^2+4\var(\xi)}\,.
		\]
		This gives~\eqref{eq:PZ-bound} with~$c_1\coloneqq c_M/2$ and~$c_2\coloneqq c_M^2/(4c_M^2+4\var(\xi))$.
		For the fixed-time conclusion, since $\E[(R_n^+)^2]\leq\E[R_n^2]=\var(\xi)\,\Sigma_n$ and $\E[R_n^+]\geq c_M\sqrt{\Sigma_n}$ by~\eqref{eq:mean-lb}, Paley--Zygmund gives $\P(R_n\geq\tfrac{1}{2}c_M\sqrt{\Sigma_n})\geq c_M^2/(4\var(\xi))$. Since $\Sigma_n\to\infty$, this yields $\P(\sup_n\E_o[S_n\mid\xi]=\infty)>0$, and Proposition~\ref{prop:01-law} upgrades to almost sure explosion.
	\end{proof}

	\subsection{Explosion via Green-function traps}
	When~$\Sigma_n\to\infty$, Proposition~\ref{prop:critical}(b) yields explosion.
	The remaining case is when the walk is transient and~$\sum_v g(o,v)^2<\infty$; then~$\Sigma_n$ stays bounded and a different approach is needed.
	
	\begin{definition}[Uniform local trap condition]\label{def:trap-family}
		For a finite set $C\subset V$ and $y\in V$, the \textbf{inverse-degree exit time} from~$C$ starting at~$y$ is
		\[
		\Theta_C(y)\coloneqq\sum_{v\in C}g_C(y,v)=\E_y\Bigl[\sum_{k=0}^{\tau_C-1}\frac{1}{\deg(X_k)}\Bigr]\,.
		\]
		We say that $G$ admits a \textbf{uniform local trap condition} if for every $L>0$ there exist $r_L$ and $M_L$, and for each $y\in V$ a finite connected set $C_L(y)$, such that
		\[
		y\in C_L(y)\subseteq B(y,r_L),\qquad |C_L(y)|\leq M_L,\qquad\Theta_{C_L(y)}(y)\geq L\,.
		\]
	\end{definition}
	
	The inverse-degree exit time $\Theta_C(y)$ is the local analogue of the clock~$A_n(x)$ from~\eqref{eq:gn-def}: the clock measures total inverse-degree time over the first~$n$ steps, while $\Theta_C$ measures it up to the exit from a local set. On a bounded-degree graph, balls satisfy the uniform local trap condition: taking $C_L(y)=B(y,R)$ with $R=\lceil Ld\rceil$ gives $\Theta_{B(y,R)}(y)\geq(R+1)/d\geq L$ and $|B(y,R)|\leq d^{R+1}$. More generally, sub-Gaussian displacement control (hypothesis~\ref{H3} in Proposition~\ref{prop:poly-growth}) also implies the uniform local trap condition on bounded-degree graphs. This hypothesis is equivalent, under standard regularity, to exit-time scaling $\E_x[\tau_{B(x,R)}]\asymp R^{d_w}$.
	
	\begin{proposition}[Doubly transient]\label{prop:doubly-transient-really-general}
		Let~$G = (V,E)$ be an infinite, locally finite, connected graph. Assume the simple random walk on $G$ is transient, that $\sum_{v\in V}g(o,v)^2<\infty$ for every $o\in V$,
		and that $G$ admits a uniform local trap condition.
		Let $(\xi(v))_{v\in V}$ be i.i.d.\ and independent of the walk with $\E[\xi]=0$, $\var(\xi)>0$, and $\E[\xi^2]<\infty$. Then~$\sup_{\tau}\E_o[S_\tau\mid\xi]=\infty$ almost surely for every~$o\in V$, where the supremum is over bounded stopping times.
	\end{proposition}
	
	On bounded-degree graphs the uniform local trap condition is automatic (see above), so the hypotheses reduce to transience and $\sum_{v\in V} g(o,v)^2<\infty$.
	
	\begin{proof}
		
		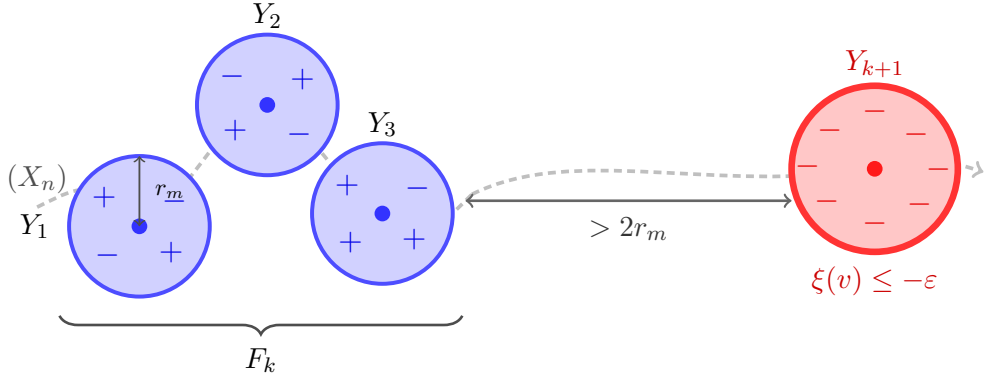
\begin{figure}[ht]
			\centering
			\begin{tikzpicture}[scale=0.85]
				\draw[gray!50, line width=1.4pt, densely dashed]
				(-0.6, -0.4) .. controls (0.2, 0.0) .. (1.0, -0.7)
				.. controls (2.0, 0.0) and (2.2, 1.0) .. (3.0, 1.2)
				.. controls (4.0, 0.6) and (4.2, -0.3) .. (4.8, -0.5)
				.. controls (5.5, -0.9) .. (6.2, -0.2);
				\draw[gray!50, line width=1.4pt, densely dashed]
				(6.2, -0.2) .. controls (7.8, 0.6) and (10.0, -0.2) .. (12.5, 0.2);
				\draw[->, gray!50, line width=1.4pt, densely dashed]
				(12.5, 0.2) .. controls (13.4, 0.4) .. (14.2, 0.15);
				\node[black!70, font=\normalsize] at (-0.6, 0.05) {$(X_n)$};
				
				\fill[blue!18] (1.0, -0.7) circle (1.1);
				\draw[blue!70, line width=1.5pt] (1.0, -0.7) circle (1.1);
				\node[blue!90!black, font=\large] at (0.45, -0.25) {$+$};
				\node[blue!90!black, font=\large] at (1.55, -0.3) {$-$};
				\node[blue!90!black, font=\large] at (0.5, -1.15) {$-$};
				\node[blue!90!black, font=\large] at (1.5, -1.1) {$+$};
				\fill[blue!80] (1.0, -0.7) circle (3.5pt);
				\node[black, font=\normalsize, anchor=east] at (-0.25, -0.7) {$Y_1$};
				
				\fill[blue!18] (3.0, 1.2) circle (1.1);
				\draw[blue!70, line width=1.5pt] (3.0, 1.2) circle (1.1);
				\node[blue!90!black, font=\large] at (2.45, 1.65) {$-$};
				\node[blue!90!black, font=\large] at (3.55, 1.6) {$+$};
				\node[blue!90!black, font=\large] at (2.5, 0.8) {$+$};
				\node[blue!90!black, font=\large] at (3.5, 0.75) {$-$};
				\fill[blue!80] (3.0, 1.2) circle (3.5pt);
				\node[black, font=\normalsize] at (3.0, 2.6) {$Y_2$};
				
				\fill[blue!18] (4.8, -0.5) circle (1.1);
				\draw[blue!70, line width=1.5pt] (4.8, -0.5) circle (1.1);
				\node[blue!90!black, font=\large] at (4.25, -0.05) {$+$};
				\node[blue!90!black, font=\large] at (5.35, -0.1) {$-$};
				\node[blue!90!black, font=\large] at (4.3, -0.95) {$+$};
				\node[blue!90!black, font=\large] at (5.3, -0.9) {$+$};
				\fill[blue!80] (4.8, -0.5) circle (3.5pt);
				\node[black, font=\normalsize] at (4.8, 0.9) {$Y_3$};
				
				\draw[<->, black!70, line width=0.7pt] (1.0, -0.7) -- (1.0, 0.4)
				node[midway, right=2pt, font=\footnotesize, black] {$r_m$};
				
				\draw[decorate, decoration={brace, amplitude=7pt, mirror},
				black!70, line width=1pt]
				(-0.25, -2.1) -- (6.05, -2.1)
				node[midway, below=8pt, font=\normalsize, black] {$F_k$};
				
				\fill[red!22] (12.5, 0.2) circle (1.3);
				\draw[red!80, line width=2.5pt] (12.5, 0.2) circle (1.3);
				\node[red!90!black, font=\large] at (11.8, 0.8) {$-$};
				\node[red!90!black, font=\large] at (13.15, 0.75) {$-$};
				\node[red!90!black, font=\large] at (11.75, -0.3) {$-$};
				\node[red!90!black, font=\large] at (13.2, -0.35) {$-$};
				\node[red!90!black, font=\large] at (12.5, 1.1) {$-$};
				\node[red!90!black, font=\large] at (12.5, -0.65) {$-$};
				\node[red!90!black, font=\large] at (11.45, 0.25) {$-$};
				\node[red!90!black, font=\large] at (13.5, 0.2) {$-$};
				\fill[red!90] (12.5, 0.2) circle (3.5pt);
				\node[red!80!black, font=\normalsize] at (12.5, 1.85) {$Y_{k+1}$};
				
				\draw[<->, black!60, line width=1pt] (6.1, -0.3) -- (11.2, -0.3)
				node[midway, below=2pt, font=\normalsize, black!80] {$> 2r_m$};
				
				\node[red!80!black, font=\normalsize] at (12.5, -1.55)
				{$\xi(v) \leq {-}\varepsilon$};
			\end{tikzpicture}
			\caption{Trap selection in the proof of Proposition~\ref{prop:doubly-transient-really-general}. The filled set~$F_k=\bigcup_{j\leq k}C_m(Y_j)$ consists of previously used trap sets (blue); the new center~$Y_{k+1}\in\mathcal{A}_k$ lies far from~$F_k$. Signs indicate~$\xi(v)$ at each vertex. In the red trap set every vertex has~$\xi(v)\leq -\varepsilon$, giving contribution~$\leq -8m$ to~$h(Y_{k+1})$.}
			\label{fig:ball-selection}
		\end{figure}
		
		The proof has two steps. First, we show $\E[u_\infty(o)]=\infty$ by constructing a.s.-finite stopping times $\tau_m$ with $\E[S_{\tau_m}]\to\infty$: the walker runs until it reaches a region where all scenery values are below~$-\varepsilon$, producing a large negative value of~$\sum_v g(X_{\tau_m},v)\,\xi(v)$ that translates into a large positive payoff. Second, we show that $u_\infty(o)<\infty$ a.s.\ would force $\var(u_\infty(o))<\infty$ by an Efron--Stein argument applied to truncations, contradicting $\E[u_\infty(o)]=\infty$.
		
		Fix $o\in V$.
		The function $h\colon V\to\R$ given by
		\[
		h(x)\coloneqq\sum_v g(x,v)\,\xi(v)\,.
		\]
		is well-defined in~$L^2$ for every~$x\in V$ by hypothesis.
		It is easy to see that
		$h(x)=\zeta(x)+\frac{1}{\deg(x)}\sum_{y\sim x}h(y)$, so~$h(X_n)+S_n$ is a~$\P_o$-martingale.
		
		\emph{Step~1: Annealed explosion.}
		Since $\var(\xi)>0$, choose $\varepsilon>0$ with $p_0\coloneqq\P(\xi(v)\leq -\varepsilon)\in(0,1)$. Fix $m\geq 1$ and set $\Lambda_m\coloneqq 8m/\varepsilon$. By Definition~\ref{def:trap-family}, choose $r_m\coloneqq r_{\Lambda_m}$ and $M_m\coloneqq M_{\Lambda_m}$, and for each $y\in V$ a set $C_m(y)\coloneqq C_{\Lambda_m}(y)$ satisfying
		\[
		y\in C_m(y)\subseteq B(y,r_m),\qquad |C_m(y)|\leq M_m,\qquad\Theta_{C_m(y)}(y)\geq 8m/\varepsilon\,.
		\]
		For $y\in V$, let $A_y$ denote the event that $\xi(v)\leq -\varepsilon$ for all $v\in C_m(y)$; on~$A_y$, every $v\in C_m(y)$ satisfies $\xi(v)\leq-\varepsilon$, and $\P(A_y)\geq p_0^{M_m}$.
		
		We select centers $Y_1,Y_2,\ldots$ along the walk path whose trap sets are pairwise disjoint and have controlled Green interaction.
		Set $T_1\coloneqq 0$ and $Y_1\coloneqq o$. Set $F_0\coloneqq\varnothing$ and $F_k\coloneqq\bigcup_{i=1}^k C_m(Y_i)$ for $k\geq 1$.
		Given~$F_k$, define the admissible set
		\[
		\mathcal{A}_k\coloneqq\Bigl\{z\in V:\mathrm{dist}(z,F_k)>2r_m\ \text{and}\ \sum_{v\in F_k}g(v,z)\leq 1\Bigr\}\,.
		\]
		Let $T_{k+1}\coloneqq\inf\{n\geq T_k:X_n\in\mathcal{A}_k\}$ and set $Y_{k+1}\coloneqq X_{T_{k+1}}$.
		The set $\mathcal{A}_k$ is co-finite: by Cauchy--Schwarz, $\bigl(\sum_{v\in F_k}g(v,z)\bigr)^2\leq|F_k|\sum_{v\in F_k}g(v,z)^2$, and summing over~$z$ gives $\sum_z\bigl(\sum_{v\in F_k} g(v,z)\bigr)^2\leq|F_k|\sum_{v\in F_k}\sum_z g(v,z)^2<\infty$ by hypothesis, so only finitely many~$z$ violate~$\sum_v g(v,z)\leq 1$. Together with the finitely many vertices within distance~$2r_m$ of~$F_k$, this gives $T_{k+1}<\infty$ a.s.
		By the distance condition the trap sets are pairwise disjoint, and by symmetry $g(Y_{k+1},v)=g(v,Y_{k+1})$, the admissibility condition gives directly
		\begin{equation}\label{eq:interaction}
			\sum_{v\in F_k}g(Y_{k+1},v)\leq 1\,.
		\end{equation}
		Since the trap sets are disjoint and the scenery i.i.d., the events $A_{Y_1},A_{Y_2},\ldots$ are conditionally independent given the walk, each with probability at least~$p_0^{M_m}$.
		Choose $\ell$ with $(1-p_0^{M_m})^{\ell-1}\leq e^{-1}$, and set
		\[
		J_m\coloneqq\min\bigl(\{1\leq i\leq\ell:A_{Y_i}\text{ occurs}\}\cup\{\ell\}\bigr),\qquad\tau_m\coloneqq T_{J_m}\,.
		\]
		Write $\mathcal{G}_i\coloneqq\sigma(\xi)\vee\mathcal{F}_{T_i}$ for the stage filtration. Since $A_{Y_i}\in\mathcal{G}_i$, the index $J_m$ is a stopping time in the stage filtration, and $\tau_m$ is an a.s.-finite stopping time for the walk. It is not deterministically bounded, so in place of Corollary~\ref{cor:RW-infinite} we use the continuation inequality
		\begin{equation}\label{eq:continuation}
			u_\infty(x)\geq\E_x\bigl[S_{\tau_D}+u_\infty(X_{\tau_D})\mid\xi\bigr]\,,
		\end{equation}
		valid for every finite connected $D\ni x$. To prove~\eqref{eq:continuation}, iterate the Bellman equation~\eqref{eq:vK-bellman} for $v_K$ up to time $\tau_D$ on every proper finite $K\supseteq D\cup\partial D$, which gives $v_K(x)\geq\E_x[S_{\tau_D}+v_K(X_{\tau_D})\mid\xi]$. Let $K\uparrow V$: for every bounded stopping time $\tau\leq N$, taking $K\supseteq B(x,N)$ gives $u_N(x)\leq v_K(x)\leq u_\infty(x)$, so $v_K\uparrow u_\infty$ pointwise, uniformly on the finite set $\partial D$; combined with $\E_x[\tau_D]<\infty$, monotone convergence yields~\eqref{eq:continuation}.

		Let $D_i$ denote the connected component of $V\setminus\mathcal{A}_i$ containing $Y_i$. Since $V\setminus\mathcal{A}_i$ is finite, so is $D_i$, and $\partial D_i\subseteq\mathcal{A}_i$ by the component definition, giving $T_{i+1}-T_i=\tau_{D_i}\circ\theta_{T_i}$. Applying~\eqref{eq:continuation} at $x=Y_i$ with $D=D_i$ and invoking the strong Markov property at $T_i$ yields, for every $1\leq i<\ell$,
		\[
		u_\infty(Y_i)\geq\E_o\bigl[(S_{T_{i+1}}-S_{T_i})+u_\infty(Y_{i+1})\,\big|\,\mathcal{F}_{T_i}\bigr]\,.
		\]
		Setting $V_i\coloneqq\E_o[S_{\tau_m}-S_{T_i}\mid\mathcal{F}_{T_i}]$ on $\{J_m\geq i\}$, backward induction from $V_\ell=0$ yields $V_i\leq u_\infty(Y_i)$ for every $1\leq i\leq\ell$: on $\{J_m=i\}$ we have $\tau_m=T_i$ and $V_i=0\leq u_\infty(Y_i)$; on $\{J_m>i\}$ the tower property and the inductive hypothesis at stage $i+1$ give $V_i=\E_o[(S_{T_{i+1}}-S_{T_i})+V_{i+1}\mid\mathcal{F}_{T_i}]\leq u_\infty(Y_i)$. In particular,
		\begin{equation}\label{eq:stage-bound}
			u_\infty(o)\geq V_1=\E_o[S_{\tau_m}\mid\xi]\,.
		\end{equation}

		To bound $\E_o[S_{\tau_m}\mid\xi]$, we pass to a martingale indexed by the stage number. Applying optional stopping to the walk martingale $h(X_n)+S_n$ at the finite-expectation exit time $\tau_{D_i}$ shows that $M_i\coloneqq h(Y_i)+S_{T_i}$ is a martingale in the stage filtration under $\P_o(\cdot\mid\xi)$. Since $J_m\leq\ell$ is bounded in the stage index, optional stopping at $J_m$ gives $\E_o[M_{J_m}\mid\xi]=M_1=h(o)$. Taking expectations and using $\E[h(o)]=\sum_v g(o,v)\,\E[\xi(v)]=0$ yields
		\begin{equation}\label{eq:payoff-identity}
			\E[S_{\tau_m}]=-\E[h(Y_{J_m})]\,.
		\end{equation}

		We bound $\E[h(Y_{J_m})]$ by conditioning on the walk. If $J_m=i<\ell$, then $A_{Y_i}$ occurs. On $A_{Y_i}$, every $v\in C_m(Y_i)$ satisfies $\xi(v)\leq-\varepsilon$, so using $g(Y_i,v)\geq g_{C_m(Y_i)}(Y_i,v)$ and the nonpositivity of the coefficients,
		\[
		\sum_{v\in C_m(Y_i)}g(Y_i,v)\,\xi(v)\leq\sum_{v\in C_m(Y_i)}g_{C_m(Y_i)}(Y_i,v)\,\xi(v)\leq -\varepsilon\,\Theta_{C_m(Y_i)}(Y_i)\leq -8m\,.
		\]
		Vertices outside $\bigcup_{j\leq i}C_m(Y_j)$ are independent of the conditioning and contribute zero in conditional expectation given~$X$. For vertices in earlier trap sets $C_m(Y_j)$ with $j<i$, conditioning on~$A_{Y_j}^c$ biases $\E[\xi(v)]$ by at most $\E[|\xi|]/(1-p_0)$, and by~\eqref{eq:interaction} the total bias from earlier trap sets is at most $B_0\coloneqq\E[|\xi|]/(1-p_0)$. This yields
		\[
		\E[h(Y_i)\mid J_m=i,X]\leq -8m+B_0\,.
		\]
		If $J_m=\ell$, then no $A_{Y_j}$ occurs for $j<\ell$; the trap set $C_m(Y_\ell)$ is unbiased under this conditioning, while the earlier trap sets contribute at most $B_0$, so
		\[
		\E[h(Y_\ell)\mid J_m=\ell,X]\leq B_0\,.
		\]
		By conditional independence of the $A_{Y_i}$ given~$X$,
		\[
		\P(J_m=\ell\mid X)=\P\!\Bigl(\bigcap_{j<\ell}A_{Y_j}^c\,\Big|\,X\Bigr)\leq(1-p_0^{M_m})^{\ell-1}\leq e^{-1}\,.
		\]
		Combining,
		\[
		\E[h(Y_{J_m})\mid X]\leq(-8m+B_0)(1-e^{-1})+B_0\,e^{-1}=-8m(1-e^{-1})+B_0\,.
		\]
		Taking walk expectations and applying~\eqref{eq:payoff-identity} gives $\E[S_{\tau_m}]\geq 8m(1-e^{-1})-B_0$, and combined with~\eqref{eq:stage-bound} this yields $\E[u_\infty(o)]\geq 8m(1-e^{-1})-B_0\to\infty$ as $m\to\infty$.
		
		\medskip
		
		\emph{Step~2: Variance bound and conclusion.}
		Suppose for contradiction that $u_\infty(o)<\infty$ a.s. Set $V_0\coloneqq\var(\xi)\sum_v g(o,v)^2$, which is finite by hypothesis.
		For~$K>0$, set~$f_K\coloneqq u_\infty(o)\wedge K$. The same Efron--Stein argument as in Proposition~\ref{prop:critical}(a)---using $|a\wedge K-b\wedge K|\leq|a-b|$ and $g_n(o,v)\leq g(o,v)$---gives $\var(u_n(o)\wedge K)\leq V_0$ for every~$n$, where $u_n(o)\coloneqq\sup_{\tau\leq n}\E_o[S_\tau\mid\xi]$. Since $u_n(o)\wedge K\uparrow f_K$, dominated convergence gives $\var(f_K)\leq V_0$ for all~$K$.
		Since~$u_\infty(o)<\infty$ a.s., $f_K\uparrow u_\infty(o)$ a.s.\ and~$\E[f_K]\uparrow\E[u_\infty(o)]=\infty$ by monotone convergence.
		Hence~$(f_K-\E[f_K])^2\to\infty$ a.s., and Fatou's lemma gives~$\liminf_{K\to\infty}\var(f_K)=\infty$, contradicting~$\var(f_K)\leq V_0$.
		It follows that~$\P(u_\infty(o)=\infty)>0$, and Proposition~\ref{prop:01-law} gives~$u_\infty(o)=\infty$ almost surely.
	\end{proof}
	
	\subsection{Explosion under symmetry}
	
	The following proposition shows that, in every qualitative explosion result of the type above, the finite-variance hypothesis can be replaced by symmetry.  The key observation is that the optimal stopping value, viewed as a function of the scenery~$\xi$, is convex.
	
	\begin{proposition}[Convexity reduction for symmetric laws]\label{prop:convexity-reduction}
		Let~$G = (V,E)$ be an infinite, locally finite, connected graph, and let $(\xi(v))_{v\in V}$ be i.i.d.\ and independent of the walk with
		$\E[\xi]=0$, $\xi\not\equiv 0$, and $\xi$ symmetric about~$0$.
		If for every i.i.d.\ scenery on~$G$ with mean~$0$, positive variance, and finite second moment, $\sup_\tau\E_o[S_\tau\mid\xi]=\infty$ a.s., then the same conclusion holds for~$\xi$. The same reduction holds with~$\sup_\tau$ replaced by~$\sup_n$.
	\end{proposition}
	
	\begin{proof}
		The strategy is to construct a bounded scenery~$Y$ with mean zero and positive variance, express~$Y$ as the midpoint of~$\xi$ and a modified copy~$\xi^\dagger\stackrel{d}{=}\xi$, and exploit convexity.
		
		\medskip
		\emph{Step~1: Construction of~$\xi^\dagger$.}\;
		Since~$\xi\not\equiv 0$, there exists~$M>0$ with~$\P(0<|\xi(v)|\leq M)>0$.  Define
		\[
		\xi^\dagger(v)\coloneqq \xi(v)\one_{\{|\xi(v)|\leq M\}}-\xi(v)\one_{\{|\xi(v)|>M\}}\,.
		\]
		Because~$\xi(v)$ is symmetric and the set~$\{|x|>M\}$ is invariant under~$x\mapsto -x$, splitting on~$\{|\xi(v)|\leq M\}$ and its complement shows that~$\xi^\dagger(v)\stackrel{d}{=}\xi(v)$.  Since each~$\xi^\dagger(v)$ depends only on~$\xi(v)$ and the~$\xi(v)$ are i.i.d., the field~$(\xi^\dagger(v))_{v\in V}$ is i.i.d.\ with the same marginals.
		
		\medskip
		\emph{Step~2: Bounded truncation.}\;
		Set~$Y(v)\coloneqq(\xi(v)+\xi^\dagger(v))/2=\xi(v)\one_{\{|\xi(v)|\leq M\}}$.  Then~$|Y|\leq M$, the variable~$Y(v)$ is symmetric (hence~$\E[Y]=0$), and~$\var(Y)>0$ since~$Y$ is not almost surely zero.  By hypothesis,~$\sup_\tau\E_o[S_\tau(Y)\mid Y]=\infty$ almost surely, where $S_n(Y)\coloneqq\sum_{k<n}Y(X_k)/\deg(X_k)$.
		
		\medskip
		\emph{Step~3: Convexity.}\;
		By Theorem~\ref{thm:nested-vol} and Corollary~\ref{cor:RW-infinite}, writing $u_\infty(o;\xi)=\sup_\tau\E_o[S_\tau\mid\xi]$,
		\[
		F_o(\xi)\coloneqq u_\infty(o;\xi)=\max\biggl\{0,\;\sup_{\substack{C\ni o\\ C\textup{ finite connected}}}\sum_{v\in C}g_C(v)\,\xi(v)\biggr\}\,
		\]
		is a supremum of linear functionals of~$\xi$ (together with the zero functional), hence convex.
		Since~$Y=(\xi+\xi^\dagger)/2$ pointwise, convexity gives
		\[
		\infty=F_o(Y)\leq\frac{F_o(\xi)+F_o(\xi^\dagger)}{2}\qquad\text{a.s.}
		\]
		Therefore~$\max\bigl(F_o(\xi),\,F_o(\xi^\dagger)\bigr)=\infty$ almost surely.  Since~$\xi^\dagger\stackrel{d}{=}\xi$, the two probabilities are equal, giving~$\P(F_o(\xi)=\infty)\geq 1/2>0$.  Proposition~\ref{prop:01-law} then gives~$\sup_\tau\E_o[S_\tau\mid\xi]=\infty$ almost surely.
		
		The same argument applies to the fixed-time functional $\sup_n\sum_v g_n(o,v)\,\xi(v)$, which is likewise a supremum of linear functionals and hence convex.  The second conclusion of Proposition~\ref{prop:01-law} replaces the first in the upgrade step.
	\end{proof}
	
	Applied to Proposition~\ref{prop:critical}(b), the convexity reduction gives: if~$\Sigma_n(o)\to\infty$ and~$\xi$ is symmetric and non-degenerate, then~$\sup_n\E_o[S_n\mid\xi]=\infty$ a.s.\ with no moment condition beyond~$\E[\xi]=0$.  Applied to Proposition~\ref{prop:doubly-transient-really-general}, it gives~$\sup_\tau\E_o[S_\tau\mid\xi]=\infty$ a.s.\ in the doubly transient case on graphs admitting a uniform local trap condition.  On bounded-degree graphs, the trap condition is automatic and the two cases are exhaustive, yielding Theorem~\ref{thm:OS}(ii).  Theorem~\ref{thm:explosion}(ii) follows: setting $\xi(v)=\sigma(v)-1$ gives $\mu-1=0$, and $\sup_{\tau}\E_x[S_\tau\mid\xi]=\infty$ implies $u_\infty(x)=\infty$ a.s.\ by Corollary~\ref{cor:RW-infinite}.
	
	\section{Subcritical regime}\label{sec:subcritical}
	
	When $\E[\xi]<0$ on a bounded-degree graph, the negative drift $D_n\leq\E[\xi]\,n/d$ penalizes long runs and the optimal strategy is to stop early. (Here and below the mean $\E[\xi]$ is finite; the case $\E[\xi]=-\infty$ reduces to finite mean by lower truncation, as in the proof of Proposition~\ref{prop:subcritical}.) Finiteness of the optimal stopping value requires controlling the fluctuation $W_n$ against this linear penalty. We first isolate a degree-free criterion for graphs with finite inverse-degree time (Proposition~\ref{prop:short-clock}). We then record the local time bounds that enter both main proofs. Proposition~\ref{prop:subcritical} proves finite moments of $\sup_n S_n$ under $p>1+2/d_s$ via tail bounds. Proposition~\ref{prop:poly-growth} gives $\sup_\tau\E_o[S_\tau\mid\xi]<\infty$ a.s.\ under the weaker condition $p>\max(2d_f/(d_w\,d_s),\,1)$ on graphs with sub-Gaussian displacement, using scenery truncation.
	
	\subsection{Stabilization under a finite inverse-degree clock}
	
	Before the moment bounds, we isolate a degree-free stabilization criterion that applies when the walk has only finite effective time.
	
	\begin{proposition}[Short-clock stabilization]\label{prop:short-clock}
		Let $G=(V,E)$ be infinite, locally finite, and connected. If for some~$o \in V$ the total inverse-degree time $\sum_{v\in V}g(o,v)$ is finite, and if $(\xi(v))_{v\in V}$ are i.i.d.\ and independent of the walk with $\E[\xi^+]<\infty$, then $\E[\sup_\tau\E_o[S_\tau\mid\xi]] < \infty$, where the supremum is over bounded stopping times.
	\end{proposition}
	\begin{proof}
		For every bounded stopping time~$\tau$,
		\[
		\E_o[S_\tau\mid\xi]
		=\sum_{v\in V}\frac{\E_o[L_\tau(v)]}{\deg(v)}\,\xi(v)
		\leq\sum_{v\in V}\frac{\E_o[L_\tau(v)]}{\deg(v)}\,\xi(v)^+
		\leq\sum_{v\in V}g(o,v)\,\xi(v)^+\,.
		\]
		The first inequality drops the negative terms; the second uses $\E_o[L_\tau(v)]\leq\deg(v)\,g(o,v)$.
		Taking the supremum over bounded~$\tau$ and then expectations gives
		$\E[\sup_\tau\E_o[S_\tau\mid\xi]]\leq\E[\xi^+]\sum_v g(o,v)<\infty$.
	\end{proof}
	
	The counterexample in Lemma~\ref{ex:counterexample} has $\sum_{v\in V} g(o,v)<\infty$, so the optimal stopping value is finite regardless of~$\E[\xi]$. On the other hand, if $A_n(x)\to\infty$, then the phase transition in~$\E[\xi]$ is nontrivial and the subcritical direction requires the negative drift $\E[\xi]\,A_n$ to dominate the fluctuation.
	
	\subsection{Local time bounds}
	
	The proofs of Propositions~\ref{prop:subcritical} and~\ref{prop:poly-growth} use the following local time estimate.
	
	\begin{lemma}\label{lem:local-time}
		Let~$G = (V,E)$ be an infinite, locally finite, connected graph on which the simple random walk $(X_n)_{n\geq 0}$ satisfies the \textbf{spectral dimension bound}
		\begin{equation}\label{eq:return-bound}
			\sup_{x\in V}\P_x(X_n=x)\leq A\,n^{-d_s/2}\,,
		\end{equation}
		for all $n\geq 1$, for some $d_s>0$ and $A<\infty$.
		Then for every $p \geq 1$ and all $x \in V$ and $n \geq 1$,
		\[
		\sum_{v\in V} \E_x[L_n(v)^p] \leq C_p\begin{cases}
			n^{1+(p-1)(1-d_s/2)} & \text{if } d_s<2,\\[3pt]
			n\,(1+\log n)^{p-1} & \text{if } d_s=2,\\[3pt]
			n & \text{if } d_s>2\,.
		\end{cases}\,
		\]
		Here $C_p<\infty$ depends only on~$p$, $d_s$, and~$A$.
	\end{lemma}
	
	\begin{proof}
		By H\"older interpolation, we may assume that~$p$ is a positive integer.
		Since $\P_x(X_0=x)=1$, we may and do take $A\geq 1$.
		Expanding the $p$-th moment,
		\[
		\sum_v \E_x[L_n(v)^p]
		= \sum_{k_1,\ldots,k_p=0}^{n-1}
		\sum_v \P_x(X_{k_1} = \cdots = X_{k_p} = v)\,.
		\]
		For every ordered tuple $0 \leq t_1 \leq \cdots \leq t_p \leq n-1$,
		the Markov property gives
		\[
		\sum_v \P_x(X_{t_1} = \cdots = X_{t_p} = v)
		= \sum_v \P_x(X_{t_1}=v)\prod_{i=2}^{p} \P_v(X_{t_i - t_{i-1}}=v)
		\leq \prod_{i=2}^{p} \frac{A}{(t_i - t_{i-1})^{d_s/2} \vee 1}\,,
		\]
		using $\sum_v \P_x(X_{t_1}=v) = 1$ and~\eqref{eq:return-bound}.
		Since at most $p!$ tuples $(k_1,\ldots,k_p)$ map to each ordered tuple,
		and $t_1$ ranges over $n$ values, setting $g_i = t_i - t_{i-1}$
		for $i \geq 2$ and enlarging the domain of summation gives
		\[
		\sum_v \E_x[L_n(v)^p]
		\leq p! A^{p-1} n
		\biggl(\sum_{g=0}^{n-1} \frac{1}{g^{d_s/2} \vee 1}\biggr)^{p-1}\,.
		\]
		For $d_s<2$ the sum is $O(n^{1-d_s/2})$ by comparison with an integral; for $d_s=2$ it is $O(\log n)$; for $d_s>2$ it is $O(1)$.
		\qedhere
	\end{proof}
	
	On every graph with degree bounded by~$d$, by~\citet[Example~5.14]{Grigoryan},
	we have
	\begin{equation}
		\label{eq:heat-kernel-upper-bound}
		\sup_{x,y} \P_x(X_n = y) \leq C(d) n^{-1/2}\,.
	\end{equation}
	In particular, Lemma~\ref{lem:local-time} gives $\sum_v\E_x[L_n(v)^p]\leq C_p\,n^{(p+1)/2}$ on such graphs.
	
	\subsection{Stabilization via tail bounds}
	
	The proof of Proposition~\ref{prop:subcritical} uses the following tail inequality for sums of independent random variables.
	
	\begin{lemma}[Tail inequality for independent sums]\label{lem:fuk-nagaev}
		Let $(Y_i)_{i\in I}$ be a finite collection of independent mean-zero random variables with $M_p\coloneqq\sum_{i\in I}\E[|Y_i|^p]<\infty$ for some $p\geq 1$.
		\begin{enumerate}[label=\textup{(\alph*)}]
			\item \textup{(Polynomial tail, $p\in[1,2]$)}\; If $p\in[1,2]$, then for all $t>0$,
			\begin{equation}\label{eq:poly-tail}
				\P\biggl(\Bigl|\sum_{i\in I}Y_i\Bigr|\geq t\biggr)\leq C_p\,\frac{M_p}{t^p}\,.
			\end{equation}
			Here $C_p<\infty$ depends only on~$p$.
			\item \textup{(Polynomial $+$ Gaussian tail)}\; If $p\geq 2$ and $B^2\coloneqq\sum_{i\in I}\E[Y_i^2]<\infty$, then for all $t>0$,
			\begin{equation}\label{eq:fuk-nagaev}
				\P\biggl(\Bigl|\sum_{i\in I}Y_i\Bigr|\geq t\biggr)\leq C_p\,\frac{M_p}{t^p}+2\exp\biggl(-c\,\frac{t^2}{B^2}\biggr)\,.
			\end{equation}
			Here $c>0$ and $C_p<\infty$ depend only on~$p$.
			\item \textup{(Bernstein)}\; If $|Y_i|\leq M$ a.s.\ for each~$i$ and $B^2\coloneqq\sum_{i\in I}\E[Y_i^2]<\infty$, then for all $t>0$,
			\begin{equation}\label{eq:bernstein}
				\P\biggl(\Bigl|\sum_{i\in I}Y_i\Bigr|\geq t\biggr)\leq 2\exp\biggl(-\frac{t^2/2}{B^2+Mt/3}\biggr)\,.
			\end{equation}
		\end{enumerate}
	\end{lemma}
	\begin{proof}
		Part~(a) is the von Bahr--Esseen inequality~\citep{vonBahrEsseen65}. Part~(b) is due to \citet{FukNagaev71} and \citet{Nagaev79}. Part~(c) is Bernstein's inequality~\citep[Chapter~III, Theorem~17]{Petrov}.
	\end{proof}
	
	The next two lemmas collect the reduction steps common to the proofs of Propositions~\ref{prop:subcritical} and~\ref{prop:poly-growth}. Recall the decomposition $S_n = D_n + W_n$ from~\eqref{eq:dn-wn-stuff}. Throughout this section, $\P_x$ and $\E_x$ denote the joint probability and expectation over the walk started at~$x$ and the i.i.d.\ masses.
	
	\begin{lemma}[Dyadic reduction]\label{lem:dyadic}
		Let~$G = (V,E)$ be an infinite, locally finite, connected graph with degree bounded by~$d$, and let $(\xi(v))_{v\in V}$ be i.i.d.\ and independent of the walk with $\E[\xi]\in(-\infty,0)$ and $\E[|\xi|^p]<\infty$ for some $p\geq 1$. For $k\geq 0$ let
		\[
		Y_k\coloneqq\Bigl(\max_{2^k\leq n<2^{k+1}}W_n-\frac{|\E[\xi]|}{d}\,2^k\Bigr)^{\!+}\,.
		\]
		Then for every $q\geq 1$ and every $x\in V$,
		\begin{equation}\label{eq:dyadic-bound}
			\E_x\bigl[(\sup_n S_n)^q\bigr]\leq\sum_{k\geq 0}\E_x\bigl[Y_k^q\bigr]\,.
		\end{equation}
	\end{lemma}
	\begin{proof}
		Bounded degree gives $D_n\leq \E[\xi]\,n/d$ pathwise. Set $Z\coloneqq\sup_{n\geq 1}(W_n+\E[\xi]\,n/d)^+$. Since $S_n=D_n+W_n\leq W_n+\E[\xi]\,n/d$, we have $\sup_n S_n\leq Z$. For every $n\in[2^k,2^{k+1})$, the bound $(W_n+\E[\xi]\,n/d)^+\leq(W_n+\E[\xi]\,2^k/d)^+\leq Y_k$ gives $Z\leq\sup_k Y_k$. Since $(\sup_k a_k)^q\leq\sum_k a_k^q$ for non-negative terms, the bound follows.
	\end{proof}
	
	\begin{lemma}[Good-walk event]\label{lem:good-walk}
		Let~$G = (V,E)$ be an infinite, locally finite, connected graph with degree bounded by~$d$ and satisfying the bound~\eqref{eq:return-bound} for some $d_s>0$ and $A<\infty$. Let $(\xi(v))_{v\in V}$ be i.i.d.\ and independent of the walk with $\E[\xi]\in(-\infty,0)$ and $\E[|\xi|^p]<\infty$ for some $p\geq 1$, and let $Y_k\coloneqq(\max_{2^k\leq n<2^{k+1}}W_n-|\E[\xi]|\,2^k/d)^+$. Let $\alpha\coloneqq(1-d_s/2)^+$ and fix $\delta\in(0,d_s/2\wedge 1)$. For $k\geq 0$, let
		\begin{equation}
			\mathcal{A}_k\coloneqq\bigl\{\max_v L_N(v)\leq N^{\alpha+\delta}\bigr\},\qquad N\coloneqq 2^{k+1}\,.
		\end{equation}
		Then:
		\begin{enumerate}[label=\textup{(\alph*)}]
			\item For every integer $r\geq 1$,
			\[
			\P_x(\mathcal{A}_k^c)\leq C_r\begin{cases}
				N^{d_s/2-r\delta} & \text{if } d_s<2,\\[3pt]
				N^{1-r\delta}(1{+}\log N)^{r-1} & \text{if } d_s=2,\\[3pt]
				N^{1-r\delta} & \text{if } d_s>2\,.
			\end{cases}\,
			\]
			\item $\E_x[Y_k^p]\leq C\,N^p$ uniformly in~$x$ and~$k$.
			\item For every $q\in[1,p)$ and $r$ large enough, $\E_x[Y_k^q\one_{\mathcal{A}_k^c}]=O(2^{-3k})$ uniformly in~$x$.
		\end{enumerate}
	\end{lemma}
	\begin{proof}
		For~(a), on $\mathcal{A}_k^c$ we have $\sum_v L_N(v)^r>N^{r(\alpha+\delta)}$, so Markov's inequality and the three cases of Lemma~\ref{lem:local-time} give the stated bounds.
		
		For~(b), since $Y_k\leq\sum_{j=0}^{N-1}|\xi(X_j)-\E[\xi]|/\deg(X_j)$, convexity gives $Y_k^p\leq N^{p-1}\sum_{j=0}^{N-1}|\xi(X_j)-\E[\xi]|^p/\deg(X_j)^p$. Since $\deg\geq 1$, taking expectations yields $\E_x[Y_k^p]\leq N^p\,\E[|\xi-\E[\xi]|^p]$.
		
		For~(c), H\"older's inequality with exponents~$p/q$ and $p/(p{-}q)$ combines parts~(a) and~(b). The logarithmic factor at $d_s=2$ contributes at most $\mathrm{poly}(k)$, which is absorbed by the exponential decay from choosing~$r$ large.
	\end{proof}
	
	\begin{proposition}\label{prop:subcritical}
		Let~$G = (V,E)$ be an infinite, locally finite, connected graph with
		degree bounded by~$d$ and satisfying the spectral dimension bound~\eqref{eq:return-bound} for some $d_s>0$ and $A<\infty$.
		Let $(\xi(v))_{v\in V}$ be i.i.d.\ and independent of the walk with
		$\E[\xi] \in [-\infty, 0)$ and $\E[(\xi^+)^p]<\infty$ for some $p>1+2/d_s$.
		Then for every $q\in[1,(p-1)(d_s/2\wedge 1))$,
		\[
		\sup_{x\in V}\E_x\bigl[(\sup_n S_n)^q\bigr]<\infty\,.
		\]
		In particular, if $(p-1)(d_s/2\wedge 1)> 1$---which holds whenever $d_s<2$ (under the stated moment condition) or $d_s\geq 2$ and $p> 2$---then $\sup_\tau\E_x[S_\tau\mid\xi]<\infty$ almost surely.
	\end{proposition}
	
	\begin{proof}[Proof of Proposition~\textup{\ref{prop:subcritical}}]
		Replacing $\xi(v)$ by $\max(\xi(v),-M)$ increases $\zeta$ pointwise, hence increases $S_n$ pathwise and gives $\E[|\xi|^p]<\infty$. Since $\E[\max(\xi,-M)]\downarrow\E[\xi]<0$, for large~$M$ the truncated mean is still negative. We may therefore assume $\E[\xi]>-\infty$ and $\E[|\xi|^p]<\infty$.
		
		By Lemma~\ref{lem:dyadic}, it suffices to show $\sum_{k\geq 0}\E_x[Y_k^q]<\infty$ uniformly in~$x$. Lemma~\ref{lem:good-walk}(c) gives $\E_x[Y_k^q\one_{\mathcal{A}_k^c}]=O(2^{-3k})$, so it remains to bound $\E_x[Y_k^q\one_{\mathcal{A}_k}]$.
		
		\emph{Increment tail bound on~$\mathcal{A}_k$.}
		For a discrete interval $I\subset[0,N)$ of length~$m$, write $L_I(v)\coloneqq\#\{j\in I:X_j=v\}$ and $W(I)\coloneqq\sum_v L_I(v)(\xi(v)-\E[\xi])/\deg(v)$. Conditionally on the walk, the summands $L_I(v)(\xi(v)-\E[\xi])/\deg(v)$ are independent and mean-zero. Write $\alpha\coloneqq(1-d_s/2)^+$ and
		\[
		\Lambda_p(m)\coloneqq\begin{cases}
			m^{(p-1)\alpha} & \text{if } d_s<2,\\[3pt]
			(1{+}\log m)^{p-1} & \text{if } d_s=2,\\[3pt]
			1 & \text{if } d_s>2\,.
		\end{cases}\,
		\]
		By the Markov property at the start of~$I$ and Lemma~\ref{lem:local-time},
		\begin{equation}\label{eq:poly-moment}
			\E_x\biggl[\sum_v L_I(v)^p\biggr]\leq C_p\,m\,\Lambda_p(m)\,.
		\end{equation}
		On~$\mathcal{A}_k$, since $L_I(v)\leq L_N(v)\leq N^{\alpha+\delta}$ and $\sum_v L_I(v)=m$, the variance bound $\sum_v L_I(v)^2\leq N^{\alpha+\delta}\,m$ holds. Applying Lemma~\ref{lem:fuk-nagaev} conditionally and taking expectations,
		\begin{equation}\label{eq:w-tail}
			\P_x\bigl(|W(I)|>t,\,\mathcal{A}_k\bigr)
			\leq\frac{C\,m\,\Lambda_p(m)}{t^p}
			+2\exp\biggl(-\frac{c'\,t^2}{m\,N^{\alpha+\delta}}\biggr)\,.
		\end{equation}
		Here we use Lemma~\ref{lem:fuk-nagaev}(b) when $p\geq 2$ (which holds whenever $d_s\leq 2$). When $p<2$ (possible only for $d_s>2$), only part~(a) applies and the Gaussian term is absent; the polynomial term alone suffices for the chaining below.
		
		\emph{Dyadic chaining and tail integration.}
		Write $\theta\coloneqq d_s/2\wedge 1$, $\gamma\coloneqq p-(p-1)\theta$, and $c_0\coloneqq|\E[\xi]|/d$. Each $W_n$ with $n\in[2^k,2^{k+1})$ decomposes into at most $k{+}1$ increments $W(I)$ over dyadic sub-intervals of $[0,N)$. If $\max_{2^k\leq n<N}W_n>u$, then $|W(I)|>u/(k{+}1)$ for some such~$I$. At each scale $0\leq j\leq k$, there are at most $2^{k+1-j}$ dyadic intervals of length~$2^j$. A union bound with~\eqref{eq:w-tail} (setting $t=u/(k{+}1)$ and $m=2^j$) gives
		\begin{equation}\label{eq:Mk-tail-Ak}
			\P_x\bigl(\max_{2^k\leq n<N}W_n>u,\,\mathcal{A}_k\bigr)
			\leq\frac{C\,\mathrm{poly}(k)\,2^{k\gamma}}{u^p}+G_k(u)\,.
		\end{equation}
		The polynomial coefficient arises from the union bound over $O(2^k)$ dyadic intervals, each contributing $O(2^j\Lambda_p(2^j))$; the geometric sum is dominated by its largest term, giving the exponent~$\gamma$. The Gaussian remainder
		\[
		G_k(u)\coloneqq 2\sum_{j=0}^k 2^{k+1-j}\exp\biggl(-\frac{c'\,u^2}{(k{+}1)^2\,2^j\,N^{\alpha+\delta}}\biggr)
		\]
		is present only when $p\geq 2$; otherwise $G_k\equiv 0$. For $u\geq c_0\,2^k$ and $\delta<\theta$, the exponent in the $j=k$ term is at least $c\,2^{k(\theta-\delta)}/(k{+}1)^2\to\infty$, using $1-\alpha=\theta$, so $G_k(c_0\,2^k)$ decays superexponentially in~$k$.
		Tail integration gives
		\[
		\E_x[Y_k^q\one_{\mathcal{A}_k}]\leq q\int_{c_0 2^k}^\infty u^{q-1}\P_x\bigl(\max_{2^k\leq n<N}W_n>u,\,\mathcal{A}_k\bigr)\,du\,.
		\]
		The polynomial term in~\eqref{eq:Mk-tail-Ak} contributes $C\,\mathrm{poly}(k)\,2^{k\gamma}\int_{c_0 2^k}^\infty u^{q-1-p}\,du=C\,\mathrm{poly}(k)\,2^{k(q-(p-1)\theta)}$, and the Gaussian remainder contributes a term summable in~$k$. Combining with $\E_x[Y_k^q\one_{\mathcal{A}_k^c}]=O(2^{-3k})$,
		\[
		\E_x[Y_k^q]\leq C\,\mathrm{poly}(k)\,2^{k(q-(p-1)\theta)}+O(2^{-3k})\,.
		\]
		Since $q<(p-1)\theta$, the geometric decay dominates $\mathrm{poly}(k)$ and $\sum_k\E_x[Y_k^q]<\infty$. The bound is uniform in~$x$ since Lemma~\ref{lem:local-time} is uniform in the starting point.
	\end{proof}
	
	Theorem~\ref{thm:stab}(i) follows: setting $\xi(v)=\sigma(v)-1$ gives $\E[\xi]=\mu-1<0$ and $\E[(\xi^+)^p]\leq\E[(\sigma^+)^p]<\infty$, so $\sup_x\E_x[(\sup_n S_n)^q]<\infty$. By Corollary~\ref{cor:RW-infinite}, $u_\infty(x)\leq\E_x[\sup_n S_n\mid\xi]$, and Jensen's inequality gives $\E[u_\infty(x)^q]<\infty$.
	
	\begin{remark}
		By~\eqref{eq:heat-kernel-upper-bound}, on every graph with degree bounded by~$d$ the hypotheses of Proposition~\ref{prop:subcritical} hold with $d_s=1$, yielding the conditions $p>3$ and $q\in[1,(p-1)/2)$ of Theorem~\ref{thm:OS}(iii).
	\end{remark}
	
	\subsection{Stabilization on graphs with sub-Gaussian displacement}\label{sec:poly-growth}
	
	\begin{proposition}\label{prop:poly-growth}
		Let $G=(V,E)$ be an infinite, locally finite, connected graph with degree bounded by~$d$.  Assume:
		\begin{enumerate}[label=\textup{(H\arabic*)}]
			\item\label{H1} \textup{(Volume growth)}\; $|B(o,r)|\leq C_{\mathrm{vol}} r^{d_f}$ for some $o\in V$, $d_f>0$, and all $r\geq 1$.
			\item\label{H2} \textup{(Spectral dimension)}\; $\sup_{x\in V}\P_x(X_n=x)\leq A n^{-d_s/2}$ for all $n\geq 1$, for some $d_s\in(0,2)$ and $A<\infty$.
			\item\label{H3} \textup{(Walk dimension)}\; For all $x\in V$, all $n\geq 1$, and all $r\geq 1$,
			\[
			\P_x\bigl(\tau_{B(x,r)}\leq n\bigr)\leq C_{\mathrm{disp}}\exp\biggl(-c_{\mathrm{disp}}\Bigl(\frac{r^{d_w}}{n}\Bigr)^{\!1/(d_w-1)}\biggr)\,,
			\]
			for some $d_w\geq 2$ and constants $C_{\mathrm{disp}},c_{\mathrm{disp}}>0$.
		\end{enumerate}
		Let $(\xi(v))_{v\in V}$ be i.i.d.\ and independent of the walk with $\E[\xi]\in[-\infty,0)$ and $\E[(\xi^+)^p]<\infty$ for some $p>\max(2d_f/(d_w\,d_s),\,1)$.  Then $\sup_\tau\E_o[S_\tau\mid\xi]<\infty$ almost surely.
	\end{proposition}
	
	The moment condition depends on the three geometric parameters $d_f$, $d_s$, and $d_w$. Hypothesis~\ref{H3} is equivalent, under standard regularity, to the exit-time scaling $\E_x[\tau_{B(x,R)}]\asymp R^{d_w}$; it implies the uniform local trap condition of Definition~\ref{def:trap-family} on bounded-degree graphs. On bounded-degree graphs, hypothesis~\ref{H3} holds with $d_w=2$ by the Carne--Varopoulos maximal displacement bound \citep[Section~13.2]{LP16}, recovering the condition $p>d_f/d_s$ and hence Theorem~\ref{thm:stab}(ii).
	
	\begin{remark}\label{rem:poly-growth-examples}
		On graphs satisfying the Einstein relation $d_s=2d_f/d_w$ with $d_s<2$, the threshold reduces to $p>1$. This includes~$\Z$ (where $d_f=1$, $d_s=1$, $d_w=2$) and the Sierpi\'nski gasket (where $d_f=\frac{\log 3}{\log 2}$, $d_s=\frac{2\log 3}{\log 5}$, $d_w=\frac{\log 5}{\log 2}$). On stationary such graphs---including~$\Z$ and the Sierpi\'nski gasket---the phase transition (Theorem~\ref{thm:stationary-phase}) gives stabilization for $\mu < 1$ and $\E[|\sigma|]<\infty$, so the exponent~$1$ is sharp.
	\end{remark}
	
	\begin{proof}[Proof of Proposition \ref{prop:poly-growth}]
		As in Proposition~\ref{prop:subcritical}, we may assume $\E[\xi]\in\R$ and $\E[|\xi|^p]<\infty$, and additionally $\xi\geq -M_0$ a.s.\ for a constant $M_0\geq 1$ with $\P(\xi\leq M_0)>0$. Since $S_n\leq Z\coloneqq\sup_{n\geq 1}(W_n+\E[\xi]\,n/d)^+$ pathwise (as in Lemma~\ref{lem:dyadic}), for every bounded stopping time~$\tau$ we have $\E_o[S_\tau\mid\xi]\leq\E_o[Z\mid\xi]$. By the $0$--$1$ law (Proposition~\ref{prop:01-law}), it suffices to find an event~$\mathcal{G}$ with $\P(\mathcal{G})>0$ on which $\E_o[Z\mid\xi]<\infty$ for $\P(\cdot\mid\mathcal{G})$-a.e.~$\xi$.
		
		Choose $\beta\in(d_f/p,\,d_w d_s/2)$; this interval is nonempty since $p>2d_f/(d_w d_s)$ gives $d_f/p<d_w d_s/2$. Fix $\delta\in(0, d_s/2-\beta/d_w)$.
		
		\emph{Step~1: A positive-probability scenery event.}
		Define
		\[
		\mathcal{G} \coloneqq \bigl\{\xi(v)\leq M_0 \vee \mathrm{dist}(o,v)^\beta \text{ for all } v\in V\bigr\}\,.
		\]
		By Markov's inequality, for $\mathrm{dist}(o,v)\geq 1$,
		\[
		\P\bigl(\xi(v)>M_0\vee \mathrm{dist}(o,v)^\beta\bigr)
		\leq \E[(\xi^+)^p](M_0\vee \mathrm{dist}(o,v)^\beta)^{-p}
		\leq \E[(\xi^+)^p] \mathrm{dist}(o,v)^{-\beta p}\,.
		\]
		Since $\beta p>d_f$ and $|B(o,r)|\leq C_{\mathrm{vol}} r^{d_f}$,
		\begin{align*}
			\sum_{v\in V}\P\bigl(\xi(v)>M_0\vee \mathrm{dist}(o,v)^\beta\bigr)
			\leq 1 + C\sum_{r\geq 1}|B(o,r)|\bigl(r^{-\beta p}-(r{+}1)^{-\beta p}\bigr) 
			\leq 1 + C\sum_{r\geq 1}r^{d_f-\beta p-1}
			< \infty\,.
		\end{align*}
		Each factor $\P(\xi\leq M_0\vee \mathrm{dist}(o,v)^\beta)\geq\P(\xi\leq M_0)>0$ and the complementary probabilities are summable, so $\P(\mathcal{G})>0$.
		
		\emph{Step~2: Recentering under~$\mathcal{G}$.}
		Because~$\mathcal{G}$ is an intersection of events each involving a single coordinate, conditioning on~$\mathcal{G}$ preserves independence.  Since $M_0\vee \mathrm{dist}(o,v)^\beta\geq 0>\E[\xi]$, the event $\{\xi>M_0\vee \mathrm{dist}(o,v)^\beta\}$ carries mass entirely above~$\E[\xi]$, so
		\[
		\E[\xi(v)\mid\mathcal{G}] = \E[\xi\mid\xi\leq M_0\vee \mathrm{dist}(o,v)^\beta] \leq \E[\xi]\,.
		\]
		Write $\bar\xi(v)\coloneqq(\xi(v)-\E[\xi(v)\mid\mathcal{G}])/\deg(v)$.  Under $\P(\cdot\mid\mathcal{G})$, the $\bar\xi(v)$ are independent and centered, and
		\begin{equation}\label{eq:poly-Wbar}
			W_n \leq \sum_v L_n(v)\bar\xi(v) \eqqcolon \bar W_n\,.
		\end{equation}
		
		\emph{Step~3: Good-walk event with displacement control.}
		Fix $k\geq 0$ and set $N=2^{k+1}$. Write $R_N\coloneqq C_0 N^{1/d_w}(\log N)^{(d_w-1)/d_w}$ for a constant $C_0$ to be chosen. Define
		\[
		\mathcal{A}_k \coloneqq \bigl\{\max_v L_N(v)\leq N^{1-d_s/2+\delta}\bigr\}
		\cap \bigl\{\tau_{B(o,R_N)}> N\bigr\}\,.
		\]
		We show that for every $r\geq 1$ and $C_0$ large enough,
		\begin{equation}\label{eq:poly-good-walk}
			\P_o(\mathcal{A}_k^c) \leq C_r N^{-r}\,.
		\end{equation}
		The local-time part follows from Lemma~\ref{lem:good-walk}(a) with the integer parameter chosen large enough. For the displacement, hypothesis~\ref{H3} with $n=N$ and $r=R_N$ gives
		\[
		\P_o\bigl(\tau_{B(o,R_N)}\leq N\bigr)\leq C_{\mathrm{disp}}\exp\biggl(-c_{\mathrm{disp}}\Bigl(\frac{R_N^{d_w}}{N}\Bigr)^{\!1/(d_w-1)}\biggr)\,.
		\]
		Choosing $C_0$ large makes the right-hand side $O(N^{-r})$.  Together the two bounds give~\eqref{eq:poly-good-walk}.
		
		\emph{Step~4: Bernstein bound on $\mathcal{A}_k\cap\mathcal{G}$.}
		On $\mathcal{A}_k\cap\mathcal{G}$, every visited vertex~$v$ satisfies $\mathrm{dist}(o,v)\leq R_N$, so $|\bar\xi(v)| \leq CR_N^{\beta}$.
		Under $\P(\cdot\mid\mathcal{G})$, $\xi(v)$ has the law of~$\xi\mid\{\xi\leq M_0\vee \mathrm{dist}(o,v)^\beta\}$; since
		\[
		\P(\xi\leq M_0\vee \mathrm{dist}(o,v)^\beta)\geq\P(\xi\leq M_0)>0\,,
		\]
		we have $\sup_v\E[|\bar\xi(v)|^p\mid\mathcal{G}]\leq C$. For every vertex~$v$ with $\mathrm{dist}(o,v)\leq R_N$: when $p<2$, the bound $\bar\xi(v)^2\leq |\bar\xi(v)|^{2-p}|\bar\xi(v)|^p$ gives $\var(\bar\xi(v)\mid\mathcal{G})\leq CR_N^{\beta(2-p)}$; when $p\geq 2$, Lyapunov's inequality gives $\var(\bar\xi(v)\mid\mathcal{G})\leq C$. In both cases,
		\[
		\var(\bar\xi(v)\mid\mathcal{G})\leq CR_N^{\beta(2-p)^+}\,.
		\]
		For a discrete interval $I\subset[0,N)$ of length~$m$, conditionally on the walk and on~$\mathcal{G}$, the variables $L_I(v)\bar\xi(v)$ (over~$v$) are independent and centered with
		\[
		|L_I(v)\bar\xi(v)| \leq CN^{1-d_s/2+\delta}R_N^{\beta}\,,
		\]
		and total variance
		\[
		V(I) \coloneqq \sum_v L_I(v)^2\var(\bar\xi(v)\mid\mathcal{G}) \leq CR_N^{\beta(2-p)^+} N^{1-d_s/2+\delta} m\,.
		\]
		Lemma~\ref{lem:fuk-nagaev}(c) gives
		\[
		\P\Bigl(\Bigl|\sum_v L_I(v)\bar\xi(v)\Bigr|>t\;\Big|\;X,\mathcal{G}\Bigr)
		\leq 2\exp\biggl(-\frac{t^2/2}{V(I) + C N^{1-d_s/2+\delta}R_N^{\beta} t/3}\biggr)\,.
		\]
		Since $R_N\asymp N^{1/d_w}(\log N)^{(d_w-1)/d_w}$, the $N$-exponent of the linear term exceeds that of~$V(I)$ by $\tfrac{\beta}{d_w}\min(p{-}1,1)>0$, so the linear term dominates for large~$N$.  Each $[0,n)$ with $n\in[2^k,2^{k+1})$ decomposes into at most $k{+}1$ dyadic sub-intervals.  If $\max_{2^k\leq n<N}\bar W_n>u$, then by pigeonhole one of the at most $2N$ dyadic sub-intervals of~$[0,N)$ contributes more than $u/(k{+}1)$.  Applying the above to each and union-bounding, with $u=|\E[\xi]|\,2^k/d$,
		\begin{equation}\label{eq:poly-Bernstein}
			\P_o\Bigl(\max_{2^k\leq n<2^{k+1}}\bar W_n > |\E[\xi]|\,2^k/d,\;\mathcal{A}_k\;\Big|\;\mathcal{G}\Bigr)
			\leq CN\exp\biggl(-\frac{cN^{d_s/2-\delta-\beta/d_w}}{(\log N)^{1+\beta(d_w-1)/d_w}}\biggr)\,,
		\end{equation}
		since $d_s/2-\delta-\beta/d_w>0$.
		
		\emph{Step~5: $\E_o[Z\one_\mathcal{G}]<\infty$.}
		With $Y_k$ as in Lemma~\ref{lem:dyadic}, $Z\leq\sum_{k\geq 0}Y_k$, and on~$\mathcal{G}$ the bound~\eqref{eq:poly-Wbar} gives $Y_k\leq (\max_{2^k\leq n<2^{k+1}}\bar W_n-|\E[\xi]|\,2^k/d)^+$.
		
		\emph{Good walk.}  On $\mathcal{A}_k\cap\mathcal{G}$,
		\[
		Y_k\leq\sum_{j<N}|\bar\xi(X_j)|\leq CN\,R_N^{\beta}\,.
		\]
		Since $Y_k>0$ on~$\mathcal{G}$ forces $\max\bar W_n>|\E[\xi]|\,2^k/d$ by~\eqref{eq:poly-Wbar}, applying~\eqref{eq:poly-Bernstein} gives
		\[
		\E_o[Y_k\one_{\mathcal{A}_k}\one_\mathcal{G}]
		\leq CN^2 R_N^{\beta}\,\P(\mathcal{G})\exp\biggl(-\frac{cN^{d_s/2-\delta-\beta/d_w}}{(\log N)^{1+\beta(d_w-1)/d_w}}\biggr)\,,
		\]
		which is summable in~$k$.
		
		\emph{Bad walk.}  On~$\mathcal{G}$, every vertex visited by time~$N$ satisfies $\mathrm{dist}(o,v)\leq N$, so $Y_k\leq CN^{1+\beta}$.  By~\eqref{eq:poly-good-walk} with $r>2+\beta$,
		\[
		\E_o[Y_k\one_{\mathcal{A}_k^c}\one_\mathcal{G}]
		\leq CN^{1+\beta}\P(\mathcal{G})\P_o(\mathcal{A}_k^c)
		\leq C_r N^{1+\beta-r}\,,
		\]
		which is summable.  Combining, $\E_o[Z\one_\mathcal{G}] \leq \sum_{k\geq 0}\E_o[Y_k\one_\mathcal{G}] < \infty$. Since $\one_\mathcal{G}$ is scenery-measurable, the tower property gives $\E[\E_o[Z\mid\xi]\one_\mathcal{G}]<\infty$, so $\E_o[Z\mid\xi]<\infty$ for $\P(\cdot\mid\mathcal{G})$-a.e.~$\xi$.  For every such~$\xi$, $\sup_\tau\E_o[S_\tau\mid\xi]\leq\E_o[Z\mid\xi]<\infty$, hence $\P(\sup_\tau\E_o[S_\tau\mid\xi]<\infty)\geq\P(\mathcal{G})>0$.  By Proposition~\ref{prop:01-law}, $\sup_\tau\E_o[S_\tau\mid\xi]<\infty$ almost surely.
	\end{proof}
	
	Theorem~\ref{thm:stab}(ii) follows by setting $\xi(v)=\sigma(v)-1$ in Proposition \ref{prop:poly-growth} and applying Corollary~\ref{cor:RW-infinite}.

	\section{Sharpness of the conditions}\label{sec:counterexamples}
	
	We establish the sharpness of the conditions in Theorems~\ref{thm:explosion} and~\ref{thm:stab}. We first show that bounded degree cannot be removed from the explosion result, and that the moment condition for stabilization is optimal under heat kernel, local time, and volume hypotheses. We then show that the i.i.d.\ assumption is essential, and construct bounded-degree graphs---both transient and recurrent---on which for every $p<3$, subcritical sandpiles with finite $p$-th moment explode.
	
	\subsection{Necessity of bounded degree}
	
	The bounded degree assumption cannot be removed from the supercritical explosion result as the next result shows.
	\begin{lemma}\label{ex:counterexample}
		There exists a locally finite tree~$T$ such that for every collection of i.i.d.\ random variables~$(\sigma(v))_{v \in V}$ with~$\E[|\sigma|] < \infty$, the divisible sandpile on~$T$ stabilizes almost surely.
	\end{lemma}
	
	\begin{proof}
		Let~$T$ be the spherically symmetric rooted tree with root~$\rho$ in which every vertex at depth~$n$ has exactly~$b_n$ children, where~$b_n = (n+2)^3$. Then~$T$ is locally finite, connected, and has unbounded degree. It suffices to show $\sup_x\sum_{v\in V}g(x,v)<\infty$, since then Proposition~\ref{prop:short-clock} gives stabilization whenever $\E[(\sigma-1)^+]<\infty$, which holds for every i.i.d.\ law with $\E[|\sigma|]<\infty$.
		
		Define~$\phi(v) \coloneqq \sum_{m=n}^{\infty} 1/b_m$ for a vertex~$v$ at depth~$n$; write~$\phi_n$ for the common value at depth~$n$. Then~$\phi \geq 0$ and~$\phi_0 = \sum_{m=0}^{\infty} 1/b_m < \infty$. We claim that $\phi(v) - P\phi(v) \geq \frac{1}{2\deg(v)}$,
		for every~$v \in V$, where~$Pf(v) = \frac{1}{\deg(v)}\sum_{w \sim v} f(w)$. For the root, $P\phi(\rho)=\phi_1$ and $\phi(\rho)-P\phi(\rho)=1/b_0=1/\deg(\rho)$.
		For a vertex~$v$ at depth~$n \geq 1$,
		\begin{align*}
			\phi_n - P\phi(v) &= \frac{1}{b_n + 1}(\phi_n - \phi_{n-1}) + \frac{b_n}{b_n + 1}(\phi_n - \phi_{n+1}) \\
			&= \frac{1}{b_n + 1}\Bigl(1 - \frac{1}{b_{n-1}}\Bigr) \geq \frac{1}{2\deg(v)}\,.
		\end{align*}
		using~$\phi_{n-1} - \phi_n = 1/b_{n-1}$,~$\phi_n - \phi_{n+1} = 1/b_n$, and~$b_{n-1} \geq 2$. By optional stopping and~$\phi\geq 0$, monotone convergence gives $\sum_v g(x,v)=\E_x[\sum_{k\geq 0}1/\deg(X_k)]\leq 2\phi(x)\leq 2\phi_0<\infty$.
	\end{proof}

	\subsection{Moment sharpness}
	
	The next lemma shows that the moment condition in Theorem~\ref{thm:stab}(i) cannot be removed: on every graph satisfying the heat kernel, local time, and volume hypotheses~\ref{A1}--\ref{A3} below, a sufficiently heavy tail forces the expected odometer to diverge.
	\begin{lemma}\label{lem:moment-sharpness}
		Let $G=(V,E)$ be an infinite, locally finite, connected graph.  Assume:
		\begin{enumerate}[label=\textup{(A\arabic*)}]
			\item\label{A1} \textup{(Upper heat kernel)}\; $\sup_{x\in V}\P_x(X_n=x)/\deg(x)\leq A\,n^{-\alpha}$ for all $n\geq 1$, for some $\alpha>0$ and $A<\infty$.
			\item\label{A2} \textup{(Local time lower bound)}\; There exist $d_w\geq 2$ and $a>0$ such that
			\[
			\E_o[L_n(v)]/\deg(v)\geq a\,n^{1-\alpha}\,.
			\]
			for every $v\in B(o,n^{1/d_w})$ and all sufficiently large~$n$.
			\item\label{A3} \textup{(Volume lower bound)}\; $|B(o,r)|\geq c_{\mathrm{vol}}\,r^{d_f}$ for all $r\geq 1$, for some $d_f>0$ and $c_{\mathrm{vol}}>0$.
		\end{enumerate}
		Let $(\sigma(v))_{v\in V}$ be i.i.d.\ with $\E[\sigma(v)]\in(-\infty,\infty]$.
		If $\E\bigl[(\sigma(v)^+)^{(d_f+d_w)/(d_w\alpha)}\bigr]=\infty$, then $\E[u_\infty(o;\sigma)]=\infty$.
	\end{lemma}
	
	The kernel $\P_x(X_n=x)/\deg(x)$ in~\ref{A1} is the diagonal of the degree-normalized $n$-step transition kernel $\P_x(X_n=y)/\deg(y)$; the Green function~$g$ from~\eqref{eq:green-def} is its sum over~$n$. On~$\Z^d$, hypotheses~\ref{A1}--\ref{A3} hold with $\alpha=d/2$, with $d_w=2$, and with $d_f=d$, giving the exponent $(d_f+d_w)/(d_w\alpha)=1+2/d$. On the Sierpi\'nski gasket, $\alpha=\log 3/\log 5$, with $d_w=\log 5/\log 2$ and $d_f=\log 3/\log 2$, giving the exponent $1+\log 5/\log 3\approx 2.46$. In both cases the exponent matches $1+2/d_s$, which is the stabilization threshold of Proposition~\ref{prop:subcritical}.
	
	\begin{proof}[Proof of Lemma \ref{lem:moment-sharpness}]
		Since $u_\infty(o)\geq(\sigma(o)-1)^+/\deg(o)$, if $\E[\sigma^+]=\infty$, then $\E[u_\infty(o)]=\infty$. We may therefore assume $\mu\coloneqq\E[\sigma(v)] \in (-\infty, \infty)$, which gives $\E[|\sigma|]<\infty$.  Set $p^*\coloneqq(d_f+d_w)/(d_w\alpha)$.  We show $\sum_{t\geq 1}\P(u_\infty(o)>t)=\infty$. Fix $t\geq 1$ and set $N_t\coloneqq \lfloor \beta t\rfloor$ and $m_t\coloneqq K t^{\alpha}$, where $\beta>2$ and $K$ are large constants.  Write $R_t\coloneqq t^{1/d_w}$.  For each $v\in B(o,R_t)$, define
		\[
		Z_v\coloneqq\sum_{w\neq v}g_{N_t}(o,w)\bigl(\sigma(w)-\mu\bigr)\,.
		\]
		By Corollary~\ref{cor:RW-infinite} with $\tau=N_t$ and the decomposition~\eqref{eq:dn-wn-stuff},
		\begin{equation}\label{eq:spike-decomp}
			u_\infty(o)\geq g_{N_t}(o,v)\bigl(\sigma(v)-\mu\bigr)+Z_v+(\mu-1)A_{N_t}(o)\,.
		\end{equation}
		
		\emph{Step~1: Local time bounds.}  Cauchy--Schwarz in~$\ell^2(\deg)$ and hypothesis~\ref{A1} give $\P_x(X_n=y)/\deg(y)\leq A\,n^{-\alpha}$ for $n\geq 1$ (after enlarging~$A$), hence
		\begin{equation}\label{eq:local-time-ub}
			\sup_v\frac{\E_o[L_{N_t}(v)]}{\deg(v)}\leq\sum_{k=0}^{N_t-1}\frac{A}{(k\vee 1)^{\alpha}}\eqqcolon H_\alpha(N_t)\,.
		\end{equation}
		Hypothesis~\ref{A2} with $n=N_t$ gives $g_{N_t}(o,v)\geq c\,t^{1-\alpha}$ for every $v\in B(o,R_t)$ and a constant $c>0$.
		
		\emph{Step~2: Large spike.}  By the bounds in Step~1,
		\[
		g_{N_t}(o,v)\cdot m_t\geq cKt
		\qquad\text{and}\qquad
		g_{N_t}(o,v)\leq H_\alpha(N_t)=o(t)\,.
		\]
		When $\sigma(v)\geq m_t$, the spike contribution in~\eqref{eq:spike-decomp} satisfies
		\[
		g_{N_t}(o,v)\bigl(\sigma(v)-\mu\bigr)\geq cKt-|\mu|\,H_\alpha(N_t)\,.
		\]
		The drift satisfies $(\mu-1)A_{N_t}(o)\geq -(1-\mu)^+\beta t$ (positive drift only helps).  Since $H_\alpha(N_t)=o(t)$, choosing $K$ large enough that $cK>4+(1-\mu)^+\beta+|\mu|$ ensures that the right-hand side of~\eqref{eq:spike-decomp} exceeds $2t$ for all large~$t$, provided $Z_v\geq -t$.  Hence
		\begin{equation}\label{eq:spike-wins}
			\sigma(v)\geq m_t\quad\text{and}\quad Z_v\geq -t\,,
		\end{equation}
		implies $u_\infty(o)>t$.
		
		\emph{Step~3: Background control.}  We show $\P(Z_v\geq -t)\geq 3/4$ for all large~$t$, uniformly in $v\in B(o,R_t)$.  Split $Z_v=Z_v^{\leq}+Z_v^{>}$ according to whether $|\sigma(w)-\mu|\leq M$ or not, where $M$ is chosen so that $\E[|\sigma-\mu|\one_{\{|\sigma-\mu|>M\}}]\leq 1/(16\beta)$.  The tail part satisfies $\E[|Z_v^{>}|]\leq A_{N_t}(o)/(16\beta)\leq t/16$, so Markov's inequality gives $\P(Z_v^{>}<-t/2)\leq 1/8$.  By centering of $Z_v$, $|\E[Z_v^{\leq}]|=|\E[Z_v^{>}]|\leq\E[|Z_v^{>}|]\leq t/16$. The bounded part has $\var(Z_v^{\leq})\leq M^2 H_\alpha(N_t)\cdot\beta t=o(t^2)$, so Chebyshev's inequality gives $\P(|Z_v^{\leq}-\E[Z_v^{\leq}]|>t/4)=o(1)$.  A union bound finally yields $\P(Z_v\geq -t)\geq 3/4$.
		
		\emph{Step~4: Tail lower bound.}  For each $v\in B(o,R_t)$, the event $A_v\coloneqq\{\sigma(v)\geq m_t,\,Z_v\geq -t\}$ implies $u_\infty(o)>t$. Since the $v$-th mass was excluded from the definition of~$Z_v$, the variables $\sigma(v)$ and $Z_v$ are independent, so $\P(A_v)\geq\frac{3}{4}\P(\sigma^+\geq m_t)\eqqcolon\frac{3}{4}p_t$. Set $X_t\coloneqq\sum_{v\in B(o,R_t)}\one_{A_v}$. Then $\E[X_t]\geq\frac{3}{4}|B(o,R_t)|\,p_t$. For the second moment: for $v\neq w$, $A_v\cap A_w\subseteq\{\sigma(v)\geq m_t\}\cap\{\sigma(w)\geq m_t\}$, so by independence of $\sigma(v)$ and $\sigma(w)$, $\P(A_v\cap A_w)\leq\P(\sigma(v)\geq m_t)\,\P(\sigma(w)\geq m_t)=p_t^2$. Therefore $\E[X_t^2]\leq|B(o,R_t)|\,p_t+|B(o,R_t)|^2 p_t^2$, and the Paley--Zygmund inequality $\P(X_t\geq 1)\geq\E[X_t]^2/\E[X_t^2]$ gives
		\begin{equation}\label{eq:tail-lb}
			\P(u_\infty(o)>t)\geq\P(X_t\geq 1)\geq c\min\bigl(|B(o,R_t)|\,p_t,\,1\bigr)\,.
		\end{equation}
		
		\emph{Step~5: Divergence.}  By~\ref{A3}, $|B(o,R_t)|\geq c_{\mathrm{vol}}\,t^{d_f/d_w}$, so it suffices to show
		\[
		\sum_{t\geq 1}t^{d_f/d_w}\P(\sigma^+\geq Kt^{\alpha})=\infty\,.
		\] For $s\in[t,t+1]$ with $t\geq 1$, we have $s^{d_f/d_w}\leq 2^{d_f/d_w}t^{d_f/d_w}$ and $\P(\sigma^+\geq Ks^{\alpha})\leq\P(\sigma^+\geq Kt^{\alpha})$, so
		\[
		\int_t^{t+1}s^{d_f/d_w}\P(\sigma^+\geq Ks^{\alpha})\; ds\leq 2^{d_f/d_w}t^{d_f/d_w}\P(\sigma^+\geq Kt^{\alpha})\,.
		\]
		Summing over $t\geq 1$ shows $\sum_{t\geq 1}t^{d_f/d_w}\P(\sigma^+\geq Kt^{\alpha})\geq c\int_1^{\infty} s^{d_f/d_w}\P(\sigma^+\geq Ks^{\alpha})\; ds$.  The substitution $s=(u/K)^{1/\alpha}$ transforms the integral into $c\int u^{(d_f+d_w)/(d_w\alpha)-1}\P(\sigma^+\geq u)\; du$.  By the tail integral formula, this diverges if and only if $\E[(\sigma^+)^{(d_f+d_w)/(d_w\alpha)}]=\infty$.  Since $(d_f+d_w)/(d_w\alpha)=p^*$, the hypothesis $\E[(\sigma^+)^{p^*}]=\infty$ yields the divergence.
	\end{proof}
	
	\subsection{Role of the i.i.d.\ assumption}
	
	\begin{lemma}[Finite perturbations change stabilization on recurrent graphs]\label{ex:finite-perturbation}
		Let $G=(V,E)$ be an infinite, locally finite, and connected graph, and fix $o\in V$. For $m>0$ consider the configuration $\sigma(v)=1+m\,\one_{\{v=o\}}$.
		Then for every $x\in V$ and $n\geq 0$, we have $u_n(x)=m g_n(x,o)$, and hence $u_\infty(x)=m\,g(x,o)\in[0,\infty]$. In particular, $\sigma$ stabilizes if and only if the random walk on $G$ is transient. If the walk is recurrent, then $\sigma$ explodes.
	\end{lemma}
	
	\begin{proof}
		Here $\zeta(v)=(\sigma(v)-1)/\deg(v)$ is nonnegative and supported on $\{o\}$, with $\zeta(o)=m/\deg(o)$.
		Hence along every walk path the payoff $S_n=\sum_{k=0}^{n-1}\zeta(X_k)$ is nondecreasing in $n$.
		Therefore the optimal stopping value is attained by $\tau=n$, so by Theorem~\ref{thm:RW},
		\[
		u_n(x)=\E_x[S_n]=\frac{m}{\deg(o)}\,\E_x[L_n(o)]=m g_n(x,o)\,.
		\]
		Letting $n\to\infty$ gives $u_\infty(x)=m\,g(x,o)$.
		The walk is transient if and only if $g(x,o)<\infty$, and recurrent if and only if $g(x,o)=\infty$.
	\end{proof}
	
	\begin{example}[$0$--$1$ law fails for independent non-i.i.d.\ masses on recurrent graphs]\label{ex:01-fails}
		Assume $G$ is recurrent and fix $o\in V$.
		Let $\sigma(v)\equiv 1$ for $v\neq o$, and let $\sigma(o)$ be independent of the rest with
		$\P(\sigma(o)=1)=1-p$ and $\P(\sigma(o)=1+m)=p$ for some $m>0$ and $p\in(0,1)$.
		Then $\P(\sigma\text{ stabilizes})=1-p\in(0,1)$: by Lemma~\ref{ex:finite-perturbation}, the sandpile stabilizes if and only if $\sigma(o)=1$, which occurs with probability~$1-p$.
	\end{example}
	
	The i.i.d.\ assumption in Theorem~\ref{thm:explosion}(ii) is also essential.
	
	\begin{lemma}[Coboundary stabilization]\label{ex:coboundary}
		Let~$f\colon V\to[0,\infty)$ and set~$\sigma\coloneqq 1-\Delta f$. Then~$u_\infty(x)\leq f(x)$ for every~$x\in V$; in particular, $\sigma$ stabilizes.
	\end{lemma}
	
	\begin{proof}
		Since~$\zeta(v)=-\Delta f(v)/\deg(v)$, optional stopping gives~$\E_x[S_\tau]=f(x)-\E_x[f(X_\tau)]\leq f(x)$ for every bounded stopping time~$\tau$, since~$f\geq 0$.  Corollary~\ref{cor:RW-infinite} then gives~$u_\infty(x)\leq f(x)$.
	\end{proof}
	
	\begin{example}[Critical stabilization without i.i.d.]\label{ex:coboundary-Zd}
		On~$\Z^d$, let~$(\eta(x))_{x\in\Z^d}$ be i.i.d.\ $\mathrm{Uniform}[0,1]$ and set~$\sigma\coloneqq 1-\Delta\eta$.
		Then~$\sigma$ is stationary, ergodic, bounded, and $2$-dependent, with~$\E[\sigma(0)]=1$ and~$\var(\sigma(0))>0$.
		By Lemma~\ref{ex:coboundary} with~$f=\eta$, the sandpile stabilizes with~$u_\infty(x)\leq\eta(x)\leq 1$.
	\end{example}
	
	\begin{figure}[ht]
		\centering
		\includegraphics[width=0.48\textwidth]{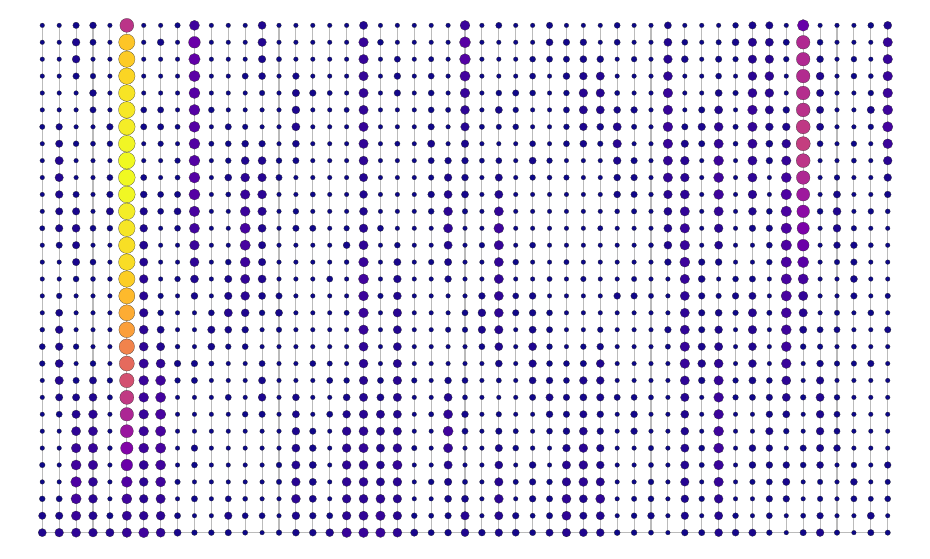}%
		\hfill
		\includegraphics[width=0.48\textwidth]{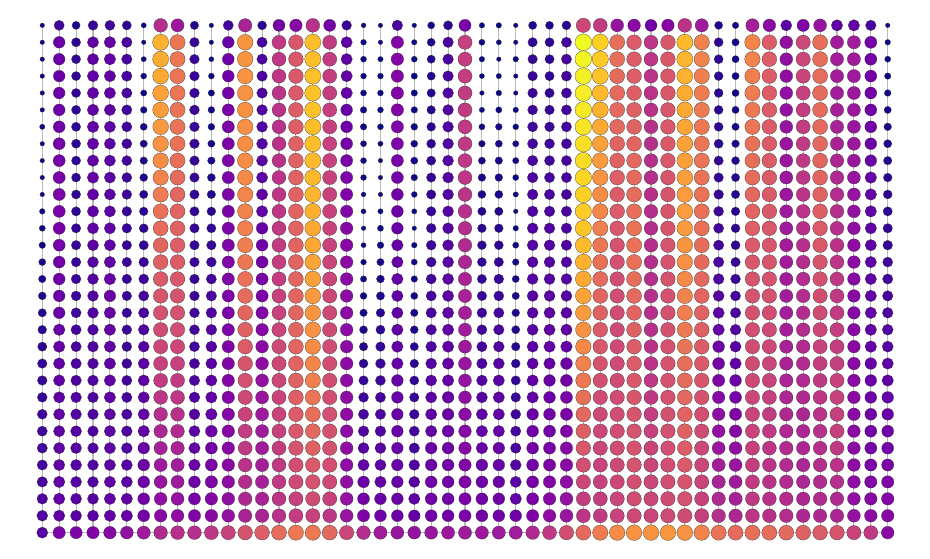}
		\caption{Odometer~$u_\infty(v)$ on the comb lattice (Example~\ref{ex:comb-nonconservation}) with i.i.d.\ $\mathrm{Exp}(\mu)$ initial masses; vertex color and size are proportional to~$u_\infty(v)$. Left: subcritical ($\mu=0.6$); right: critical ($\mu=1.0$). The odometer concentrates along the spine (degree~$4$), which absorbs mass from the teeth (degree~$2$), illustrating the failure of conservation of mass. The comb is not stationary, but Theorems~\ref{thm:explosion} and~\ref{thm:stab} still give explosion for~$\mu \geq 1$ and stabilization for~$\mu < 1$ with sufficient moments.}
		\label{fig:comb}
	\end{figure}
	
	\begin{remark}[Non-conservation of mass on the comb lattice]\label{ex:comb-nonconservation}
		On the comb lattice~$\mathcal{C}$ (vertex set~$\Z^2$, edges~$(x,0)\sim(x\pm 1,0)$ and~$(x,y)\sim(x,y\pm 1)$), spine vertices have degree~$4$ and tooth vertices degree~$2$. For i.i.d.\ masses with~$\E[\sigma]=\mu$ and $a\coloneqq\E[(\sigma-1)^+]>0$, one round of parallel toppling gives $\E[\sigma_1(x,0)]=\mu+a/2>\mu$ and $\E[\sigma_1(x,1)]=\mu-a/4<\mu$, so neither the unweighted nor the degree-weighted mean mass is conserved pointwise. Mass flows from the teeth toward the spine.
	\end{remark}
	
	\subsection{Electrical estimates on tree-of-pipes combs}\label{sec:comb-estimates}
	
	The counterexamples in the next two subsections share a common electrical network ingredient: voltage estimates on a comb subnetwork of a $B$-ary tree whose edges are replaced by pipes.  We collect these estimates here.
	
	The killed Green function $g_C$ from~\eqref{eq:gC-PDE} can be interpreted as the voltage in an electrical network with unit conductances: $g_C(o,v)$ is the potential at~$v$ when unit current flows from~$o$ to the boundary $V\setminus C$ \citep[Proposition~2.1]{LP16}. The identity~\eqref{eq:voltage-identity} then reads as the superposition of Ohm's law over all vertices.
	
	Let $B\geq 2$ be an integer and let $\alpha\in(1/2,1)$.  Assume $B$ is large enough that
	\begin{equation}\label{eq:comb-B-cond}
		B^{\alpha}\geq 4\,,
		\qquad
		2B^{\alpha}\leq B-1\,,
		\qquad
		2B^{1-2\alpha}<1\,,
		\qquad
		\lambda\coloneqq\frac{B^{2\alpha-1}}{4}>1\,.
	\end{equation}
	Such $B$ exists for every $\alpha\in(1/2,1)$: as $B\to\infty$, $B^{1-2\alpha}\to 0$ and $B^{2\alpha-1}\to\infty$ (since $\alpha>1/2$), while $B^{\alpha}/B\to 0$ (since $\alpha<1$).  Define the pipe lengths $L_j\coloneqq\lfloor B^{\alpha j}\rfloor$ for $j\geq 1$. Since $B^{\alpha}\geq 4$, for every $j\geq 1$ one has
	\begin{equation}\label{eq:comb-Lj-bounds}
		\tfrac{1}{2}\,B^{\alpha j}
		\leq L_j
		\leq B^{\alpha j}\,.
	\end{equation}
	Start from the rooted $B$-ary tree of depth~$n$ and replace each edge between generations $j{-}1$ and~$j$ by a path (\textbf{pipe}) of length~$L_j$.  For a terminal word $w\in\{1,\ldots,B\}^n$, the \textbf{comb} $D_{w,n}$ is the union of the trunk from the root to the terminal pipe~$P_w$ and all $B{-}1$ sibling pipes at each trunk branching vertex, with far endpoints excluded and serving as the absorbing set.  The root may additionally have at most one extra boundary edge of unit resistance (this occurs when the comb is embedded in a larger network).
	Write $b_0,b_1,\ldots,b_{n-1}$ for the trunk branching vertices.  When unit current enters at $b_0$ and exits through the absorbing set, write $g\coloneqq g_{D_{w,n}}(b_0,\cdot)$ for the resulting voltage on the comb.  Let $I_k$ denote the current through the $k$-th trunk pipe, let $V_j\coloneqq g(b_j)$, and let $R_j$ denote the effective resistance from $b_j$ to the boundary through the chosen continuation branch, so that $V_j=I_{j+1}\,R_j$.
	
	\begin{proposition}[Comb estimates]\label{prop:comb-estimates}
		Under the setup above, the following hold.
		\begin{enumerate}[label=\textup{(\alph*)}]
			\item \textup{(Continuation resistance)}\; For every $0\leq j\leq n{-}1$, one has $R_j\leq 2\,L_{j+1}$, and hence $V_j\leq 2\,I_{j+1}\,L_{j+1}$.
			\item \textup{(Trunk currents)}\; The trunk currents satisfy $I_1\geq 1/(4B)$ and $I_{j+1}\geq I_j/(2B)$ for every $1\leq j\leq n{-}1$. Consequently, $I_n\geq(4B)^{-n}$.
			\item \textup{(Total voltage mass)}\; There exists $C_{\mathrm{comb}}>0$ depending only on $B$ and~$\alpha$ such that
			\[
			\sum_{u\in D_{w,n}}g(u)\leq C_{\mathrm{comb}}\,I_n\,L_n^2\,.
			\]
			\item \textup{(Spike contribution)}\; Let $a\colon D_{w,n}\to\R$ satisfy $a(u)\geq -b$ for all~$u$, for some finite $b\geq 0$, and set $K\coloneqq 2+2b\,(C_{\mathrm{comb}}+1)$. If a vertex~$v$ in the first half of the terminal pipe~$P_w$ (at distance at least $L_n/2$ from the boundary) satisfies $a(v)\geq K\,L_n-b$, then
			\[
			\sum_{u\in D_{w,n}}g(u)\,a(u)\geq I_n\,L_n^2\,.
			\]
			\item \textup{(Exponential growth)}\; $I_n\,L_n^2\geq\lambda^n/4$.
		\end{enumerate}
	\end{proposition}
	
	\begin{proof}
		(a)\; At $j=n{-}1$, the continuation is the terminal pipe, so $R_{n-1}=L_n$.  For $j\leq n{-}2$, the continuation from $b_j$ consists of the level-$(j{+}1)$ trunk pipe (resistance $L_{j+1}$) followed by the parallel combination of the $B{-}1$ side stubs at $b_{j+1}$ (resistance $L_{j+2}$ each) and the further continuation.  By the series and parallel laws for resistors \citep[Section~2.3]{LP16}, $R_j\leq L_{j+1}+L_{j+2}/(B{-}1)$.  By~\eqref{eq:comb-Lj-bounds}, $L_{j+2}\leq 2B^{\alpha}L_{j+1}$, and the condition $2B^{\alpha}\leq B{-}1$ from~\eqref{eq:comb-B-cond} gives $L_{j+2}/(B{-}1)\leq L_{j+1}$, hence $R_j\leq 2\,L_{j+1}$.  The voltage bound follows from $V_j=I_{j+1}\,R_j$; this is Ohm's law applied to the continuation branch.
		
		\medskip
		(b)\; At the root $b_0$, the outgoing currents are at most $(B{-}1)V_0/L_1$ through the $B{-}1$ side branches, $I_1$ through the trunk, and $V_0$ through the extra boundary edge (if present).  Since unit current enters and $V_0\leq 2I_1 L_1$ by part~(a), and since $2L_1\leq 2B^{\alpha}\leq B{-}1$ by~\eqref{eq:comb-B-cond},
		\[
		1\leq 2I_1 L_1 + 2(B{-}1)I_1 + I_1 \leq (B{-}1)I_1 + 2(B{-}1)I_1 + I_1 \leq 4B\,I_1\,,
		\]
		giving $I_1\geq 1/(4B)$.  For $1\leq j<n$, the current $I_j$ at $b_j$ splits between the continuation current $I_{j+1}$ and the $B{-}1$ side currents $V_j/L_{j+1}$ each.  By part~(a), $I_j\leq I_{j+1}+2(B{-}1)I_{j+1}\leq 2B\,I_{j+1}$, hence $I_{j+1}\geq I_j/(2B)$.  Iterating gives $I_n\geq(4B)^{-n}$.
		
		\medskip
		(c)\; At level~$j$, there are at most $B$ pipes of $L_j$ vertices each with voltage at most $V_{j-1}$.  The level-$j$ pipe contribution is at most $B\,L_j\,V_{j-1}\leq 2B\,I_j\,L_j^2$ by part~(a).  By part~(b), $I_j\leq(2B)^{n-j}\,I_n$, and by~\eqref{eq:comb-Lj-bounds}, $L_j^2\leq 4\,B^{-2\alpha(n-j)}\,L_n^2$.  Therefore
		\[
		2B\,I_j\,L_j^2\leq 8B\,(2B^{1-2\alpha})^{n-j}\,I_n\,L_n^2\,.
		\]
		Since $2B^{1-2\alpha}<1$ by~\eqref{eq:comb-B-cond}, the geometric series converges and the total pipe contribution is at most a constant times $I_n\,L_n^2$.  The branching-vertex contribution $\sum_{j=0}^{n-1}V_j\leq 2\sum_{j=1}^{n}I_j\,L_j$ is bounded by the same geometric series (since $L_j\geq 1$ gives $I_j\,L_j\leq I_j\,L_j^2$), and the result follows.
		
		\medskip
		(d)\; The vertex~$v$ lies at distance at least $L_n/2$ from the boundary, so $g(v)\geq I_n\,L_n/2$.  The spike gives $g(v)(KL_n - b)\geq (K/2 - b/(2L_n))\,I_n\,L_n^2$, and the negative background contributes at most $b\,C_{\mathrm{comb}}\,I_n\,L_n^2$ by part~(c).  Since $L_n\geq 1$, the coefficient of $I_n\,L_n^2$ is at least $K/2 - b\,(C_{\mathrm{comb}}+\tfrac{1}{2})\geq 1$.
		
		\medskip
		(e)\; By part~(b) and~\eqref{eq:comb-Lj-bounds}, $I_n\,L_n^2\geq(4B)^{-n}\cdot\tfrac{1}{4}\,B^{2\alpha n}=\tfrac{1}{4}\,\lambda^n$.
	\end{proof}
	
	\subsection{Non-stabilization on transient bounded-degree graphs}\label{sec:transient-nonstab}
	
	The next result shows that the moment condition $p > 3$ in Theorem~\ref{thm:stab}(i) cannot be improved to $p > 2$, even on transient graphs.
	
	\begin{theorem}\label{thm:transient-nonstab}
		For every $p\in(0,3)$, there exists a transient bounded-degree graph~$G$ such that for every $\mu<1$, there is an i.i.d.\ law bounded from below with $\E[\sigma]=\mu$ and $\E[|\sigma-\mu|^p]<\infty$ for which $u_\infty(o)=\infty$ almost surely. In particular, the divisible sandpile does not stabilize.
	\end{theorem}
	
	Let $p\in(0,3)$. Choose $q$ with $\max(p,1)<q<3$ and then choose $\alpha$ with
	\begin{equation}
		\frac{1}{2}<\alpha<\min\Bigl(1,\frac{1}{q-1}\Bigr)\,;
	\end{equation}
	this interval is nonempty since $q<3$ gives $1/(q-1)>1/2$. Let $B\geq 2$ and pipe lengths $L_j$ be as in Section~\ref{sec:comb-estimates}, with conditions~\eqref{eq:comb-B-cond} and bounds~\eqref{eq:comb-Lj-bounds}.
	
	Let $Y\geq 1$ be a Pareto random variable with $\P(Y\geq t)=t^{-q}$ for $t\geq 1$. Then $\E[Y]=q/(q{-}1)<\infty$ and $\E[Y^p]<\infty$ since $q>p$. Fix $\mu_0 < 1$ and set $\sigma(v)\coloneqq Y_v-\E[Y]+\mu_0$, where the $Y_v$ are i.i.d.\ copies of~$Y$. Then $\E[\sigma]=\mu_0<1$ and $\E[|\sigma-\mu_0|^p]<\infty$; moreover $\sigma\geq 1-\E[Y]+\mu_0>-\infty$. Write
	\begin{equation}\label{eq:tr-ab}
		a(v)\coloneqq\sigma(v)-1=Y_v-b\,,
		\qquad
		b\coloneqq\E[Y]-\mu_0+1\,.
	\end{equation}
	Every vertex satisfies $a(v)\geq -b$.
	
	Start from the rooted $B$-ary tree~$T_B$ (every vertex has exactly $B$ children). Replace each edge between generations $n{-}1$ and $n$ by a path (\textbf{pipe}) of length~$L_n$. Call the resulting graph~$G$, rooted at~$o$. The maximum degree is $B{+}1$ (each non-root branching vertex has one parent pipe and $B$ child pipes; the root has $B$ child pipes; each pipe interior vertex has degree~$2$).
	
	\begin{figure}[ht]
		\centering
		\begin{tikzpicture}[scale=0.85,
			pipe/.style={thick},
			bv/.style={circle,fill,inner sep=1.5pt},
			pv/.style={circle,fill=gray,inner sep=1pt},
			lf/.style={circle,draw,inner sep=1pt}]
			
			\node[bv,label=above:{$o$}] (root) at (0,0) {};
			
			\node[bv] (A) at (-3,-1.5) {};
			\node[bv] (B) at (0,-1.5) {};
			\node[bv] (C) at (3,-1.5) {};
			\draw[pipe] (root) -- node[pv,pos=0.5]{} (A);
			\draw[pipe] (root) -- node[pv,pos=0.5]{} (B);
			\draw[pipe] (root) -- node[pv,pos=0.5]{} (C);
			\node[font=\scriptsize] at (1.8,-0.5) {$L_1$};
			
			\node[bv] (Ba) at (-1,-3.5) {};
			\node[bv] (Bb) at (0,-3.5) {};
			\node[bv] (Bc) at (1,-3.5) {};
			\draw[pipe] (B) -- node[pv,pos=0.3]{}
			node[pv,pos=0.65]{} (Ba);
			\draw[pipe] (B) -- node[pv,pos=0.3]{}
			node[pv,pos=0.65]{} (Bb);
			\draw[pipe] (B) -- node[pv,pos=0.3]{}
			node[pv,pos=0.65]{} (Bc);
			\node[font=\scriptsize] at (1.2,-2.3) {$L_2$};
			
			\node[lf] (Bb1) at (-0.5,-6.0) {};
			\node[lf] (Bb2) at (0,-6.0) {};
			\node[lf] (Bb3) at (0.5,-6.0) {};
			\draw[pipe] (Bb) -- node[pv,pos=0.2]{}
			node[pv,pos=0.4]{} node[pv,pos=0.6]{}
			node[pv,pos=0.8]{} (Bb1);
			\draw[pipe] (Bb) -- node[pv,pos=0.2]{}
			node[pv,pos=0.4]{} node[pv,pos=0.6]{}
			node[pv,pos=0.8]{} (Bb2);
			\draw[pipe] (Bb) -- node[pv,pos=0.2]{}
			node[pv,pos=0.4]{} node[pv,pos=0.6]{}
			node[pv,pos=0.8]{} (Bb3);
			\node[font=\scriptsize] at (0.9,-4.5) {$L_3$};
			
			\node at (-3,-2.8) {$\vdots$};
			\node at (3,-2.8) {$\vdots$};
			\node at (-1,-4.8) {$\vdots$};
			\node at (1,-4.8) {$\vdots$};
			
		\end{tikzpicture}
		\caption{The tree-of-pipes graph~$G$ with $B=3$ (schematic; the proof requires $B$ large enough to satisfy~\eqref{eq:comb-B-cond}).
			Each edge of the rooted $B$-ary tree is replaced
			by a pipe of length~$L_n$ at level~$n$; the pipes
			grow geometrically longer at deeper levels.
			Filled circles are branching vertices; gray
			circles are pipe interior vertices (degree~$2$);
			open circles are leaves.}
		\label{fig:tr-tree}
	\end{figure}
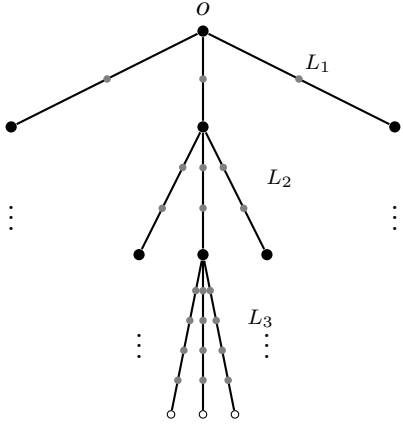
	
	Set $K\coloneqq 2+2b\,(C_{\mathrm{comb}}+1)$ as in Proposition~\ref{prop:comb-estimates}(d). Call a level-$n$ pipe \textbf{good} if some vertex in its first half (at distance $\geq L_n/2$ from the absorbing endpoint) has $Y_v\geq KL_n$.
	
	\begin{lemma}[Good pipes at every level]\label{lem:tr-good}
		Almost surely, for all sufficiently large~$n$, at least one level-$n$ pipe is good.
	\end{lemma}
	
	\begin{proof}
		A single pipe has at least $L_n/2$ vertices in its first half, so
		\[
		\P(\text{a fixed pipe is good})
		\geq 1-(1-(KL_n)^{-q})^{L_n/2}
		\geq c\,L_n^{1-q}\,,
		\]
		for all large~$n$. The $B^n$ level-$n$ pipes have disjoint vertex sets, hence independent indicators. The expected number of good pipes is at least $c\,B^n L_n^{1-q}\geq c\,B^{n(1-\alpha(q-1))}$. Since $\alpha<1/(q{-}1)$, the exponent $1-\alpha(q{-}1)$ is positive and the expected count tends to infinity. Writing $p_n\coloneqq\P(\text{a fixed pipe is good})$,
		\[
		\P(\text{no good pipe at level }n)
		=(1-p_n)^{B^n}
		\leq e^{-B^n p_n}
		\leq\exp\bigl(-c\,B^{n(1-\alpha(q-1))}\bigr)\,,
		\]
		which is summable since $1-\alpha(q{-}1)>0$. By Borel--Cantelli, good pipes exist at all large levels almost surely.
	\end{proof}
	
	\begin{proof}[Proof of Theorem~\ref{thm:transient-nonstab}]
		The graph~$G$ is transient: the unit flow sending current~$B^{-n}$ through each generation-$n$ pipe has energy $\sum_{n\geq 1}L_n/B^n\leq C\sum B^{(\alpha-1)n}<\infty$ since $\alpha<1$, so Thomson's principle \citep[Section~2.4]{LP16} gives finite effective resistance from~$o$ to infinity. By Lemma~\ref{lem:tr-good}, good pipes exist at all large levels almost surely. For each good level-$n$ pipe indexed by a word~$w$, the comb $D_{w,n}$ from Section~\ref{sec:comb-estimates} satisfies $\sum_{u\in D_{w,n}}g(u)\,a(u)\geq I_n L_n^2$ by Proposition~\ref{prop:comb-estimates}(d), and $I_n L_n^2\geq\lambda^n/4\to\infty$ by part~(e). It follows from Theorem~\ref{thm:nested-vol} that
		\[
		u_\infty(o)
		=\max\biggl\{0,\;\sup_C\sum_{v\in C} g_C(v)\bigl(\sigma(v)-1\bigr)\biggr\}
		=\infty\qquad\text{a.s.}\qedhere
		\]
	\end{proof}
	
	\begin{remark}
		On the stretched binary tree ($B=2$, pipes $L_n=\lfloor b^n\rfloor$ with $1<b<2$), the sandpile stabilizes for $p>1$ by Proposition~\ref{prop:poly-growth} (with the Einstein relation $d_s=2d_f/d_w$). On every bounded-degree graph, stabilization holds for $p>3$. The threshold thus depends on the graph. Whether stabilization holds at $p=3$ on every bounded-degree graph remains open.
	\end{remark}

	\subsection{Non-stabilization on recurrent bounded-degree graphs}\label{sec:recurrent-nonstab}
	
	The following result complements Theorem~\ref{thm:transient-nonstab} by showing that the same sharpness holds on recurrent graphs.
	
	\begin{theorem}\label{thm:recurrent-nonstab}
		For every $p\in(0,3)$ and every $d_f\in(\max(p,2),3)$, there exists a recurrent bounded-degree graph~$G$ with
		\begin{equation}
			|B(o,R)|\leq C\,R^{d_f}
			\qquad\text{for all } R\geq 1\,,
		\end{equation}
		such that for every $\mu<1$, there is an i.i.d.\ law bounded from below with $\E[\sigma]=\mu$ and $\E[|\sigma-\mu|^p]<\infty$ for which $u_\infty(o)=\infty$ almost surely. Hence the divisible sandpile does not stabilize.
		Moreover, $|B(o,r_k)|\geq c\,r_k^{d_f}$ along a subsequence $r_k\to\infty$, so $\limsup_{r\to\infty}\frac{\log|B(o,r)|}{\log r}=d_f$.
	\end{theorem}
	
	We use the electrical network interpretation of~$g_C$ from Section~\ref{sec:comb-estimates}.  Two properties of the voltage will be used repeatedly.
	
	\emph{Kirchhoff's node law} (current across a separating cut).  If $U\subset C$ contains~$o$ and no vertex of $\partial C$, then the total current crossing from~$U$ to $C\setminus U$ is~$1$.
	
	\emph{Maximum principle for dead ends} (dead-end subgraphs carry constant voltage).  If $H\subset C$ is finite, connected, attached to $C\setminus H$ at exactly one vertex~$x$, and $H$ contains neither the source~$o$ nor a boundary vertex of~$C$, then $g_C$ is constant on~$H$, equal to $g_C(x)$: the function $g_C$ is harmonic on $H\setminus\{x\}$ with all boundary values equal to $g_C(x)$, so the maximum principle \citep[Section~2.1]{LP16} forces $g_C\equiv g_C(x)$.
	
	Let $p\in(0,3)$ and $d_f\in(\max(p,2),3)$, and set
	\begin{equation}
		\alpha\coloneqq\frac{1}{d_f-1}\,.
	\end{equation}
	Since $d_f\in(2,3)$, one has $\alpha\in(1/2,1)$.  Let $B\geq 2$ and pipe lengths $L_j$ be as in Section~\ref{sec:comb-estimates}, with conditions~\eqref{eq:comb-B-cond} and bounds~\eqref{eq:comb-Lj-bounds}.
	Define~$Y$, $\sigma(v)$, $a(v)$, and~$b$ as in the proof of Theorem~\ref{thm:transient-nonstab}, with~$d_f$ in place of~$q$: the Pareto tail is~$\P(Y\geq t)=t^{-d_f}$, and
	\begin{equation}\label{eq:rec-ab}
		a(v)\coloneqq\sigma(v)-1=Y_v-b\,,
		\qquad
		b\coloneqq\E[Y]-\mu_0+1\,.
	\end{equation}
	Every vertex satisfies $a(v)\geq -b$.
	
	Let $m\geq 1$.  Start from the rooted $B$-ary tree of depth~$m$ and replace each edge between generations $j{-}1$ and~$j$ by a path (a \textbf{pipe}) of length~$L_j$.  Call the resulting rooted graph $H(m)$, with root~$x$.  The distance from the root to every generation-$m$ leaf is
	\begin{equation}
		R_m\coloneqq\sum_{j=1}^m L_j\,.
	\end{equation}
	The number of non-root vertices is
	\begin{equation}
		N_m\coloneqq|H(m)\setminus\{x\}|
		=\sum_{j=1}^m B^j L_j\,,
	\end{equation}
	since each generation-$j$ pipe contributes exactly $L_j$ new vertices.  We set $N_0\coloneqq 0$ and $R_0\coloneqq 0$.
	
	\begin{lemma}[Gadget geometry]\label{lem:rec-geometry}
		Let $C_R\coloneqq 2/(1-B^{-\alpha})$.  There exist constants $c_N, C_N, C_{\mathrm{ball}}>0$ depending only on $B$ and~$\alpha$ such that the following hold.
		\begin{enumerate}
			\item[\textup{(a)}] For every $m\geq 1$, $L_m\leq R_m\leq C_R\,L_m$.
			\item[\textup{(b)}] For every $m\geq 0$, $c_N\,R_m^{d_f}\leq N_m\leq C_N\,R_m^{d_f}$.
			\item[\textup{(c)}] For every $m\geq 1$ and every $1\leq t\leq R_m$, $|B_{H(m)}(x,t)\setminus\{x\}|\leq C_{\mathrm{ball}}\,t^{d_f}$.
		\end{enumerate}
	\end{lemma}
	
	\begin{proof}
		All estimates follow from $L_j\asymp B^{\alpha j}$, giving geometric series dominated by the last term.
		
		\emph{Part~\textup{(a)}.}  The lower bound $R_m\geq L_m$ is immediate.  For the upper bound, $R_m\leq\sum_{j=1}^m B^{\alpha j} =B^{\alpha m}\sum_{i=0}^{m-1}B^{-\alpha i} \leq B^{\alpha m}/(1-B^{-\alpha})$.  Since $B^{\alpha m}\leq 2L_m$ by~\eqref{eq:comb-Lj-bounds}, $R_m\leq C_R L_m$.
		
		\emph{Part~\textup{(b)}.}  For $m=0$ the bound is trivial since $N_0=R_0=0$; assume $m\geq 1$.
		The successive terms $B^j L_j$ grow geometrically: by~\eqref{eq:comb-Lj-bounds}, $B^j L_j/(B^{j-1}L_{j-1}) \geq B\cdot B^{\alpha j}/(2B^{\alpha(j-1)}) =B^{1+\alpha}/2$.  Since $B^{1+\alpha}/2>2$ for $B$ large, the last term dominates:
		\begin{equation}\label{eq:rec-Nm-asymp}
			B^m L_m\leq N_m\leq 2\,B^m L_m\,.
		\end{equation}
		By~\eqref{eq:comb-Lj-bounds}, $B^m\asymp L_m^{1/\alpha}$, so $B^m L_m\asymp L_m^{1+1/\alpha}$.  Since $1+1/\alpha=d_f$ and $L_m\asymp R_m$ by part~(a), one has $N_m\asymp R_m^{d_f}$.
		
		\emph{Part~\textup{(c)}.}  Choose~$j$ so that $R_{j-1}<t\leq R_j$.  The ball of radius~$t$ contains all vertices in generations $1,\ldots,j{-}1$ (contributing at most $N_{j-1}$) and at most~$t$ vertices on each of the $B^j$ generation-$j$ pipes:
		\begin{equation}\label{eq:rec-ball-bound}
			|B_{H(m)}(x,t)\setminus\{x\}|
			\leq N_{j-1}+B^j\,t\,.
		\end{equation}
		By part~(b), $N_{j-1}\leq C_N R_{j-1}^{d_f} <C_N t^{d_f}$.  If $j=1$, then $B^j t=Bt\leq B t^{d_f}$ (since $t\geq 1$ and $d_f>1$).  If $j\geq 2$, then $t>R_{j-1}\geq L_{j-1} \geq B^{\alpha(j-1)}/2$, so $B^j\leq B(2t)^{1/\alpha}$.  Therefore $B^j t\leq 2^{1/\alpha}B\,t^{d_f}$.  Combining gives $|B_{H(m)}(x,t)\setminus\{x\}| \leq C_{\mathrm{ball}}\,t^{d_f}$.
	\end{proof}
	
	Let $m\geq 1$ and let $\omega\in\{1,\ldots,B\}^m$ be a terminal word.  Write $b_0=x,b_1,\ldots,b_{m-1}$ for the trunk branching vertices on the root-to-leaf path indexed by~$\omega$.  The \textbf{comb} $D_{m,\omega}$ consists of:
	\begin{itemize}
		\item the trunk pipes from $b_{j-1}$ to $b_j$ for $1\leq j\leq m{-}1$,
		\item at each $b_{j-1}$, the other $B{-}1$ child pipes with their far endpoints \textbf{deleted},
		\item at level~$m$, the chosen terminal pipe and all $B{-}1$ sibling pipes, each with its leaf \textbf{deleted}.
	\end{itemize}
	When the comb is used inside the global graph, the root~$x$ will also have one extra boundary edge of resistance~$1$ going forward along the backbone.  We include this edge in the local analysis below.
	
	Assume one unit of current arrives at~$x$ from the left.  Let $g$ be the resulting voltage, with all exits grounded at~$0$.  For $1\leq j\leq m$, let $I_j$ denote the current through the chosen level-$j$ trunk pipe.  For $0\leq j\leq m{-}1$, write $V_j\coloneqq g(b_j)$.  The effective resistance from $b_j$ to the boundary through the chosen continuation branch is denoted~$R_j$, so that
	\begin{equation}\label{eq:rec-VjIj}
		V_j=I_{j+1}\,R_j\,.
	\end{equation}
	By Proposition~\ref{prop:comb-estimates}(a) and~\eqref{eq:rec-VjIj},
	\begin{equation}
		V_j\leq 2\,I_{j+1}\,L_{j+1}
		\qquad(0\leq j\leq m{-}1)\,.
	\end{equation}
	By Proposition~\ref{prop:comb-estimates}(b), $I_m\geq(4B)^{-m}$.  Let $C_{\mathrm{comb}}$ be the constant from Proposition~\ref{prop:comb-estimates}(c), and set
	\begin{equation}\label{eq:rec-K}
		K\coloneqq2+2b\,(C_{\mathrm{comb}}+1)\,.
	\end{equation}
	By~\eqref{eq:rec-ab}, $a(v)=Y_v-b$, so $Y_v\geq K\,L_m$ implies $a(v)\geq K\,L_m-b$. Proposition~\ref{prop:comb-estimates}(d) then gives $\sum_{u\in D_{m,\omega}}g(u)\,a(u)\geq I_m\,L_m^2$.
	
	\begin{corollary}[Exponential growth of the local contribution]\label{cor:rec-loc}
		Define $\delta\coloneqq\frac{\log\lambda}{\alpha\log B}>0$ and $\lambda=\frac{B^{2\alpha-1}}{4}$. There exists $c_{\mathrm{loc}}>0$ such that whenever a vertex~$v$ in the first half of the terminal pipe satisfies $Y_v\geq KL_m$, then
		\[
		\sum_{u\in D_{m,\omega}}
		g(u)\,a(u)
		\geq c_{\mathrm{loc}}\,R_m^{\delta}\,.
		\]
	\end{corollary}
	
	\begin{proof}
		By Proposition~\ref{prop:comb-estimates}(e), $I_m L_m^2\geq\lambda^m/4$.  By Lemma~\ref{lem:rec-geometry}(a), $R_m\leq C_R B^{\alpha m}$, hence $R_m^{\delta}\leq C_R^{\delta}\,B^{\alpha\delta m} =C_R^{\delta}\,\lambda^m$ (since $B^{\alpha\delta}=\lambda$).  Therefore $I_m L_m^2\geq C_R^{-\delta}\,R_m^{\delta}/4$.
	\end{proof}
	
	We now construct the global graph.  Choose
	\begin{equation}\label{eq:rec-rho}
		0<\rho<\frac{\delta}{d_f+1}\,.
	\end{equation}
	Start from the ray $r_0{=}o,\,r_1,\,r_2,\,\ldots\,$ and attach copies of $H(m_k)$ at selected positions, as follows.
	The depths $m_k$ are chosen inductively so large that, with
	\begin{equation}\label{eq:rec-sk}
		s_k\coloneqq\bigl\lceil R_{m_k}^{\,\rho}\bigr\rceil\,,
	\end{equation}
	all of the following hold for $k\geq 2$:
	\begin{align}
		s_k&>s_{k-1}+R_{m_{k-1}}\,,
		\label{eq:rec-cond-sep}\\
		\sum_{j<k}N_{m_j}&\leq s_k^{\,d_f}\,,
		\label{eq:rec-cond-vol}\\
		c_{\mathrm{loc}}\,R_{m_k}^{\,\delta}
		&\geq \bigl(s_k{+}1\bigr)^{d_f+1} k\,.\label{eq:rec-cond-div}
	\end{align}
	These conditions are compatible and do not depend on the choice of~$\mu_0$.  Indeed, when choosing $m_k$, the quantities $s_{k-1}$, $R_{m_{k-1}}$, and $N_{m_j}$ for $j<k$ are already fixed.  Since $R_{m_k}\to\infty$ as $m_k\to\infty$: condition~\eqref{eq:rec-cond-sep} holds because $s_k\asymp R_{m_k}^{\rho}\to\infty$; condition~\eqref{eq:rec-cond-vol} holds because $s_k^{d_f}\asymp R_{m_k}^{\rho d_f} \to\infty$; condition~\eqref{eq:rec-cond-div} holds because $\delta>\rho(d_f{+}1)$ (by~\eqref{eq:rec-rho}) ensures that $R_{m_k}^{\delta}$ dominates $s_k^{d_f+1} k$. For each $k\geq 1$, attach the root of $H(m_k)$ to the ray vertex $r_{s_k}$.  Call the resulting graph~$G$ (see Figure~\ref{fig:rec-graph}).
	
	\begin{figure}[ht]
		\centering
		\begin{tikzpicture}[scale=0.75,
			ray/.style={very thick},
			pipe/.style={thick},
			bv/.style={circle,fill,inner sep=1.5pt},
			pv/.style={circle,fill=gray,inner sep=1pt},
			lf/.style={circle,draw,inner sep=1pt}]
			
			\node[bv,label=above:{$o$}] (r0) at (0,0) {};
			\foreach \i in {1,...,14} {
				\node[pv] (r\i) at (\i*0.7,0) {};
			}
			\node (rdots) at (10.5,0) {$\cdots$};
			\draw[ray] (r0) -- (r1) -- (r2) -- (r3) -- (r4)
			-- (r5) -- (r6) -- (r7) -- (r8) -- (r9);
			\draw[ray] (r9) -- (10.2,0);
			\draw[ray] (10.8,0) -- (r11) -- (r12) -- (r13)
			-- (r14);
			
			\node[bv,label=above:{$r_{s_1}$}] at (r2) {};
			\foreach \j/\ang in {1/-130, 2/-90, 3/-50} {
				\node[lf] (g1\j) at ($(r2)+(\ang:1.0)$) {};
				\draw[pipe] (r2) -- node[pv,pos=0.5]{} (g1\j);
			}
			\node[font=\scriptsize] at ($(r2)+(0,-1.5)$)
			{$H(m_1)$};
			
			\node[bv,label=above:{$r_{s_2}$}] at (r6) {};
			\foreach \j/\ang in {1/-140, 2/-90, 3/-40} {
				\node[bv] (g2b\j) at ($(r6)+(\ang:1.2)$) {};
				\draw[pipe] (r6) --
				node[pv,pos=0.4]{} node[pv,pos=0.7]{} (g2b\j);
			}
			\foreach \k/\kang in {1/-130, 2/-90, 3/-50} {
				\node[lf] (g2l\k) at
				($(g2b2)+(\kang:0.9)$) {};
				\draw[pipe] (g2b2) --
				node[pv,pos=0.33]{} node[pv,pos=0.67]{} (g2l\k);
			}
			\node[font=\tiny] at ($(g2b1)+(0,-0.6)$)
			{$\vdots$};
			\node[font=\tiny] at ($(g2b3)+(0,-0.6)$)
			{$\vdots$};
			\node[font=\scriptsize] at ($(r6)+(0,-3.0)$)
			{$H(m_2)$};
			
			\node[bv,label=above:{$r_{s_3}$}] at (r13) {};
			\draw[pipe] (r13) -- +(-1.5,-1.0)
			node[bv] (g3a) {};
			\draw[pipe] (r13) -- +(0,-1.2)
			node[bv] (g3b) {};
			\draw[pipe] (r13) -- +(1.5,-1.0)
			node[bv] (g3c) {};
			\foreach \nd in {g3a,g3b,g3c} {
				\draw[pipe] (\nd) -- +(-0.5,-0.8)
				node[bv] (x) {};
				\draw[pipe] (\nd) -- +(0,-0.9)
				node[bv] (y) {};
				\draw[pipe] (\nd) -- +(0.5,-0.8)
				node[bv] (z) {};
			}
			\node[font=\tiny] at ($(g3b)+(0,-1.5)$)
			{$\vdots$};
			\node[font=\scriptsize] at ($(r13)+(0,-3.5)$)
			{$H(m_3)$};
			
			\draw[decorate,
			decoration={brace,amplitude=4pt,raise=2pt}]
			(r2) -- (r6)
			node[midway,above=7pt,font=\scriptsize]
			{$s_2-s_1$};
			
		\end{tikzpicture}
		\caption{The graph~$G$: a ray (thick horizontal
			line) with tree-of-pipes gadgets $H(m_k)$ attached
			at widely spaced positions $s_1<s_2<s_3<\cdots$.
			Each gadget is a $B$-ary tree whose edges are
			replaced by pipes of increasing lengths.
			Later gadgets are larger; the radial intervals
			$[s_k,\,s_k+R_{m_k}]$ are disjoint.}
		\label{fig:rec-graph}
	\end{figure}
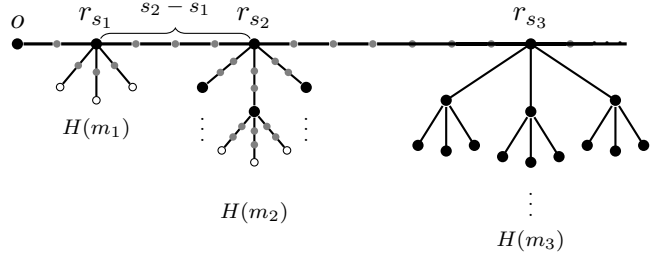
	
	The maximum degree of $G$ is $B{+}2$.  The graph~$G$ is recurrent: it is one-ended (the only edge leaving $\{r_0,\ldots,r_n\}$ and its gadgets is~$r_n r_{n+1}$), so every unit flow from~$o$ to infinity traverses every backbone edge and has infinite energy, and Thomson's principle \citep[Section~2.4]{LP16} gives recurrence.
	
	\begin{proposition}[Volume growth]
		\label{prop:rec-growth}
		There exists $C_G>0$ such that $|B(o,r)|\leq C_G\,r^{d_f}$ for every $r\geq 1$.  Moreover, along $r_k\coloneqq s_k+R_{m_k}$, one has $|B(o,r_k)|\geq c\,r_k^{d_f}$ for some $c>0$.
	\end{proposition}
	
	\begin{proof}
		By~\eqref{eq:rec-cond-sep}, the gadgets are radially disjoint: $s_k+R_{m_k}<s_{k+1}$ for every $k\geq 1$.  Therefore a ball around~$o$ can partially intersect at most one gadget.
		
		\emph{Upper bound.}  Fix $r\geq 1$ and let $k$ be the largest index with $s_k\leq r$ (set $k=0$ if $r<s_1$).  If $k=0$, then $|B(o,r)|\leq r{+}1\leq 2r^{d_f}$.  Assume $k\geq 1$.  If $s_k\leq r<s_k+R_{m_k}$, then $B(o,r)$ contains the ray segment (at most $r{+}1$ vertices), all earlier gadgets $H(m_j)$ for $j<k$, and a ball of radius $r{-}s_k$ inside $H(m_k)$.  By~\eqref{eq:rec-cond-vol}, $\sum_{j<k}N_{m_j}\leq s_k^{d_f}\leq r^{d_f}$.  By Lemma~\ref{lem:rec-geometry}(c), the partial gadget contributes at most $C_{\mathrm{ball}}\,r^{d_f}$.  In total, $|B(o,r)|\leq C\,r^{d_f}$.  If $r\geq s_k+R_{m_k}$, then no later gadget is reached (by radial disjointness), so $|B(o,r)|\leq r{+}1+\sum_{j\leq k}N_{m_j}$.  By~\eqref{eq:rec-cond-vol} and Lemma~\ref{lem:rec-geometry}(b), since $R_{m_k}\leq r$
		\[
		\sum_{j\leq k}N_{m_j} \leq s_k^{d_f}+N_{m_k} \leq r^{d_f}+C_N R_{m_k}^{d_f} \leq C\,r^{d_f}\,.
		\]
		
		\emph{Lower bound.}  At $r_k=s_k+R_{m_k}$, the entire gadget $H(m_k)$ is contained in $B(o,r_k)$, so $|B(o,r_k)|\geq N_{m_k}\geq c_N R_{m_k}^{d_f}$.  Since $s_k=\lceil R_{m_k}^{\rho}\rceil =o(R_{m_k})$ (because $\rho<1$), one has $r_k\asymp R_{m_k}$, which gives $|B(o,r_k)|\geq c\,r_k^{d_f}$.
	\end{proof}
	
	Recall the constant $K$ from~\eqref{eq:rec-K}.  For each $k\geq 1$, let $\mathcal{A}_k$ be the set of vertices in the first half of every terminal pipe of $H(m_k)$ (at distance $\geq L_{m_k}/2$ from the absorbing endpoint).  Since $L_{m_k}\geq 4$ (by $B^{\alpha}\geq 4$) and there are $B^{m_k}$ terminal pipes,
	\begin{equation}\label{eq:rec-Ak}
		|\mathcal{A}_k|\geq\tfrac{1}{4}\,
		B^{m_k}\,L_{m_k}\,.
	\end{equation}
	Define the event
	\begin{equation}
		E_k\coloneqq\bigl\{\text{there exists } v\in\mathcal{A}_k
		\text{ with } Y_v\geq K L_{m_k}\bigr\}\,.
	\end{equation}
	
	\begin{lemma}[Good events]\label{lem:rec-good}
		There exists $\eta>0$ such that $\P(E_k)\geq\eta$ for every $k\geq 1$.  The events $E_1,E_2,\ldots$ are independent.  In particular, infinitely many $E_k$ occur almost surely.
	\end{lemma}
	
	\begin{proof}
		Set $p_k\coloneqq(K\,L_{m_k})^{-d_f}$.  By~\eqref{eq:rec-Ak},
		\[
		|\mathcal{A}_k|\,p_k
		\geq\tfrac{1}{4}\,K^{-d_f}\,
		B^{m_k}\,L_{m_k}^{1-d_f}\,.
		\]
		Since $\alpha(d_f{-}1)=1$ and $L_{m_k}\leq B^{\alpha m_k}$, one has $L_{m_k}^{1-d_f}\geq B^{-\alpha m_k(d_f-1)} =B^{-m_k}$, so $B^{m_k}L_{m_k}^{1-d_f}\geq 1$.  Therefore $|\mathcal{A}_k|\,p_k\geq 1/(4K^{d_f})$, and
		\[
		\P(E_k)
		\geq 1-(1-p_k)^{|\mathcal{A}_k|}
		\geq 1-e^{-1/(4K^{d_f})}
		=:\eta>0\,.
		\]
		Independence holds because different gadgets have disjoint vertex sets, and the second Borel--Cantelli lemma applies.
	\end{proof}
	
	\begin{proof}[Proof of Theorem~\ref{thm:recurrent-nonstab}]
		Fix $k$ such that $E_k$ occurs, and choose a terminal word~$\omega$ such that some vertex in the first half of the terminal pipe satisfies $Y_v\geq K\,L_{m_k}$.  Define the finite connected set
		\[
		C_k\coloneqq\{r_0,\ldots,r_{s_k}\}
		\;\cup\;\bigcup_{j<k}H(m_j)
		\;\cup\;D_{m_k,\omega}\,.
		\]
		Write $g_k\coloneqq g_{C_k}$.  By Kirchhoff's node law, unit current traverses every backbone edge, giving $g_k(r_i)\leq s_k{+}1$ for every $0\leq i\leq s_k$.  By the maximum principle for dead ends, $g_k$ is constant on each earlier gadget $H(m_j)$ with $j<k$, equal to $g_k(r_{s_j})\leq s_k{+}1$.  Since $a\geq -b$, the backbone and earlier gadgets contribute at least $-2b\,(s_k{+}1)^{d_f+1}$, using condition~\eqref{eq:rec-cond-vol} and $d_f>2$.  By Corollary~\ref{cor:rec-loc}, the target comb contributes at least $c_{\mathrm{loc}}\,R_{m_k}^{\,\delta}$.  By condition~\eqref{eq:rec-cond-div}, the positive term dominates for large $k$:
		\[
		\sum_{u\in C_k}g_k(u)\,a(u)\geq (k-2b)\,(s_k{+}1)^{d_f+1}\,.
		\]
		By Lemma~\ref{lem:rec-good}, infinitely many $E_k$ occur almost surely, so
		\[
		u_\infty(o)
		=\max\biggl\{0,\;\sup_C\sum_{v\in C} g_C(v)\bigl(\sigma(v)-1\bigr)\biggr\}
		=\infty\qquad\text{a.s.}\qedhere
		\]
	\end{proof}
	
	\section{Open questions}
	
	\subsubsection*{Critical explosion without finite variance or symmetry}
	
	Does explosion hold at $\mu=1$ on bounded-degree graphs when $\E[\sigma^2]=\infty$? 
	The argument in~\citet{div-sand-crit} implies that this holds on graphs that are recurrent and satisfy conservation of mass. When $\sigma-1$ is symmetric, \citet*{MR3834853} proved explosion on vertex-transitive graphs (both recurrent and transient) under the additional assumption that $\sigma$ lies in the domain of attraction of an $\alpha$-stable law for $\alpha\in[1,2)$. Our Proposition~\ref{prop:convexity-reduction} reduces the symmetric case to the finite-variance case via a convexity argument, yielding critical explosion under symmetry on all bounded-degree graphs with no moment hypothesis beyond~$\E[\sigma]=1$. The general case without symmetry remains open.
	
	\subsubsection*{Endpoint moment conditions}
	
	Theorem~\ref{thm:stab}(i) gives stabilization for $p>3$, and our counterexamples give explosion for every $p<3$. Does stabilization hold at the endpoint $p=3$ on every bounded-degree graph? Similarly, the refined threshold $p>d_f$ in Theorem~\ref{thm:stab}(ii) and the threshold $p>1+2/d$ on~$\Z^d$ from Proposition~\ref{prop:subcritical} leave their respective endpoints open. On~$\Z^d$, Lemma~\ref{lem:moment-sharpness} shows that $\E[(\sigma^+)^{1+2/d}]=\infty$ forces $\E[u_\infty(0)]=\infty$, but this does not rule out $u_\infty(0)<\infty$ almost surely, which occurs by conservation of mass for~$\mu < 1$. 
	
	\subsubsection*{Acknowledgments}
	Thanks to Yuwen Wang for discussions at an early stage of this project.
	The research of E.\ Sava-Huss was funded by the Austrian Science Fund (FWF) 10.55776/\allowbreak{}PAT3123425 and research of Y. Peres was supported by National Science Foundation of China grant RFIS-W2531011.

	\bibliographystyle{plainnat}
	\bibliography{references}

\end{document}